\documentclass[reqno]{amsart}

\usepackage[
left=3cm,
right=3cm,
top=3cm, 
bottom=3cm,
a4paper]{geometry}

\let\tempone\itemize
\let\temptwo\enditemize
\renewenvironment{itemize}{\tempone \vspace{5pt}\addtolength{\itemsep}{0.5\baselineskip}}{\vspace{5pt} \temptwo}

\let\tempenum\enumerate
\let\tempenumtwo\endenumerate
\renewenvironment{enumerate}{\tempenum \vspace{5pt} \addtolength{\itemsep}{0.5\baselineskip}}{ \vspace{5pt} \tempenumtwo}

\makeatletter
\let\origsection\section
\renewcommand\section{\@ifstar{\starsection}{\nostarsection}}
\newcommand\nostarsection[1]
{\sectionprelude\origsection{#1}\sectionpostlude}
\newcommand\starsection[1]
{\sectionprelude\origsection*{#1}\sectionpostlude}
\newcommand\sectionprelude{%
	\vspace{0.75em} 
}
\newcommand\sectionpostlude{%
	\vspace{1em}   
}
\makeatother

\newcommand\Item[1][]{%
  \ifx\relax#1\relax  \item \else \item[#1] \fi
  \abovedisplayskip=0pt\abovedisplayshortskip=0pt~\vspace*{-\baselineskip}}

\makeatletter
\let\origsubsection\subsection
\renewcommand\subsection{\@ifstar{\starsubsection}{\nostarsubsection}}
\newcommand\nostarsubsection[1]
{\sectionprelude\origsubsection{#1}\subsectionpostlude}
\newcommand\starsubsection[1]
{\subsectionprelude\origsubsection*{#1}\subsectionpostlude}
\newcommand\subsectionprelude{%
	\vspace{0.25em} 
}
\newcommand\subsectionpostlude{%
	\vspace{0.0em}   
}
\makeatother

\usepackage[bb=boondox]{mathalfa}

\usepackage[T1]{fontenc}
\usepackage{lmodern}
\usepackage[utf8]{inputenc}
\usepackage{mathrsfs}
\usepackage{bbm}

\usepackage{mathtools}

\usepackage{esvect}
\usepackage[UKenglish]{babel}

\usepackage{amssymb}
\usepackage{amsthm}
\usepackage{amsmath}
\usepackage{mathrsfs}
\usepackage{tikz-cd}

\usepackage{tikz}

\usetikzlibrary{decorations.markings}
\usetikzlibrary{shapes.geometric, patterns, angles, quotes}

\usetikzlibrary{hobby}

\usepackage{amsmath,amssymb,amsthm,graphicx,amscd,amsfonts,latexsym}
\usepackage{pstricks}
\usepackage[normalem]{ulem}

\usepackage{lscape}
\usepackage[hidelinks]{hyperref}
\usepackage{tikz}
\usetikzlibrary{arrows}
\usetikzlibrary{decorations}
\usetikzlibrary{shapes}
\usetikzlibrary{decorations.pathreplacing,decorations.markings} 

\usetikzlibrary{calc,intersections,through,backgrounds}
\usepackage{graphicx}
\usepackage{caption} 
\usepackage{float}
\usepackage{upgreek}

\usetikzlibrary{shapes.geometric,positioning}
\usetikzlibrary{arrows,decorations.pathmorphing,decorations.pathreplacing}

\usepackage{ifthen} 

\makeatletter
\newcases{rightalignedcases}{\quad}{%
\hfil$\m@th\displaystyle{##}$\hfil}{$\m@th\displaystyle{##}$\hfil}{\lbrace}{.}
\makeatother

\newcounter{pos} 
\tikzset{									
	initcounter/.code={\setcounter{pos}{0}},
	style between/.style n args={3}{
		postaction={
			initcounter,
			decorate,
			decoration={
				show path construction,
				curveto code={
					\addtocounter{pos}{1}
					\pgfmathtruncatemacro{\min}{#1 - 1}
					\ifthenelse{\thepos < #2 \AND \thepos > \min}{
						\draw[#3]
						(\tikzinputsegmentfirst)
						..
						controls (\tikzinputsegmentsupporta) and (\tikzinputsegmentsupportb)
						..
						(\tikzinputsegmentlast);
					}{}
				}
			}
		},
	},
}

\tikzset{
    clip even odd rule/.code={\pgfseteorule}, 
    invclip/.style={
        clip,insert path=
            [clip even odd rule]{
                [reset cm](-\maxdimen,-\maxdimen)rectangle(\maxdimen,\maxdimen)
            }
    }
}

\usepackage{epstopdf}

\makeatletter
\newcommand{\colim@}[2]{%
	\vtop{\m@th\ialign{##\cr
			\hfil$#1\operator@font colim$\hfil\cr
			\noalign{\nointerlineskip\kern1.5\ex@}#2\cr
			\noalign{\nointerlineskip\kern-\ex@}\cr}}%
}
\newcommand{\colim}{%
	\mathop{\mathpalette\colim@{\rightarrowfill@\textstyle}}\nmlimits@
}
\makeatother

\DeclareMathOperator{\Hom}{Hom}
\DeclareMathOperator*{\proj}{proj}
\DeclareMathOperator{\End}{End}

\DeclareMathOperator{\deco}{\mathscr{D}}
\DeclareMathOperator{\marked}{\mathcal{M}}
\DeclareMathOperator{\DA}{\mathsf{D}^b(A)}

\DeclareMathOperator{\Deck}{\mathsf{Deck}}
\DeclareMathOperator{\H^0}{H^0}
\DeclareMathOperator{\HH}{H}
\DeclareMathOperator{\triv}{T}

\DeclareMathOperator{\reg}{reg}
\DeclareMathOperator{\Tw}{Tw}
\DeclareMathOperator{\Fuk}{Fuk}
\DeclareMathOperator{\MCG}{MCG}
\DeclareMathOperator{\Arf}{Arf}
\newcommand{\oInt}{\,\vv{\cap}\,}
\newcommand{\Ob}[1]{\operatorname{Ob}(#1)}
\newcommand{\Kb}[1]{\operatorname{K}^b(\proj #1)}

\newcommand{\seq}{\mathbbm{m}}

\newcommand{\Db}[1]{\mathsf{D}^b(#1)}

\newcommand{\BV}[1]{\operatorname{V}(#1)} 
\newcommand{\BE}[1]{\operatorname{E}(#1)} 
\newcommand{\BH}[1]{\operatorname{H}(#1)} 
\newcommand{\BF}[1]{\operatorname{F}(#1)} 

\newcommand\restr[2]{{
		\left.\kern-\nulldelimiterspace 
		#1 
		\vphantom{\big|} 
		\right|_{#2} 
}}

\tikzset{
	set arrow inside/.code={\pgfqkeys{/tikz/arrow inside}{#1}},
	set arrow inside={end/.initial=>, opt/.initial=},
	/pgf/decoration/Mark/.style={
		mark/.expanded=at position #1 with
		{
			\noexpand\arrow[\pgfkeysvalueof{/tikz/arrow inside/opt}]{\pgfkeysvalueof{/tikz/arrow inside/end}}
		}
	},
	arrow inside/.style 2 args={
		set arrow inside={#1},
		postaction={
			decorate,decoration={
				markings,Mark/.list={#2}
			}
		}
	},
}

\makeatletter
\tikzset{commutative diagrams/.cd,arrow style=tikz,diagrams={>=latex'}}\tikzset{join/.code=\tikzset{after node path={%
			\ifx\tikzchainprevious\pgfutil@empty\else(\tikzchainprevious)%
			edge[every join]#1(\tikzchaincurrent)\fi}}}
\makeatother

\tikzset{>=stealth',every on chain/.append style={join},
	every join/.style={->}}

\tikzset{every loop/.style={min distance=25mm,in=50,out=100,looseness=5}}

\newtheorem{prf}{Proof}[section]

\theoremstyle{remark}

\newtheoremstyle{ownTheoremStyle}
{1em}
{1em}
{\itshape}
{}
{\bfseries}
{.}
{ }
{}

\newtheoremstyle{ownDefinitionStyle}
{1em}
{1em}
{}
{}
{\bfseries}
{.}
{ }
{}

\theoremstyle{ownTheoremStyle}

\newtheorem{thm}[prf]{Theorem}

\newtheorem{Introthm}{Theorem}

\newtheorem{lem}[prf]{Lemma}
\newtheorem{prp}[prf]{Proposition}

\newtheorem{cor}[prf]{Corollary}

\theoremstyle{ownDefinitionStyle}
\newtheorem{exa}[prf]{Example}
\newtheorem*{convention}{Convention}

\newtheorem{definition}[prf]{Definition}
\newtheorem{rem}[prf]{Remark}

\newtheorem{notation}[prf]{Notation}

\newcommand{\quotient}[2]{{\left.\raisebox{.2em}{$#1$}\middle/\raisebox{-.2em}{$#2$}\right.}}

\numberwithin{equation}{section}

\newcommand{\cA}{\mathcal{A}}
\newcommand{\cB}{\mathcal{B}}

\newcommand{\cF}{\mathcal{F}}

\newcommand{\bA}{\mathbb{A}}
\newcommand{\bB}{\mathbb{B}}

\newcommand{\bF}{\mathbb{F}}

\newcommand{\sC}{\mathsf{C}}

\newcommand{\punct}{\mathscr{P}}

\newcommand{\rmod}[1]{{\rm mod \,}#1}

\newcommand{\dg}[1]{\left\lVert #1\right\rVert}

\def\centerarc[#1](#2)(#3:#4:#5) 
{ \draw[#1] ($(#2)+({#5*cos(#3)},{#5*sin(#3)})$) arc (#3:#4:#5); }

\title{\MakeUppercase{Derived equivalence classification of Brauer graph algebras}}
\author{Sebastian Opper}

\address{Charles University, Faculty of Mathematics and Physics, Ke Karlovu 3, 121 16 Praha 2, Czech Republic}
\email{opper@karlin.mff.cuni.cz}

\author{Alexandra Zvonareva}
\address[]{Universit\"{a}t Stuttgart, Institut f\"ur Algebra und Zahlentheorie, Pfaffenwaldring 57, 70569 Stuttgart, Germany}
\email{alexandra.zvonareva@mathematik.uni-stuttgart.de}

\begin{document}

\begin{abstract}

We classify Brauer graph algebras up to derived equivalence by showing that the set of derived invariants introduced by Antipov is complete. These algebras first appeared in representation theory of finite groups and can be defined for any suitably decorated graph on an oriented surface.
 Motivated by the connection between Brauer graph algebras and gentle algebras we consider $A_{\infty}$-trivial extensions of partially wrapped Fukaya categories associated to surfaces with boundary. This construction naturally enlarges the class of Brauer graph algebras and provides a way to construct derived equivalences between Brauer graph algebras with the same derived invariants. As part of the proof we provide an interpretation of derived invariants of Brauer graph algebras as orbit invariants of line fields under the action of the mapping class group.\end{abstract}
	
		\maketitle
	\tableofcontents
	\setcounter{tocdepth}{1}
\section*{Introduction}

\noindent Brauer graph algebras are a class of finite dimensional algebras defined by quivers and relations. The quiver of a Brauer graph algebra can be constructed from a Brauer graph which is a graph embedded into an oriented surface in a minimal way together with a multiplicity function which assigns a non-zero natural number to each vertex. The minimality of the embedding, in particular, determines the surface uniquely up to homeomorphism. As one might expect the graph and the surface play a crucial role in understanding the representation theory  and the derived category of the associated Brauer graph algebra.

 Brauer graph algebras first emerged in modular representation theory in form of blocks with cyclic \cite{Dade} or dihedral defect group \cite{Donovan}. Thereafter, these algebras appeared in other classification results such as in the classification of self-injective cellular algebras of polynomial growth \cite{Ariki1KaseMiyamotoWada}, tame blocks of Hecke algebras \cite{ARIKI, ArikiRep}, symmetric $2$-Calabi-Yau-tilted algebras of finite representation type \cite{Ladkani}, and blocks of tame infinitesimal group schemes \cite{FarsteinerSkowronski}.  Recently, Brauer graph algebras also appeared in connection with dessins d'enfants \cite{SchrollMalic}. 
 
Over an algebraically closed field the class of Brauer graph algebras coincides with the class of symmetric special biserial algebras \cite{AntipovGeneralov, SchrollTrivialExtension}, whose representation theory has been studied quite extensively. All finite dimensional indecomposable modules over such algebras are classified in terms of string and band combinatorics \cite{GelfandPonomarev, DonovanFreislich, ButlerRingel, WaldWaschbusch} and the structure of the Auslander-Reiten components is well understood \cite{ErdmannSkowronsk}. Ext-algebras (also known as Yoneda algebras) of Brauer graph algebras are finitely generated \cite{AntipovGeneralov, GreenSchrollSnashallTaillefer}. Well-known homological  conjectures such as the finitistic dimension conjecture \cite{ErdmannHolmIyamaSchroeer} and the Auslander-Reiten conjecture on the number of simple modules \cite{Pogorzaly, AntipovZvonareva} have been established for this class of algebras. A more comprehensive list of results on Brauer graph algebras  can be found in the survey article \cite{Schroll}.
 
The class of Brauer graph algebras is closed under derived equivalence \cite{AntipovZvonareva, ZvonarevaAntipov}. The study of derived equivalences inside of this class has attracted attention in the recent years. The fact that most of these algebras are derived wild, which, for example, follows from \cite{BekkertHernanVelezMarulanda}, makes this study more interesting but also leads to certain difficulties.\ \medskip

\noindent The main result of this paper is a complete classification of Brauer graph algebras up to derived equivalence. We restrict ourselves to the non-local case since derived equivalent local algebras are isomorphic.
 
 \begin{Introthm}[Theorem \ref{TheoremClassificationBGAs}]\label{IntroTheoremCriterionDerivedEquivalence}
Let $B$ and  $B'$ be two Brauer graph algebras over an algebraically closed field $\Bbbk$ with connected Brauer graphs $\Gamma$ and $\Gamma'$ and assume that $B$ and $B'$ are not local. Then, $B$ and $B'$ are derived equivalent if and only if all of the following conditions are satisfied.
\begin{enumerate}
\setlength\itemsep{1ex}
	\item $\Gamma$ and $\Gamma'$ have the same number of vertices, edges and faces, in particular the surfaces of $B$ and $B'$ are homeomorphic;
    \item \label{Enum2} the multi-sets of perimeters of faces and the multi-sets of the multiplicities of vertices of $\Gamma$ and $\Gamma'$ coincide;
    \item \label{Enum3} Either both or none of $\Gamma$ and $\Gamma'$ are bipartite.
\end{enumerate}
\end{Introthm}

\noindent  All invariants in Theorem \ref{IntroTheoremCriterionDerivedEquivalence} can be easily computed from the Brauer graph and their interpretation as invariants of the derived category of a Brauer graph algebra is briefly recalled in Remark \ref{RemarkDerivedInvariants}. For example, the number of edges of a Brauer graph agrees with the rank of the Grothendieck group, whereas the number and the perimeters of faces are related to the number and the ranks of tubes in the Auslander-Reiten quiver of the singularity category.  

The first result in the spirit of Theorem \ref{IntroTheoremCriterionDerivedEquivalence} was shown by Rickard in connection with stable equivalences of blocks with cyclic defect group. He showed that, up to derived equivalence, a Brauer tree algebra is determined by the
number of edges of its Brauer tree and the multiplicity at the exceptional vertex \cite{RickardDerivedStable} (see also \cite{MembrilloHernandez}). The derived invariants used in Theorem \ref{IntroTheoremCriterionDerivedEquivalence} were described by Antipov \cite{Antipov} (see also \cite{ZvonarevaAntipov}) who realized the importance of the surface of the Brauer graph for the study of derived equivalences. Using ad-hoc methods he proved a vast generalization of Rickard's result by showing that these invariants are sufficient to distinguish derived equivalence classes of Brauer graph algebras whose surfaces have genus zero. However, it is known that his approach does not generalize to surfaces of higher genus, see Remark \ref{RemarkAAC}. 

The strategy of the proof of Theorem \ref{IntroTheoremCriterionDerivedEquivalence} relies on the interpretation of Brauer graph algebras with trivial  multiplicities as trivial extensions of gentle algebras \cite{SchrollTrivialExtension} and the connection between gentle algebras and partially wrapped Fukaya categories \cite{HaidenKatzarkovKontsevich,LekiliPolishchukGentle,OpperPlamondonSchroll} which led to a classification of gentle algebras up to derived equivalence \cite{AmiotPlamondonSchroll, OpperDerivedEquivalences}. There is a strong connection between derived equivalence of algebras and of their trivial extensions: a theorem of Rickard states that the former implies the latter \cite{RickardDerivedStable}. However, the naive approach of trying to deduce a derived equivalence classification of Brauer graph algebras from the derived equivalence classification of gentle algebras does not work; for more details see Remark \ref{RemarkMainTheoremRickardTheorem} and Example \ref{ExampleBGACuts}.

Our approach is based on the idea that one should consider a larger class of $A_\infty$-algebras to allow
for more derived equivalences as easy intermediate steps. Building on \cite{BocklandtPuncturedSurface} and \cite{HaidenKatzarkovKontsevich}, we introduce a class of $A_{\infty}$-categories which we call \textit{Brauer graph categories}. These categories are associated to certain collections of arcs on \textit{graded} punctured surfaces, that is punctured surfaces equipped with a line field, whose punctures are endowed with multiplicities. When all multiplicities are equal to $1$, the morphisms in the Brauer graph category correspond, roughly speaking, to intersections of the arcs from the collection of arcs and their degrees are induced from the line field.
In the special cases when the surface, the line field and the arc system arise from a Brauer graph, the homotopy category of  twisted complexes over the Brauer graph category is equivalent to the category of perfect complexes over the corresponding Brauer graph algebra. By a well-known theorem of Rickard \cite{RickardMoritaTheory} any equivalence between the categories of perfect complexes amounts to an equivalence of the corresponding derived categories, so it is enough to study such equivalences for our purposes.

We prove that any Brauer graph category with trivial multiplicities is the trivial extension of an $A_{\infty}$-category from \cite{HaidenKatzarkovKontsevich} corresponding to the partially wrapped Fukaya category of a marked surface obtained by ``dragging'' punctures to the boundary. Brauer graph categories with higher multiplicities are interpreted as orbit categories of Brauer graph categories with trivial multiplicities associated with branched covers of the initial surface under the action of their deck transformation groups. This construction extends some of the results from \cite{GreenSchrollSnashall,AsashibaCovers} to our setup.

The $A_\infty$-structure obtained in this way guarantees that adding and deleting certain arcs yields a Morita equivalence between the corresponding $A_\infty$-categories, similarly to the case of Fukaya categories \cite{BocklandtPuncturedSurface, HaidenKatzarkovKontsevich}. An analogue of the classical result stating that any two triangulations are connected by a sequence of flips implies that the Morita equivalence class of a Brauer graph category is independent of the chosen arc collection and depends only on the surface and the homotopy class of the line field. Given a self-diffeomorphism $f$ which identifies  two line fields on a given surface we can fix an arc collection and transport it along $f$ without changing the associated Brauer graph category. As a result, the Morita equivalence class of a Brauer graph category only depends on the orbit of the homotopy class of the line field under the action of the mapping class group. 
Subsequently, Theorem \ref{IntroTheoremCriterionDerivedEquivalence} is derived from the classification of such 
orbits after \cite{LekiliPolishchukGentle} by expressing the invariants in conditions \eqref{Enum2} and \eqref{Enum3} in Theorem \ref{IntroTheoremCriterionDerivedEquivalence} in terms of orbit invariants of the line field and vice versa. From the point of view of graded surfaces, Theorem \ref{IntroTheoremCriterionDerivedEquivalence} can be rephrased in the following way.
\begin{Introthm}
With the assumptions of Theorem \ref{IntroTheoremCriterionDerivedEquivalence}, two Brauer graph algebras are derived equivalent if and only if their associated graded surfaces are diffeomorphic and their multi-sets of multiplicities coincide.
\end{Introthm}

\noindent We wonder whether certain variations of Brauer graph categories admit an interpretation as Fukaya categories analogous to the case of homologically smooth gentle algebras which can be thought of as formal generators of partially wrapped Fukaya categories of marked surfaces. The trivial extension construction appears in several places in the literature where one starts with a generator of a triangulated category associated to a space of algebraic or geometric nature and produces a generator of an analogous category over a fibration of said space. 

Two instances of this phenomenon are the following. Suppose $Z$ is a smooth Fano variety over the complex numbers of dimension $n-1$ and let $\mathbb{E}=\End^{\bullet}(X)$ be the extension algebra of a generator $X \in \Db{Z}$ naturally equipped with a minimal $A_{\infty}$-structure. Let $\omega_Z$ denote the canonical bundle of $Z$ and $\jmath: Z \hookrightarrow \omega_Z$ the corresponding zero section. According to \cite{BridgelandStern}, $\jmath_{\ast}X$ is a generator of $\Db{\omega_Z}$ and by \cite{SegalDeformationTheoryPoint} and \cite{BallardSheavesLocalCalabiYau}, the $A_{\infty}$-trivial extension of $\mathbb{E}$ by the shifted dual bimodule $\left(\mathbb{D}\mathbb{E}\right)[-n]$ (called the \textit{$n$-dimensional cyclic completion} in \cite{SegalDeformationTheoryPoint}) is a minimal model of $\End^{\bullet}(\jmath_{\ast}X)$. A somewhat similar appeareance of the trivial extension can be found in \cite{SeiBook}. Here one considers a generator of the directed Fukaya category associated to a regular fiber of a Lefschetz fibration and passing to a trivial extension yields a triangulated category which contains the appropriate Fukaya category associated to the total space of the fibration. In this context, there exists a natural geometric interpretation of the passage from the cyclic completions of dimension $m$ to the $m-2$-dimensional completion as a two-fold \textit{suspension} of the Lefschetz fibration due to Seidel \cite{SeidelLefFib} which increases the (real) dimension of the total space by $4$. In light of the previous discussion, trivial extensions of  $A_{\infty}$-categories $\mathcal{F}=\mathcal{F}_{A}(S)$ of an arc system $\cA$ from \cite{HaidenKatzarkovKontsevich} by $\left(\mathbb{D}{\cF}\right)[-2]$ seems to be a more natural candidate to have an interpretation in terms of Fukaya categories and might suggest to view Brauer graph categories as a form of de-suspension of the former.
\vspace{-5pt}

\subsection*{Acknowledgements}
The authors thank Wassilij Gnedin and Mikhail Gorsky for encouraging discussions and Isaac Bird and Steffen Koenig for their constructive feedback on the first version of this paper. The authors thank the anonymous referee for helpful comments and suggestions. The first named author was supported by the Czech Science Foundation as part of the project ``Symmetries, dualities and approximations in derived algebraic geometry and representation theory'' (20-13778S) and the Charles University Research Center program (UNCE/SCI/022). Part of the research was carried out during several stays of the first named author at the University of Stuttgart. He would like to thank Steffen Koenig and the members of the Representation Theory group for their hospitality. Both authors were participating in the Junior Trimester Program "New Trends in Representation Theory"  while working on this paper and would like to thank the Hausdorff Institute of Mathematics in Bonn for organizing the program.

\addtocontents{toc}{\protect\setcounter{tocdepth}{0}}
\section*{Conventions} 
\noindent Throughout this paper $\Bbbk$ will denote an algebraically closed field. For a finite-dimensional $\Bbbk$-algebra $A$ we will denote by $\rmod{A}$ its category of finite-dimensional right  modules; by $\DA$ the bounded derived category of $\rmod{A}$ and by $\Kb{A}$ the homotopy category of bounded complexes of finitely generated projective $A$-modules, also called the category of perfect complexes. The product of arrows inside the path algebra of a quiver is to be read from  right to left, i.e. a path $3\xleftarrow{\beta}2\xleftarrow{\alpha}1$ coincides with the product $\beta\alpha$. For an arrow $\alpha$ we will denote by $s(\alpha)$ its source and by $t(\alpha)$ its target. By $-\otimes-$ we will mean the tensor product over $\Bbbk$.
	
\addtocontents{toc}{\protect\setcounter{tocdepth}{1}}	
	
\section{Brauer graph algebras and their derived invariants} \label{SectionBrauerGraphAlgebras}
\noindent In this section we recall the definition of a Brauer graph algebra and its ribbon graph surface as well as list its known derived invariants.

\subsection{Ribbon graphs and their ribbon surfaces}\ \medskip

\noindent Suppose we are given an embedding 
$\kappa: \Gamma \hookrightarrow \Sigma$ of a graph $\Gamma$ (consisting of vertices and edges) into the interior of an oriented surface $\Sigma$. Assume that $D$ is a small disc which contains a single vertex $v$ of $\Gamma$ and whose boundary meets any edge of $\Gamma$ not more than once. Then, each connected component of $D \cap \left(\Gamma \setminus \{v\}\right)$ is a segment of an edge. We call such a segment a \textit{half-edge} and the union of all half-edges in $D \cap \left(\Gamma \setminus \{v\}\right)$ will be denoted by $H_v$. The orientation of $\Sigma$ endows the set $H_v$ of all half-edges around $v$ with a cyclic order, say the counterclockwise order. This can be rephrased by saying that there is a cyclic permutation $\rho_v: H_v \rightarrow H_v$ which maps any half-edge to its successor in the cyclic order. Passing to the set $H=\bigsqcup_{v} H_v$ of all half-edges, the cyclic permutations give rise to a permutation $\rho: H \rightarrow H$.

If $w$ is another vertex (possibly $v=w$) and $e$ is an edge which connects $v$ with $w$, then $e$ contains a unique pair of half-edges $(h_v, h_w)$ at $v$ and $w$. One may phrase this fact for all vertices at once by defining an involution $\iota: H \rightarrow H$ without fixed points. In the above case, $\iota(h_v)=h_w$. The orbits of $\iota$ are in bijection with the edges of $\Gamma$ and each orbit consists of the two half-edges associated with an edge. \medskip

\noindent Once we neglect the embedding $\kappa$, one  arrives at the combinatorial definition of a \textit{ribbon graph}:

\begin{definition}\label{DefinitionRibbonGraph}
A \textbf{ribbon graph} is a tuple $\Gamma=(V, H, s, \iota, \sigma)$, where

\begin{enumerate}
    \item $V$ is a finite set whose elements are called \textbf{vertices};
    
    \item $H$ is a finite set whose elements are called \textbf{half-edges};
    
    \item $s: H \rightarrow V$ is a function;
    
    \item $\iota: H \rightarrow H$ is an involution without fixed points, and
    
    \item $\rho: H \rightarrow H$ is a permutation whose cycles correspond to the sets $H_v \coloneqq s^{-1}(v)$, $v \in V$.
\end{enumerate}
\end{definition}

\noindent Every ribbon graph defines a graph with vertex set $V$ whose edges are the orbits of $\iota$. An edge $\{h, \iota(h)\}$ is incident to the vertices $s(h)$ and $s(\iota(h))$. Any ribbon graph gives rise to a surface:
\begin{lem}
Up to homeomorphism, there exists a unique compact, oriented surface $\Sigma_{\Gamma}$ with boundary and a (non-unique) embedding $\kappa:\Gamma \hookrightarrow \Sigma_{\Gamma}$ into the interior of $\Sigma_\Gamma$ such that $\Gamma$ is isomorphic to the ribbon graph of the embedding $\kappa$ and such that $\kappa$ is \textit{filling}, i.e.\ $\Sigma_{\Gamma} \setminus \Gamma \cong \partial \Sigma_{\Gamma} \times [0,1)$ is a union of half-open annuli.
\end{lem}
\noindent The surface $\Sigma_{\Gamma}$ is called the \textbf{ribbon surface} of $\Gamma$.

\begin{notation}
If $\Gamma$ is a ribbon graph, we write $V(\Gamma)$ and $H(\Gamma)$ for its sets of vertices and half-edges as well as $H_v(\Gamma)$ for all half-edges $h \in \BH{\Gamma}$ such that $s(h)=v \in \BV{\Gamma}$. To simplify the notation, we often write $h^{\pm} \coloneqq \rho^{\pm 1}(h)$ for the successor and predecessor of a half-edge $h$ and $\overline{h}$ for its associated edge. The set of all edges is denoted by $\BE{\Gamma}$.
\end{notation}

\noindent Unless stated otherwise, we will assume that $\Gamma$ is connected, i.e.\ its underlying graph is connected.

\begin{definition}\label{DefinitionMorphismRibbonGraph}
Let $\Gamma, \Gamma'$ be ribbon graphs. A \textbf{morphism} $\varphi:\Gamma \rightarrow \Gamma'$ is a pair $(\varphi_V, \varphi_H)$ consisting of maps $\varphi_V: \BV{\Gamma} \rightarrow \BV{\Gamma'}$ and $\varphi_H: \BH{\Gamma} \rightarrow \BH{\Gamma'}$ which commute with the functions $\rho, \iota$ and $s$. That is, for all $h\in \BH{\Gamma}$, $s\big(\varphi_H(h)\big)=\varphi_V(s(h))$, $\iota\big(\varphi_H(h)\big)= \varphi_H(\iota(h))$ and $\rho(\varphi_H(h))=\varphi_H(\rho(h))$.
\end{definition}

\subsection{Brauer graphs and their algebras}
In this section we recall the definition of a \textit{Brauer graph algebra}.

\begin{definition}\label{DefinitionBrauerGraph} A \textbf{Brauer graph} is a pair $(\Gamma, \seq)$ consisting of a ribbon graph  $\Gamma$ and a function 
$\seq: \BV{\Gamma}\rightarrow \mathbb{N}$.
\end{definition}
\noindent The function $\seq$ in Definition \ref{DefinitionBrauerGraph} is referred to as the \textbf{multiplicity function} and its values as \textbf{multiplicities}. Frequently, we omit $\seq$ from the notation and refer to $\Gamma$ as a Brauer graph. We denote by $\mathbf{n}$ any constant multiplicity function with value $n$ at each vertex and say that a Brauer graph $(\Gamma, \seq)$ is \textbf{multiplicity-free} if $\seq=\mathbf{1}$.\medskip

\noindent To any Brauer graph $(\Gamma, \seq)$ one can associate a quiver $Q=Q_\Gamma$ and an ideal of relations $I=I_\Gamma$ in the path algebra $\Bbbk Q$ as follows.

\begin{enumerate}
    \item The vertices of $Q$ correspond to the edges of $\Gamma$ and for every $h \in  H$, there is an arrow $\alpha_h: \overline{h}\rightarrow \overline{h^+}$. The assignment $\alpha_h \mapsto \alpha_{h^-}$ defines a permutation $\pi=\pi_{\Gamma}$ of the arrows of $Q$ whose orbits are in bijection with $\BV{\Gamma}$. Hence every arrow $\alpha$ defines a closed path 
\begin{displaymath}
C_{\alpha}=\alpha \pi(\alpha) \cdots \pi^l(\alpha),
\end{displaymath} 

\noindent where $l+1$ denotes the cardinality of the $\pi$-orbit of $\alpha$. 
Every vertex of $Q$ is the starting point of exactly two cycles of the form $C_{\alpha}$. If $\alpha=\alpha_h$, set $\seq(C_{\alpha}) \coloneqq \seq(s(h))$.

\item The ideal $I_\Gamma$ is generated by the following set of relations:\smallskip

\begin{enumerate}

    \Item  
    \begin{align*}  
      \alpha\beta,
    \end{align*}  
    \noindent where $\alpha, \beta \in Q_1$ are composable and
    $\pi(\alpha)\neq \beta$;
     \Item  
    \begin{align*}  
      C_\alpha^{\seq(C_\alpha)} - C_{\beta}^{\seq(C_{\beta})},
    \end{align*} 
    \noindent where $\alpha, \beta \in Q_1$ and $t(\alpha)=t(\beta)$, that is $\alpha$ and $\beta$ end at the same edge of $\Gamma$.
\end{enumerate}

\end{enumerate}

\noindent The resulting finite dimensional $\Bbbk$-algebra $\Bbbk Q_{\Gamma} / I_{\Gamma}$ will be denoted by $B_{\Gamma}$. By a \textbf{path of $B_{\Gamma}$} we mean a path of $Q_{\Gamma}$ which is not contained in $I_{\Gamma}$. Note that every path of $B_{\Gamma}$ is a subpath of a cycle $C_\alpha^{\seq(C_\alpha)}$ for some arrow $\alpha$. 
\begin{definition}
 A $\Bbbk$-algebra $B$ is called a \textbf{Brauer graph algebra} if there exists a Brauer graph $(\Gamma, \seq)$ such that $B \cong B_{\Gamma}$ as $\Bbbk$-algebras.
\end{definition}
\noindent We will always use the presentation of the form $\Bbbk Q_\Gamma/I_\Gamma$ when working with Brauer graph algebras.
\begin{exa}
Consider a triangle in the plane with multiplicities $1, 2$ and $3$ as indicated in Figure \ref{FigureExampleBGA}. Its ribbon surface  is homeomorphic to an annulus. The associated Brauer graph algebra is generated by the idempotents corresponding to the three edges of $\Gamma$ and arrows $\alpha, \dots, \zeta$ in Figure \ref{FigureExampleBGA}. The ideal of relations $I_\Gamma$ is generated by
\begin{displaymath}
 \big\{\varepsilon \alpha, \delta \beta, \,  \gamma \varepsilon, \beta \zeta, \,  \alpha \gamma, \zeta\delta \big\} \cup \left\{(\beta \alpha)^2-(\gamma \delta),\, (\delta \gamma)-(\varepsilon \zeta)^3,\, (\zeta \varepsilon)^3- (\alpha \beta)^2\right\}.
\end{displaymath}

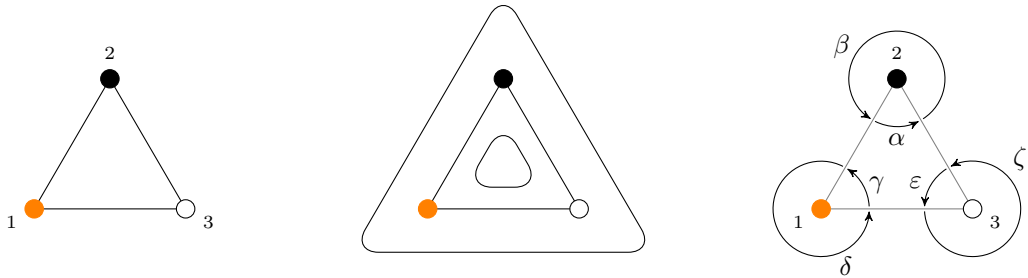
\begin{figure}[h]
\begin{tikzpicture}[scale=1.15]

\begin{scope}[shift={(-4.5,0)}]

\def\r{1};
\def\thickness{3pt}; 
\def\factor{1.3}; 
\def\rtwo{0.75}; 
\def\posi{0.75}; 
\def\dist{-7pt};

 \foreach \i in {1,2,3}
 \draw ({\factor*\r*cos(330-\i*120)},{\factor*\r*sin(330-\i*120)}) node[black]{$\scriptstyle \i$};
 
 \foreach \i in {1,2,3}
 \draw ({\r*cos(330-\i*120)},{\r*sin(330-\i*120)})--({\r*cos(210-\i*120)},{\r*sin(210-\i*120)});

 \draw[fill=white] ({\r*cos(-30)},{\r*sin(-30)}) circle (\thickness);
 \filldraw[black] ({\r*cos(90)},{\r*sin(90)}) circle (\thickness);
 \filldraw[orange] ({\r*cos(210)},{\r*sin(210)}) circle (\thickness);

\end{scope}
\begin{scope}

\def\r{1};
\def\thickness{3pt}; 
\def\factor{1.3}; 
\def\rtwo{0.75}; 
\def\posi{0.75}; 
\def\dist{-7pt};

 \foreach \i in {1,2,3}
 \draw ({\factor*\r*cos(330-\i*120)},{\factor*\r*sin(330-\i*120)});
 
 \foreach \i in {1,2,3}
 \draw ({\r*cos(330-\i*120)},{\r*sin(330-\i*120)})--({\r*cos(210-\i*120)},{\r*sin(210-\i*120)});
\foreach \i in {0.5, 2}
 \draw[rounded corners=3.5mm] ({330-1*120}:{\i*\r})--({330-2*120}:{\i*\r})--({330-3*120}:{\i*\r})--cycle;

 \draw[fill=white] ({\r*cos(-30)},{\r*sin(-30)}) circle (\thickness);
 \filldraw[black] ({\r*cos(90)},{\r*sin(90)}) circle (\thickness);
 \filldraw[orange] ({\r*cos(210)},{\r*sin(210)}) circle (\thickness);
\end{scope}
\begin{scope}[shift={(4.5,0)}]
\def\r{1};
\def\thickness{3pt}; 
\def\factor{1.3}; 
\def\rtwo{0.55}; 
\def\posi{0.75}; 
\def\dist{-7pt};

 \foreach \i in {1,2,3}
 \draw ({\factor*\r*cos(330-\i*120)},{\factor*\r*sin(330-\i*120)}) node[black]{$\scriptstyle \i$};
 
 \foreach \i in {1,2,3}
 \draw[opacity=0.5] ({\r*cos(330-\i*120)},{\r*sin(330-\i*120)})--({\r*cos(210-\i*120)},{\r*sin(210-\i*120)});

 \draw[fill=white] ({\r*cos(-30)},{\r*sin(-30)}) circle (\thickness);
 \filldraw[black] ({\r*cos(90)},{\r*sin(90)}) circle (\thickness);
 \filldraw[orange] ({\r*cos(210)},{\r*sin(210)}) circle (\thickness);

\def\j{2};
  \draw[->] ({\r*cos(330-\j*120)+\rtwo*cos(180-\j*120+3)},{\r*sin(330-\j*120)+\rtwo*sin(180-\j*120+3)}) arc ({180-\j*120+3}:{210-\j*120+270-3}:{\rtwo}) node[pos=0.75, label={[label distance=\dist]{180-\j*120+\posi*300}:$\beta$}] {};
    \draw[->] ({\r*cos(330-\j*120)+\rtwo*cos(120-\j*120+3)},{\r*sin(330-\j*120)+\rtwo*sin(120-\j*120+3)}) arc ({120-\j*120+3}:{180-\j*120-3}:{\rtwo}) node[pos=0.5, label={[label distance=\dist*2]{-30-\j*120}:$\alpha$}] {};
    
\def\j{3};
  \draw[->] ({\r*cos(330-\j*120)+\rtwo*cos(180-\j*120+3)},{\r*sin(330-\j*120)+\rtwo*sin(180-\j*120+3)}) arc ({180-\j*120+3}:{210-\j*120+270-3}:{\rtwo}) node[pos=0.75, label={[label distance=\dist]{180-\j*120+\posi*300}:$\zeta$}] {};
    \draw[->] ({\r*cos(330-\j*120)+\rtwo*cos(120-\j*120+3)},{\r*sin(330-\j*120)+\rtwo*sin(120-\j*120+3)}) arc ({120-\j*120+3}:{180-\j*120-3}:{\rtwo}) node[pos=0.5, label={[label distance=\dist*2.5]{-30-\j*120}:$ \varepsilon$}] {};

\def\j{1};
  \draw[->] ({\r*cos(330-\j*120)+\rtwo*cos(180-\j*120+3)},{\r*sin(330-\j*120)+\rtwo*sin(180-\j*120+3)}) arc ({180-\j*120+3}:{210-\j*120+270-3}:{\rtwo}) node[pos=0.75, label={[label distance=\dist]{180-\j*120+\posi*300}:$\delta$}] {};
    \draw[->] ({\r*cos(330-\j*120)+\rtwo*cos(120-\j*120+3)},{\r*sin(330-\j*120)+\rtwo*sin(120-\j*120+3)}) arc ({120-\j*120+3}:{180-\j*120-3}:{\rtwo}) node[pos=0.5, label={[label distance=\dist*2.5]{-30-\j*120}:$\gamma$}] {};

\end{scope}
\end{tikzpicture}
    \caption{A Brauer graph (left) with its ribbon surface (middle) and its quiver (right).}
    \label{FigureExampleBGA}
\end{figure}

\end{exa}

\begin{convention} Unless stated otherwise we assume that a Brauer graph has at least two edges. This is equivalent to the condition that the corresponding Brauer graph algebra is not local.
\end{convention}

\noindent Note that any two local algebras are derived equivalent if and only if they are Morita equivalent \cite[Corollary 2.13]{RouquierZimmermann}. Since Morita equivalent basic algebras are isomorphic, two local Brauer graph algebras are derived equivalent if and only if they are isomorphic. \medskip

\begin{rem}\label{RemarkLocalCase} The Brauer graph of a non-local Brauer graph algebra $B$ is an invariant of its isomorphism class and is denoted by $\Gamma_{B}$. However, there are two exceptions in the local case: the Brauer graph algebras of the Brauer graphs $(\Gamma, \mathbf{1})$, where $\Gamma$ is a loop, and $(\Gamma',  \mathbf{2})$, where $\Gamma'$ consists of two vertices and a single edge, are isomorphic for certain choices of $\Bbbk$ \cite{ZvonarevaAntipov}.
\end{rem}

\noindent Later on we will need the following modification of a Brauer graph algebra which is a specialization of a \textit{quantized Brauer graph algebra} in the sense of \cite{GreenSchrollSnashall}. It differs from the original definition only by a sign in the defining relations.

\begin{definition}
 Let $(\Gamma, \seq)$ be a Brauer graph and let $\omega:\BE\Gamma\rightarrow \mathbb{Z}$ be a function which assigns an integer to each edge of $\Gamma$. The \textbf{modified Brauer graph algebra}, associated to $(\Gamma,\seq,\omega)$ is defined as $B\simeq \Bbbk Q_\Gamma/I'_\Gamma$, where $I'_\Gamma$ is generated by the following relations:
\begin{enumerate}

    \Item  
    \begin{align*}  
      \alpha\beta,
    \end{align*}  
    \noindent where $\alpha, \beta \in Q_1$ are composable and
    $\pi(\alpha)\neq \beta$;
     \Item  
    \begin{align*}  
      C_\alpha^{\seq(C_\alpha)} - (-1)^{\omega(\bar{h})}C_{\beta}^{\seq(C_{\beta})},
    \end{align*} 
    \noindent where $\alpha, \beta \in Q_1$ and $t(\alpha)=t(\beta)$, $\alpha=\alpha_{h^-}$, $\beta=\alpha_{\iota(h)^-}$, that is the sign in the relation arises from the unique edge of $\Gamma$ at which $\alpha$ and $\beta$ end.
\end{enumerate} 
 
\end{definition}

\noindent If $\omega = 0$, we recover the usual definition of a Brauer graph algebra.

\subsection{Derived invariants of Brauer graph algebras} \ \medskip

\noindent Let $(\Gamma, \seq)$ be a Brauer graph and $\Sigma=\Sigma_{\Gamma}$ its ribbon surface. In particular, $\Sigma \setminus \Gamma \cong \partial \Sigma \times [0,1)$ and hence every connected component of $\Sigma \setminus \Gamma$ is bounded by a unique boundary component $B \subseteq \partial \Sigma$ and a subset $\Delta_B \subset \Gamma$ of edges. The cyclic sequence of half-edges of this (possibly self-glued) polygon encodes a \textit{face} of $\Gamma$.

\begin{definition}\label{DefinitionFaces} A \textbf{face} of \textbf{perimeter} $m \in \mathbb{N}$ is an equivalence class of primitive cyclic sequences

\begin{displaymath}
F=(h_1, \dots, h_{2m}),
\end{displaymath}

\noindent where $h_i \in \BH{\Gamma}$ are such  that $h_{2i+2}=h_{2i+1}^+$ and   $h_{2i+1}=\iota(h_{2i})$ for all $i \in [0,m-1]$ (with indices modulo $2m$). We call a sequence $F$ primitive, if $F$ is not a power of another sequence of this form. Two sequences of such kind are considered equivalent if they agree after a cyclic permutation of their entries.
\end{definition}

\noindent The set of faces of $\Gamma$ is denoted by $\BF{\Gamma}$ and is in bijection with the set of boundary components of $\Sigma_{\Gamma}$. The perimeter of a face coincides with the number of edges of the corresponding polygon $\Delta_B$, where edges are counted separately if they are glued together.\medskip

\noindent We recall that a graph $\Gamma$ is \textbf{bipartite} if and only if every cycle in $\Gamma$ has even length. Equivalently, $\Gamma$ is bipartite if its set of vertices $V$ admits a partition $V=V_1 \sqcup V_2$ into  disjoint subsets $V_1$ and $V_2$ such that every edge connects a vertex in $V_1$ with a vertex in $V_2$. Define $\sigma(\Gamma) \in \{0,1\}$ to be $0$ if $\Gamma$ is bipartite and $1$ otherwise.\medskip

\noindent The following theorem contains a  list of derived invariants for Brauer graph algebras.

\begin{thm}[\cite{Antipov}, \cite{ZvonarevaAntipov}]\label{TheoremDerivedInvariantsBrauerGraphAlgebra}
Let $\Gamma$ and $\Gamma'$ be Brauer graphs (with at least two edges by convention). If $B_{\Gamma}$ and $B_{\Gamma'}$ are derived equivalent, that is $\mathsf{D}^b(B_{\Gamma})$ and $\mathsf{D}^b(B_{\Gamma'})$ are equivalent as triangulated categories, then the following data associated with $\Gamma$ and $\Gamma'$ coincide:
\begin{enumerate}
    \item \label{Surfaces} the number of vertices, the number of edges and the number of faces of $\Gamma$ and $\Gamma'$, in particular, the surfaces $\Sigma_{\Gamma}$ and $\Sigma_{\Gamma'}$ are homeomorphic: $\Sigma_{\Gamma} \cong \Sigma_{\Gamma'}$;
    \item \label{NumberFaces} the multi-sets of perimeters of faces of $\Gamma$ and $\Gamma'$;
    \item the multi-sets of multiplicities of $\Gamma$ and $\Gamma'$;
    
    \item  $\sigma(\Gamma)$ and $\sigma(\Gamma')$.
\end{enumerate}
\end{thm}
\noindent  Theorem \ref{IntroTheoremCriterionDerivedEquivalence}  states that the converse of Theorem \ref{TheoremDerivedInvariantsBrauerGraphAlgebra} holds.

\begin{rem}\label{RemarkDerivedInvariants}Let $A$ be a Brauer graph algebra with ribbon graph $(\Gamma, \seq)$. The invariants in Theorem \ref{TheoremDerivedInvariantsBrauerGraphAlgebra} admit the following interpretation as invariants of $\DA$:
\begin{enumerate}
    \item The rank of the Grothendieck group $\mathcal{K}_0(\DA)$ is $|\BE{\Gamma}|$;
    \item The rank of $\mathcal{K}_0(\operatorname{\mathsf{D}_{\text{sg}}}(A))$ is $|\BE{\Gamma}|-|\BV{\Gamma}|+1-\sigma(\Gamma)$ (see \cite{Antipov4}). Here, $\operatorname{\mathsf{D}_{\text{sg}}}(A)=\quotient{\DA}{\Kb{A}}$ denotes the singularity category of $A$. Since $A$ is symmetric, $\operatorname{\mathsf{D}_{\text{sg}}}(A)$ is equivalent the stable module category of $A$ \cite{Buchweitz, RickardDerivedStable} and invariants of the stable category are also derived invariants.
   \item The multi-set of perimeters (and hence $|\BF{\Gamma}|$) is, roughly speaking, encoded in the ranks of tubes in the Auslander-Reiten quiver of the singularity category. 
    \item The center of $A$ (or, equivalently the $0$-th Hochschild cohomology $\operatorname{HH}^0(A,A)$) determines  the multi-set of non-trivial multiplicities;
   \end{enumerate}
\end{rem}	
\section{\texorpdfstring{Preliminaries on $A_{\infty}$-categories}{Preliminaries on A-infinity categories}} \label{SectionPreliminariesAinfinity}
\noindent In this section we recall some necessary background material on $A_{\infty}$-categories, categories of twisted complexes and trivial extensions of $A_{\infty}$-categories.

\subsection{$A_{\infty}$-categories, homotopy categories and strict functors}

\begin{definition}	
	
\noindent An \textbf{$A_\infty$-category}  $\bA$ consists of a set of objects $\Ob{\mathbb{A}}$,
a $\mathbb{Z}$-graded $\Bbbk$-vector space $\Hom_\mathbb{A}(X,Y)$  for each pair of objects $X,Y \in \Ob{\mathbb{A}}$, and graded $\Bbbk$-linear maps
$$
\begin{tikzcd} \mu^n:=\mu^n_\bA : \Hom_\mathbb{A}(X_{n-1},X_n) \otimes \cdots \otimes \Hom_\mathbb{A}(X_0,X_1) \arrow{r} & \Hom_\mathbb{A}(X_0,X_n) \end{tikzcd}$$

\noindent of degree $2 - n$ for each $n \geq 1$, satisfying the $A_\infty$-relations:

\begin{equation}\label{EquationAInfinityConstraint}
   \sum _{k,m}(-1)^{\dg{a_k}+\cdots +\dg{a_1}}\mu^{n-m+1}(a_n,..., a_{k+m+1},\mu^{m}(a_{k+m},..., a_{k+1}), a_k,..., a_1)=0,
\end{equation}

\noindent for each sequence of composable elements $a_n,\dots,a_2,a_1$, where $\dg{a_i}=|a_i|-1$ denotes the \textit{reduced} degree of $a_i$.
\end{definition}

\noindent   If suitable, we will omit $\mathbb{A}$ and the arity from the structure maps $\mu^n_{\mathbb{A}}$.

\begin{definition}
An $A_\infty$-category $\bA$ is \textbf{strictly unital} if for every object $X \in \Ob{\mathbb{A}}$ there is an element $1_X \in \Hom_{\mathbb{A}}^0(X, X)$ such that 
$\mu^1(1_X)=0,$ and  $\mu^2(a, 1_X) = (-1)^{|a|}\mu^2(1_Y , a) = a$, for all $a \in \Hom_{\mathbb{A}}(X, Y)$, and $\mu^n(\dots, 1_X, \dots) = 0$ for all $n \geq 3$.
\end{definition}

\noindent Strictly unital $A_\infty$-categories with $\mu^k = 0$ for $k \geq 3$ correspond to small dg-categories with differential
$da := (-1)^{|a|}\mu^1(a)$ and composition $a \circ b:=(-1)^{|b|}\mu^2(a,b)$. In particular, any strictly unital $A_{\infty}$-category $\bA$ such that $\mu^k_{\bA}=0$ for all $k \neq 2$ can be naturally identified with a graded $\Bbbk$-linear category.

\begin{convention}
Every $A_{\infty}$-category which we consider in this paper will be strictly unital. 
\end{convention}

\noindent  It follows from \eqref{EquationAInfinityConstraint} that the maps $\mu^1: \Hom_{\mathbb{A}}(X,Y) \rightarrow \Hom_{\mathbb{A}}(X,Y)$ can be considered as differentials, i.e.\ $\mu_{\bA}^1 \circ \mu_{\bA}^1=0$.

\begin{definition} The \textbf{homotopy category} $\H^0\bA$ is the category whose objects are the objects of $\bA$ and whose morphism spaces $\Hom_{\H^0\bA}(X,Y)$ are the $0$-th cohomology groups $\H^0(\Hom_\bA(X,Y))$ of the complexes $\Hom_\bA(X,Y)$ with respect to the differential $\mu^1$. The composition of classes $[a]$ and $[b]$ is given by $[a] \circ [b]:=[\mu^2(a,b)]$. 
\end{definition}

\noindent The homotopy category of $\bA$ is an ordinary $\Bbbk$-linear category. Next, we recall the appropriate notion of functors between $A_\infty$-categories. However, as it suffices for our purposes, we will restrict ourselves to strict $A_{\infty}$-functors and will not define $A_{\infty}$-functors in full generality.

\begin{definition}
	Let $\mathbb{A}, \mathbb{B}$ be $A_{\infty}$-categories. A \textbf{(strict) $A_{\infty}$-functor} $F: \mathbb{A} \rightarrow \mathbb{B}$ is a map $\Ob{\bA} \rightarrow \Ob{\bB}$ together with graded $\Bbbk$-linear maps $\Hom_{\mathbb{A}}(X,Y) \rightarrow \Hom_{\mathbb{B}}(F(X), F(Y))$ for all pairs $(X,Y) \in \Ob{\mathbb{A}}^2$ such that for all homogeneous elements	
	\begin{displaymath}
	a_n \otimes \cdots \otimes a_1 \in \Hom_\bA(X_{n-1}, X_n) \otimes \dots \otimes \Hom_\bA(X_0, X_1),
	\end{displaymath}
	
\noindent the following holds:
\begin{displaymath} 
\mu_\bB(F(a_n), \dots , F(a_1)) =F( \mu_\bA (a_{n}, \dots , a_1)).
\end{displaymath}

\noindent An $A_{\infty}$-functor $F: \mathbb{A} \rightarrow \mathbb{B}$ is said to be an \textbf{isomorphism} if the associated map $\Ob{\bA} \rightarrow \Ob{\bB}$ is bijective and all the maps $\Hom_{\mathbb{A}}(X,Y) \rightarrow \Hom_{\mathbb{B}}(F(X), F(Y))$ are isomorphisms.
	\end{definition}
	
	\begin{convention}
	All $A_{\infty}$-functors $F:\mathbb{A} \rightarrow \mathbb{B}$ in this paper are strict and \textbf{strictly unital}, i.e.\ for all objects $X \in \Ob{\mathbb{A}}$ and $F(1_X) = 1_{F(X)}$.
	\end{convention}

\subsection{Twisted complexes} \ \medskip

\noindent Let us recall the definition of the category of twisted complexes $\Tw\bA$ over an $A_{\infty}$-category $\bA$. The homotopy category of $\Tw\bA$ is triangulated \cite[Proposition 3.29]{SeiBook} and can be thought of as a generalization of the category of perfect complexes, see Remark \ref{RemarkTwistedComplexesForFiniteDimensionalAlgebras}. 
\medskip

\noindent For an object $E$ of an $A_{\infty}$-category $\bA$ and $k \in \mathbb{Z}$, we will denote by $E[k]$ its $k$-th formal \textbf{shift}.

\begin{definition}

\noindent A \textbf{twisted complex} over $\bA$ is a pair $(E, \delta)$ which consists of a formal direct sum $E = \bigoplus_{i=1}^N E_i[k_i]$ of shifted objects of $\bA$ and a strictly lower triangular differential $\delta \in \End^1(E)$.
To spell this out, $\delta$ is a collection of maps 
\begin{displaymath}
 \begin{array}{cc} \delta_{i j}\in \Hom_{\bA}^{k_j-k_i+1}(E_i,E_j), & 1\leq i< j\leq N, \end{array}
\end{displaymath}
 
 \noindent such that $\sum_{k\geq1}\mu^k(\delta,\dots,\delta)=0$. Since $\delta$ is strictly lower triangular, the sum is finite and  for all $1 \leq i < j \leq N$, we get
\begin{displaymath}
\sum_{k\geq 1}  \left( \sum_{i=i_0<i_1<\dots<i_k=j} \mu^k_\bA(\delta_{i_{k-1}i_k},\dots,\delta_{i_0i_1})\right)=0.
\end{displaymath}

\noindent A \textbf{morphism of degree $d$} between twisted complexes is a map of degree $d$ between their underlying formal direct sums. That is, if $E = \bigoplus E_i[k_i]$ and $F=\bigoplus F_j[l_j]$, then an element $f \in \Hom_{\Tw\bA}^d(E, F)$ is a collection of morphisms $f_{ij} \in \Hom_{\bA}^{d+l_j-k_i}(E_i,F_j)$.

\noindent Given twisted complexes $(E_0, \delta^0),\dots ,(E_n, \delta^n)$, $n \geq 1$, and morphisms $f_i \in
\Hom_{\Tw\bA}(E_{i-1}, E_i)$,  set
\begin{equation}\label{EquationMuTwistedComplexes}
\mu_{\Tw\bA}^n(f_n, \dots , f_1) = \displaystyle{\sum_{j_0,\dots,j_n\geq 0}}\mu_\bA^{n+j_0+\dots+j_n}(\underbrace{\delta^n,\dots,\delta^n}_{j_n},f_n,\dots,\underbrace{\delta^1, \dots,\delta^1}_{j_1},f_1,\underbrace{\delta^0,\dots,\delta^0}_{j_0}). 
\end{equation}

\noindent Here the sum is finite again, since $\delta^i$ are strictly lower triangular. The set of twisted complexes together with the  morphism spaces and structure maps defined above forms an  $A_{\infty}$-category $\Tw\bA$ which is called the \textbf{category of twisted complexes}. Note that $\mu_{\Tw \bA}^n=0$ for all $n \geq 3$ whenever $\mu^n_{\bA}=0$ for all $n \geq 3$.
\end{definition}

\begin{rem}\label{RemarkTwistedComplexesForFiniteDimensionalAlgebras}
Suppose $A$ is a basic finite-dimensional algebra and $A=\bigoplus_{i=1}^m P_i$ is a decomposition of $A$ into indecomposable projective $A$-modules. Denote by $\mathbb{A}$ the $A_\infty$-category with objects $\Ob{\mathbb{A}}=\{P_1, \dots, P_m\}$ and morphisms $\Hom_{\mathbb{A}}(P_i, P_j)=\Hom_A(P_i, P_j)$, concentrated in degree $0$. Set $\mu_{\mathbb{A}}^k=0$ for all $k\neq 2$ and $\mu_{\mathbb{A}}^2(a,b) \coloneqq a\cdot b$. Then, $\mathbb{A}$ is an $A_{\infty}$-category and $\Tw \bA$ (considered as a dg-category) is equivalent to the dg-category  of bounded complexes of finite-dimensional projective modules as can be seen by writing out \eqref{EquationMuTwistedComplexes} for $n=1$. Hence, $\H^0(\Tw\bA)$ is equivalent to the category $\Kb{A}$ of perfect complexes.
\end{rem}

\noindent Any functor $\bF: \bA \rightarrow \bB$ induces a functor $\Tw\bF: \Tw\bA \rightarrow \Tw\bB$ and the associated functor $\H^0(\Tw\bF):\H^0(\Tw\bA) \rightarrow \H^0(\Tw\bB)$ is exact.

\begin{definition}
An $A_\infty$-functor $F:\mathbb{A} \rightarrow \mathbb{B}$ is a \textbf{Morita equivalence} if the induced functor $\H^0(\Tw \mathbb{A}) \rightarrow \H^0(\Tw \mathbb{B})$ is an equivalence of triangulated categories.
\end{definition}

\begin{rem}
Let $\mathbb{B}$ be an $A_{\infty}$-category and suppose  $\mathbb{A} \subseteq \mathbb{B}$ is a full $A_{\infty}$-subcategory, i.e.\
\begin{enumerate}
    \item $\Ob{\mathbb{A}} \subseteq \Ob{\mathbb{B}}$,
    \item $\Hom_{\mathbb{A}}(X,Y)=\Hom_{\mathbb{B}}(X,Y)$ for all $X,Y \in \Ob{\mathbb{A}}$, and
    \item for all $k \geq 1$, $\mu_{\mathbb{A}}^k$ is the restriction of $\mu_{\mathbb{B}}^k$.
\end{enumerate} 

\noindent Then, the following holds:
\begin{enumerate}
    \item The category $\Tw \mathbb{A}$ is a full subcategory of $\Tw \mathbb{B}$.
    \item The inclusion $\mathbb{A} \hookrightarrow \mathbb{B}$ induces a fully faithful and exact functor $\H^0(\Tw \mathbb{A}) \hookrightarrow \H^0(\Tw \mathbb{B})$.
\end{enumerate}
\end{rem}

\subsection{Trivial extensions of $A_{\infty}$-categories} \ \medskip

\noindent For any  $A_{\infty}$-category $\mathbb{A}$ one can form a new $A_{\infty}$-category $\triv(\mathbb{A})$ 
which generalizes the construction of a trivial extension of a finite dimensional algebra by its dual bimodule. Recall that for an ordinary $\Bbbk$-algebra $A$, its trivial extension is the algebra on the vector space $A\oplus \Hom_\Bbbk(A,\Bbbk)$ whose multiplication is given by the formula $(a,f)(b,g)\coloneqq (ab, a.g+f.b)$, where ``.'' denotes the action of $A$ on the $A-A$-bimodule $\Hom_{\Bbbk}(A,\Bbbk)$.

\begin{definition}Let $V$ be a graded $\Bbbk$-vector space. Its \textbf{dual} is the graded vector space $\mathbb{D}V$ whose $i$-th homogeneous component is  $(\mathbb{D}V)^i= \Hom_{\Bbbk}(V^{-i}, \Bbbk)$.
\end{definition}

\begin{definition}\label{DefinitionInfinityTrivialExtension}
Let $\mathbb{A}$ be an $A_{\infty}$-category. Its \textbf{trivial extension} is an $A_{\infty}$-category 
$\triv(\mathbb{A})$ with the same objects as $\mathbb{A}$. For $X,Y \in \triv(\mathbb{A})$, set
$$\Hom_{\triv(\mathbb{A})}(X,Y) \coloneqq \Hom_{\mathbb{A}}(X,Y) \oplus \big(\mathbb{D} \Hom_{\mathbb{A}}(Y,X)\big).$$

The multi-linear structure maps $\mu_{\triv(\mathbb{A})}^r$ are uniquely determined by the following requirements:

\begin{enumerate}
\item For $a_r, \dots, a_1$ with $a_j \in \Hom_{\mathbb{A}}(X_{j-1}, X_{j})$,
\begin{displaymath}
 \mu_{\triv(\mathbb{A})}(a_r, \dots, a_1) = \mu_{\mathbb{A}}(a_r, \dots, a_1);
\end{displaymath}
\item For $a_{r},\dots, a_{i+1}, a_{i-1},\dots, a_1$ with $a_j \in \Hom_{\mathbb{A}}(X_{j-1}, X_{j})$ and $f \in \mathbb{D}\Hom_{\mathbb{A}}(X_i, X_{i-1})$,
\begin{equation}\label{EquationFormulaTrivialExtension}
 \mu_{\triv(\mathbb{A})}(a_r, \dots, a_{i+1}, f, a_{i-1}, \dots, a_1)(-) = (-1)^{\dagger} \cdot  f\big( \mu_{\mathbb{A}}(a_{i-1}, \dots, a_1, -, a_r, \dots, a_{i+1}) \big),
\end{equation}

\noindent where $\dagger=\sum_{j=1}^{r}\dg{a_j} +\dg{f}$.

\item For all $f \in \mathbb{D}\Hom_{\mathbb{A}}(U,V)$ and $g \in \mathbb{D}\Hom_{\mathbb{A}}(X,Y)$, the following type of expressions vanishes: 
\begin{displaymath}\mu_{\triv(\mathbb{A})}(\dots,f,\dots,g,\dots).
\end{displaymath}

\end{enumerate}
\end{definition}

\begin{rem} The definition of a trivial extension of an $A_{\infty}$-algebra by a bimodule can be found in \cite{SeidelSub,SeidelLefFib}. The sign $\dagger$ is obtained by combining the signs for the bimodule structure on $\bA$ from \cite{SeidelLefFib}; the definition of a dual bimodule from \cite{Tradler,Mescher} (the sign convention in \cite{Mescher} agrees with ours); and the trivial extension from \cite{SeidelSub}. Note that in \cite{SeidelSub} and \cite{Mescher} the order of the entries in $\mu_{\triv(\mathbb{A})}$ are reversed, however, the signs are not affected.
\end{rem}

\noindent The following is an immediate consequence of Definition \ref{DefinitionInfinityTrivialExtension}.

\begin{lem}\label{LemmaInclusionsTrivialExtension}
There exists a canonical (strict) $A_{\infty}$-functor $\mathbb{A} \rightarrow \triv(\mathbb{A})$. Moreover, if $\mathbb{B} \subseteq \mathbb{A}$ is a full $A_{\infty}$-subcategory, then $\triv(\mathbb{B})$ is a full $A_\infty$-subcategory of $\triv(\mathbb{A})$.
\end{lem}

\begin{lem}
Suppose that  $\mu_{\mathbb{A}}^i=0$ for all $i \neq 2$. Then, $\mu_{\triv(\mathbb{A})}^i=0$ for all $i \neq 2$. In particular,

\begin{enumerate}
    \item if $A$ is a finite dimensional algebra with trivial extension $T$, then $\triv(A) \cong T$ as $A_{\infty}$-categories with one object;
    \item  if $\mathbb{A}$ denotes the associated categories of indecomposable projective $A$-modules as in Remark \ref{RemarkTwistedComplexesForFiniteDimensionalAlgebras}, then
     \begin{displaymath}
\H^0\big(\!\Tw \triv(\bA)\big) \simeq \Kb{T}.
\end{displaymath}
\end{enumerate}
\end{lem}
\begin{proof}It follows from \eqref{EquationFormulaTrivialExtension} that $\mu_{\triv(\mathbb{A})}^i=0$ for all $i \neq 2$. Suppose $\mathbb{A}$ is a finite dimensional algebra $A$, considered as an $A_{\infty}$-category with one object. In particular, every morphism in $\mathbb{A}$ and  $\mathbb{D}A$ is of degree zero. Moreover, by bilinearity
\begin{equation*}
\mu_{\triv(A)}\big( (a_2, f_2), (a_1, f_1) \big) = ( a_2 \cdot a_1,  a_2.f_1 + f_2.a_1).
\end{equation*}

\noindent This proves the first assertion. Now suppose that $\mathbb{A}$ is the category formed by indecomposable projective $A$-modules of a finite dimensional algebra $A$. The second assertion follows from the fact that  $\triv(\mathbb{A})$ coincides with the $A_{\infty}$-category of indecomposable projective $T$-modules as defined in Remark \ref{RemarkTwistedComplexesForFiniteDimensionalAlgebras}.
\end{proof}

\section{Partially wrapped Fukaya categories of marked surfaces}\label{SectionFukayaCategories}

\noindent We recall the definition of the partially wrapped Fukaya category of a graded marked surface  due to Auroux, Bocklandt and Haiden-Katzarkov-Kontsevich \cite{Auroux, BocklandtPuncturedSurface, HaidenKatzarkovKontsevich}. In the exposition we closely follow \cite{HaidenKatzarkovKontsevich}. Sections \ref{SectionDecoratedSurfaces}, \ref{SectionLineFields} and \ref{SectionWinding} contain necessary preliminaries on line fields, gradings on arcs, degree of intersections and winding numbers. We further prove some properties of winding numbers that we are going to use later.

\subsection{Decorated surfaces and arc systems} \label{SectionDecoratedSurfaces}

\begin{definition}
	Let $\Sigma$ be a (smooth) connected, compact oriented surface with non-empty boundary $\partial \Sigma$ and let $\deco \subseteq \Sigma$ be a finite non-empty subset of $\Sigma$.
The pair $(\Sigma, \deco)$ is called a \textbf{decorated surface}. Interior points of $\deco$ are called \textbf{punctures} and the set of all punctures is denoted by $\mathscr{P}$. Boundary points in $\deco$  are called \textbf{marked points} and their set is denoted by $\marked$.
\end{definition}

\begin{definition}\label{DefinitionCurvesHomotopies}A \textbf{curve} on a decorated surface $(\Sigma, \deco)$ is a smooth immersion $\gamma:\Omega \rightarrow \Sigma$, where $\Omega \in \{[0,1], S^1\}$, such that $\gamma^{-1}(\deco \cup \partial \Sigma)=\partial \Omega$ and $\gamma(\partial \Omega) \subseteq \deco$. If $\Omega=[0,1]$, then $\gamma$ is called an \textbf{arc} and if $\Omega=S^1$, then $\gamma$ is called a  \textbf{loop}. A \textbf{homotopy} between curves $\gamma_0$ and $\gamma_1$ is a smooth map $H:[0,1] \times \Omega \rightarrow \Sigma$ such that for every $t \in [0,1]$, $H(t, -): \Omega \rightarrow \Sigma$ is a curve and $H(0,-)=\gamma_0$, $H(1,-)=\gamma_1$. A curve $\gamma: \Omega \rightarrow \Sigma$ is \textbf{simple} if $\gamma|_{\Omega \setminus \partial \Omega}$ is an embedding. 
\end{definition}

\noindent In particular, by definition homotopies between loops are free homotopies and any homotopy between arcs is necessarily constant on the end points since the set $\deco$ is finite. Note that our notion of homotopy is commonly referred to as \textit{regular homotopy} and hence every homotopy induces a homotopy between the derivatives of the corresponding curves. Note also that two curves without closed contractible subpaths (also called ``tear drops'') are homotopic in the sense of Definition \ref{DefinitionCurvesHomotopies} if and only if they are homotopic as smooth maps in the usual sense, see \cite{Chillingworth}. Usually, we drop $\deco$ from the notation of a decorated surface. The respective meaning of the symbols $\deco, \punct$ and $\marked$ should always be apparent from the context.

\begin{definition} A \textbf{boundary arc} is an arc which is homotopic (as a map $[0,1] \rightarrow \Sigma$) relative to its end points to a boundary segment between two consecutive marked points.
\end{definition}

\begin{definition}
	An \textbf{arc system} on $\Sigma$ is a finite set $\mathcal{A}=\{\gamma_1, \dots, \gamma_n\}$ of simple arcs in $\Sigma$ such that for all $i \neq j$, $\gamma_i$ and $\gamma_j$ are non-homotopic and only intersect transversely in $\deco$.
\end{definition}
\noindent For any arc system $\mathcal{A}$ of $\Sigma$ we denote by $|\mathcal{A}| \subseteq \Sigma$ the union of all its arcs. In what follows, we always assume that either $\mathscr{P}=\deco$ or $\mathscr{P}=\emptyset$. In the former case, $\Sigma$ is called \textbf{punctured}. If $\deco=\marked$, $\Sigma$ is called a \textbf{marked surface}. A point $x \in \Sigma \setminus \deco$ is called \textbf{regular}. We denote by $\Sigma_{\reg}\coloneqq \Sigma \setminus \deco$ the set of regular points.
\begin{definition}\label{DefinitionFullAdmissible}
	An arc system $\mathcal{A}$ is called \textbf{full} if $\Sigma \setminus |\mathcal{A}|$ is a disjoint union of discs and annuli which contain exactly one boundary component $B$ of $\Sigma$ such that $B\cap \deco=\emptyset$. 
\end{definition}
\noindent An inclusion-maximal arc system on a marked  surface is called a \textbf{triangulation}.
Note that every triangulation necessarily contains a representative from every homotopy class of boundary arcs.

\noindent Every full arc system $\mathcal{A}$ on a punctured surface defines a ribbon graph $\Gamma=\Gamma_{\cA}$ whose ribbon surface $\Sigma_{\Gamma}$ is naturally embedded into $\Sigma$. Note however that $\Sigma_{\Gamma}$ and $\Sigma$ are not homeomorphic in general as $\Sigma_{\Gamma}$ contains additional boundary components which correspond to  ``inner discs'' on $\Sigma$, i.e.\ components of $\Sigma \setminus |\cA|$ which are not bordered by any boundary component of $\Sigma$. A face of $\Gamma_\cA$ will be called a \textbf{false face} if it corresponds to such an additional component of $\partial \Sigma_{\Gamma}$.

\subsection{Line fields and graded arcs} \label{SectionLineFields} \ \medskip

\noindent We recall the notion of a \textit{line field} and how it can be used to equip curves with an additional structure of a \textit{grading}.

\subsubsection{Line fields}

\begin{definition}Let $(\Sigma, \deco)$ be a decorated surface. A \textbf{line field} is a smooth section $\eta: \Sigma_{\reg} \rightarrow \mathbb{P}(T\Sigma)$, where $\mathbb{P}(T\Sigma)$ denotes the projectivized tangent bundle of $\Sigma$. The triple $(\Sigma, \deco, \eta)$ is called a \textbf{graded surface}. Two line fields $\eta_0$ and $\eta_1$ are called \textbf{homotopic} ($\eta_0\simeq \eta_1$) if they are homotopic as maps $\Sigma_{\reg} \rightarrow \mathbb{P}(T\Sigma)$.
\end{definition}

\noindent  The above definition allows line fields to have singularities at marked points and punctures.  Every nowhere vanishing vector field which is defined on $\Sigma_{\reg}$ gives rise to a line field  by virtue of the projection $T(\Sigma) \rightarrow \mathbb{P}(T\Sigma)$. The converse, however, does not hold and if a line field $\eta$ arises from a vector field in this way, we say that $\eta$ is \textbf{orientable}.

\begin{exa}\label{ExampleLineFieldRibbonGraph}
Every ribbon graph $\Gamma$ gives rise to a homotopy class of line fields on its ribbon surface $\Sigma=\Sigma_{\Gamma}$. We regard $\Gamma$ as a subset of $\Sigma$ as usual. Let $\Delta=\{\delta_e \, | \, e \in \BE{\Gamma}\}$ be a set of disjoint embedded paths with endpoints on $\partial \Sigma$ such that for each $e \in \BE{\Gamma}$, $\delta_e$ crosses $e$ exactly once and transversely and does not cross any other edge of $\Gamma$. Moreover, $\delta_e$ intersects $\partial \Sigma$ transversely and only at its end points.  The collection $\Delta$ cuts $\Sigma$ into polygons which are bounded by paths in $\Delta$ and segments of $\partial \Sigma$. Each polygon $P$ contains a unique vertex $p$ of $\Gamma$ and admits a line field $\eta_P: P \setminus \{p\} \rightarrow \mathbb{P}(T\Sigma)$ whose foliation is depicted in Figure \ref{FigureLineFieldRibbonGraph}. In particular, $\eta_P$ is parallel to the  paths in $\Delta$ at the boundary of $P$ by which we mean that $\eta_P$ and $\dot \delta_e$ coincide as paths in $\mathbb{P}(T\Sigma)$. Moreover, every edge of $\Gamma$ is everywhere transverse to $\eta_{\Gamma}$.
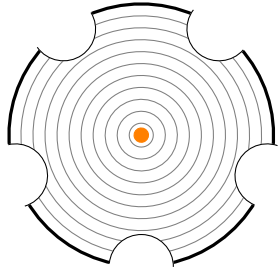
\begin{figure}[H]
\begin{tikzpicture}[scale=1.5, rotate around={18: (0,0)}]

\def\nl{10}; 
\def\r{3pt}; 
\foreach \i in {1, ..., \nl}
{
\draw[gray] (0,0) circle ({\r*\i});

}
\filldraw[orange] (0,0) circle (0.6*\r);
\def\nd{5}; 
\def\dist{\r*(\nl+1)}; 
\def\angle{360/\nd}; 
\def\length{240/\nd}; 

\foreach \i in {1,..., \nd}
{
\draw[very thick] ({\dist * cos(\angle*\i-\length*0.5)},{\dist*sin(\angle*\i-\length*0.5)}) arc ({\angle*\i-\length*0.5}:{\angle*\i+\length*0.5}:{\dist}); 
}
\def\angb{180/\nd};
\def\rtwo{8pt};
\foreach \i in {1,..., \nd}
{
\filldraw[white] ({\dist* cos(\angb * (2*\i+1)},{\dist* sin(\angb* (2*\i+1)}) circle (\rtwo);
\def\startangle{180+\angle*(\i+1)-\length*0.5-105};
\draw[black] ({\dist* cos(\angb * (2*\i+1))+\rtwo*cos(\startangle+2)},{\dist* sin(\angb* (2*\i+1))+\rtwo*sin(\startangle+2)}) arc ({\startangle+2}:{\startangle+180-2}:\rtwo);
}
\end{tikzpicture}
    \caption{The foliation of $\eta_P$ with a singularity at a vertex of $\Gamma$ in the center. Half circles in the boundary represent segments of $\partial \Sigma$.}
    \label{FigureLineFieldRibbonGraph}
\end{figure}
\noindent Since each $\eta_P$ is parallel to the paths in $\Delta$, the line fields $\eta_P$ can be glued to a line field $\eta_{\Gamma}: \Sigma_{\reg} \rightarrow \mathbb{P}(T\Sigma)$.
Lemma \ref{LemmaBipartiteOrientable} shows that $\eta_{\Gamma}$ is orientable if and only if $\Gamma$ is bipartite.
\end{exa}

\noindent Line fields as in Example \ref{ExampleLineFieldRibbonGraph} will be of special importance, which motivates the following definition.

\begin{definition}\label{DefinitionRibbonTypeLineField}
Suppose that $\Sigma$ is a punctured surface. A line field $\eta$ on $\Sigma$ is said to be of \textbf{ribbon type} if there exist a ribbon graph $\Gamma$ and a filling embedding $\kappa: \Gamma \rightarrow \Sigma$ which maps $\BV{\Gamma}$ bijectively onto $\punct$ such that $\eta \simeq \eta_{\Gamma}$ and such that for each $p \in \punct$ there exists a local coordinate chart around $p$ in which all leaves of $\eta$ are loops which encircle $p$. In this case, the graded surface $(\Sigma, \punct,\eta)$ is also said to be of \textbf{ribbon type}.
\end{definition}

\subsubsection{Gradings} \ \medskip

\noindent The choice of a line field allows us to endow arcs with the additional structure of a grading. In what follows $(\Sigma, \deco, \eta)$ shall denote a graded surface.

\begin{definition}A \textbf{path} in $(\Sigma, \deco, \eta)$ is an immersion  $\delta:I \rightarrow \Sigma$, where $I \subseteq \mathbb{R}$ is any (not necessarily compact) interval such that $\delta^{-1}(\deco) \subset \partial I$. Define $J \coloneqq I \setminus \delta^{-1}(\deco)=\delta^{-1}(\Sigma_{\reg})$ and $\delta^{\ast}\eta\coloneqq \eta \circ \delta|_J$, i.e.\ $\delta^{\ast}\eta$ is a path in $\mathbb{P}(T\Sigma)$.
\end{definition}  
\begin{definition}\label{DefinitionGradingArcs}
Let $(\Sigma, \eta)$ be a graded surface and let $\gamma:I \rightarrow \Sigma$ be a path. A \textbf{grading} of $\gamma$ is a homotopy class $\mathfrak{g}$ of homotopies from $\gamma^{\ast}\eta$ to $\dot \gamma|_{J}$. By the $n$-th shift $\mathfrak{g}[n]$ of a grading $\mathfrak{g}$, where $n\in \mathbb{Z}$, we mean the homotopy class of  $\mathfrak{g}[n]:[0,1]\times J \rightarrow \mathbb{P}(T\Sigma)$, such that for every $t \in J$, the homotopy class of the concatenated path $\mathfrak{g}[n]|_{[0,1]\times \{t\}} \ast \bar{\mathfrak{g}}|_{[0,1]\times \{t\}}$ is identified with  $n\in \mathbb{Z}$ under the isomorphism $\mathbb{Z} \cong \pi_1(S^1, 1) \cong \pi_1(\mathbb{P}(T\Sigma)_{\gamma(t)}, \eta(\gamma(t)))$, which identifies $1$ with a counter-clockwise loop. The pair $(\gamma, \mathfrak{g})$ is called a \textbf{graded path} and a \textbf{graded arc} if $\gamma$ is an arc.
\end{definition}

\noindent Most of the time we will not distinguish between a grading and a representative of it.

\begin{definition}
Let $\gamma_1, \gamma_2$ be arcs on $\Sigma$ and denote by $\Omega_i$ the domain of $\gamma_i$. An \textbf{oriented intersection} from $\gamma_1$ to $\gamma_2$ is a pair $(t_1, t_2) \in \Omega_1 \times \Omega_2$ such that  $p \coloneqq \gamma_1(t_1)=\gamma_2(t_2)$ is a transverse intersection of $\gamma_1$ and $\gamma_2$;  additionally, if $p \in \partial \Sigma$, we require that locally around $p$, $\gamma_1$ and $\gamma_2$ are arranged as in Figure \ref{FigureIntersection}. The set of oriented intersections from $\gamma_1$ to $\gamma_2$ is denoted by $\gamma_1 \oInt \gamma_2$. If $\alpha, \beta, \gamma$ are arcs and $p=(r, s) \in \alpha \oInt \beta$ and $q=(s,t) \in \beta \oInt \gamma$, then $(r,t) \in \alpha \oInt \gamma$ and we set $q p \coloneqq (r,t)$.
\end{definition}
 
	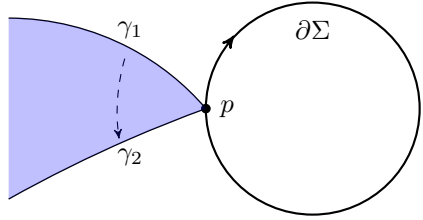
\begin{figure}[h]
		\begin{displaymath}
		\begin{tikzpicture}[scale=0.8, hobby]
		\draw[thick, 
		decoration={markings, mark=at position 0.4 with {\arrow{<}}},
		postaction={decorate}
		](0,0) circle (50pt);
		\filldraw (-50pt, 0) circle (2pt);
		\draw (-40pt,0) node{$p$};
		\draw (0,1.3) node {$\partial \Sigma$};
		\draw [line width=0.5, color=black] plot  [ tension=1] coordinates {  (-50pt,0) (-3, 1) (-5, 1.5)  };
		\draw [line width=0.5, color=black] plot  [ tension=1] coordinates {  (-50pt,0) (-3, -0.5) (-5, -1.5)  };
		\draw (-3,1.3) node{$\gamma_1$};
		\draw (-3,-0.8) node{$\gamma_2$};
		
		\draw[dashed, ->] ({3.2*cos(165)},{3.2*sin(165)}) arc (164:188:3.2);
		\begin{scope}
		\path[clip] plot  [ tension=1] coordinates {  (-50pt,0) (-3, 1) (-5, 1.5)  }-- plot  [ tension=1] coordinates { (-5, -1.5) (-3, -0.5)  (-50pt,0)   };
		\filldraw[blue, fill opacity=0.25] (0,0) circle (10);
		\end{scope}
		\end{tikzpicture}
		\end{displaymath}
		\caption{An oriented boundary intersection $p$ from $\gamma_1$ to $\gamma_2$ and its sector (colored).} 
    \label{FigureIntersection}
	\end{figure}  
\noindent Frequently, we do not distinguish between an oriented intersection and its corresponding image in the surface.  Note that any interior intersection of curves $\gamma_1$ and $\gamma_2$, for instance at a puncture, gives rise to a pair $(p, p^{\ast})$ with $p \in \gamma_1 \oInt \gamma_2$ and $p^{\ast} \in \gamma_2 \oInt \gamma_1$ and if $p=(t_1,t_2)$, then $p^\ast=(t_2, t_1)$. On the other hand, every intersection at the boundary gives rise to an element in exactly one of the two sets of oriented intersections.

Locally every oriented intersection $p \in \gamma_1 \oInt \gamma_2$ at a point of $\deco$ corresponds to a unique sector which is bounded by a segment of $\gamma_1$ and $\gamma_2$. More precisely, if $\gamma_1$ and $\gamma_2$ are oriented so that $p$ is their start point, then the sector lies to the left of $\gamma_1$ and the right of $\gamma_2$.\medskip

\subsubsection{The degree of an intersection}\label{SectionDegreeIntersection} \ \medskip

\noindent Next, suppose that $(\gamma_1,\mathfrak{g}_1)$ and $(\gamma_2, \mathfrak{g}_2)$ are graded paths on a graded surface $(\Sigma, \eta)$ and suppose that $(t_1,t_2)=p \in \gamma_1 \oInt \gamma_2$ is an oriented intersection such that $\gamma_1$ and $\gamma_2$ intersect in $\gamma_1(t_1)=\gamma_2(t_2)$. We attach an integer to $p$ as follows. First, we assume that $p \not \in \deco$, i.e.\ $\gamma_1(t_1)=\gamma_2(t_2) \not \in \deco$.\medskip

\noindent Then we have the following homotopy classes of paths in $\mathbb{P}(T\Sigma)_{p}$, where as usual we identified $p$ with the point $\gamma_1(t_1)=\gamma_2(t_2)$.

\begin{itemize}
    
    \item $\mathfrak{g}_1(t_1)$ which is a path in $\mathbb{P}(T\Sigma)_{p}$ from $\eta(p)$ to $\dot \gamma_1(t_1)$ and $\overline{\mathfrak{g}_2}(t_2)$ which is a path from $\dot \gamma_2(t_2)$ to $\eta(p)$;
    
    \item an embedded counter-clockwise path $\varepsilon$ from $\dot \gamma_1(t_1)$ to $\dot \gamma_2(t_2)$. Since $\dot \gamma_1(t_1) \neq \dot \gamma_2(t_2)$, the assumptions on $\varepsilon$ guarantee its uniqueness up to homotopy.
    
\end{itemize}

\noindent The concatenation $\delta \coloneqq \mathfrak{g}_1 \ast \varepsilon \ast  \overline{\mathfrak{g}_2}$ is closed and defines an element in $\pi_1(\mathbb{P}(T\Sigma)_{p}, \eta(p))$. Hence $\delta$ defines an integer $\deg(p)$ by the isomorphism  $\pi_1(\mathbb{P}(T\Sigma)_{p}, \eta(p)) \cong \pi_1(S^1, 1) \cong \mathbb{Z}$. One observes that 
\begin{displaymath}
\deg p + \deg p^\ast=1,\end{displaymath}
\noindent where $p^\ast=(t_2, t_1) \in \gamma_2 \oInt \gamma_1$.

If $p \in \deco$, let $u:[0,1] \rightarrow \Sigma$ be an embedded path which lies in the sector corresponding to $p$ and which intersects $\gamma_1$ and $\gamma_2$ exactly once and transversely at its end points. More precisely, we assume that there exist $s_1, s_2 \in (0,1)$ such that $u(0)=\gamma_1(s_1) \in \Sigma \setminus \deco$ and $u(1)=\gamma_2(s_2) \in \Sigma \setminus \deco$.   After choosing a grading on $u$, the oriented intersections $q_1=(s_1, 0) \in \gamma_1 \oInt u$ and $q_2=(s_2, 1) \in \gamma_2 \oInt u$ give rise to an integer
\begin{displaymath}
\deg(p) \coloneqq \deg(q_1) - \deg(q_2) = \deg(q_1) + \deg(q_2^{\ast})-1, 
\end{displaymath}
\noindent which is independent of $u$ and the chosen grading on $u$.  The integer $\deg(p)$ is referred to as the \textbf{degree} of $p$.

\subsection{Winding numbers}\label{SectionWinding} \ \medskip

\noindent The definition of the winding number is somewhat technical and the reader may skip Definition \ref{DefinitionWindingNumbers} on a first read to continue with Remark \ref{RemarkIntuivitveWindingNumbers} which contains a more direct description.

\begin{definition} \label{DefinitionWindingNumbers}Let $\eta$ be a line field on a decorated surface $\Sigma$ and let $\gamma \subseteq \Sigma$ be a loop. The \textbf{winding number of $\gamma$} is the intersection number of its derivative $\dot \gamma: S^1 \rightarrow \mathbb{P}(T\Sigma)$ with $\eta(\Sigma_{\reg}) \subseteq \mathbb{P}(T\Sigma)$. We denote by $\omega_{\eta}$ the function which associates to a loop $\gamma$ its winding number.
\end{definition}

\noindent By construction, the winding number of a loop $\gamma$ is the image of $\dot \gamma$ under the homomorphism $H_1(\mathbb{P}(T\Sigma), \mathbb{Z}) \rightarrow \mathbb{Z}$ determined by the intersection product with $\eta(\Sigma_{\reg})$.  We stress that therefore the winding number of a loop is an invariant of its homotopy class in the sense of Definition \ref{DefinitionCurvesHomotopies} but is not invariant under arbitrary smooth homotopies.

\begin{rem}\label{RemarkIntuivitveWindingNumbers} A more explicit description of the winding number may be given as follows. Suppose $\gamma: S^1 \rightarrow \Sigma$ is a loop. Then, at each point $t_0 \in S^1$ there are two distinguished points in the fiber $\mathbb{P}(T\Sigma)_{\gamma(t_0)} \cong S^1$  at $z\coloneqq \gamma(t_0)$. The first one is the line $\eta(z)$ and the second is the tangent line $\dot \gamma(t_0)$ of $\gamma$ at $z$. Locally, we find coordinates so that $c \coloneqq\dot \gamma(t) \in S^1$ is constant for all $t \in S^1$ close enough to $t_0$. As one follows the loop $\gamma$, the lines $\eta(\gamma(t))$ draw a path $\delta$ in $S^1$. One may count the intersection number of $\delta$ and the point $c$ in $S^1$ which determines an integer. In other words, $\omega_{\eta}$ counts with sign how often the tangent line of $\gamma$ agrees with the line which is determined by $\eta$. Figure \ref{FigureWindingNumbers} depicts possible contributions to the winding number of $\gamma$.
\end{rem}
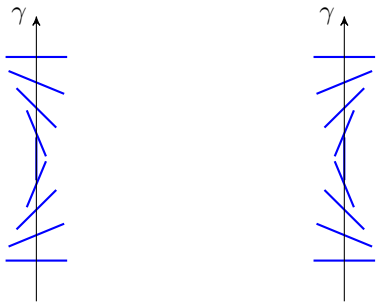
\begin{figure}[H]
	\begin{displaymath}
	\begin{tikzpicture}
		
		\begin{scope}[scale=2.7]	

		\def\steps{8}; 
		\def\l{0.15};

		\foreach \i in {0,1,...,\steps}
		{	
	\pgfmathsetmacro{\ll}{0.7*\l+ 0.3*\l*abs(cos((1/\steps)*\i*180))};
			\draw[blue, thick] ({\ll*cos( (1/\steps)*\i*180)},{(1/\steps)*\i+\ll*sin( (1/\steps)*\i*180))})--({\ll*cos( (1/\steps)*\i*180+180)},{(1/\steps)*\i+\ll*sin( (1/\steps)*\i*180+180))});	
				
		}

				\draw[ ->] (0,-0.2)--(0,1.2) node[left]{$\gamma$};
		\end{scope}
		
	\begin{scope}[scale=2.7, shift={(1.5,0)}]	

	\def\steps{8}; 
	\def\l{0.15};

	\foreach \i in {0,1,...,\steps}
	{	
		\pgfmathsetmacro{\ll}{0.7*\l+ 0.3*\l*abs(cos((1/\steps)*\i*180))};
		\draw[blue, thick] ({\ll*cos( -(1/\steps)*\i*180)},{(1/\steps)*\i+\ll*sin( -(1/\steps)*\i*180))})--({\ll*cos(- (1/\steps)*\i*180+180)},{(1/\steps)*\i+\ll*sin( -(1/\steps)*\i*180+180))});	
		
	}
	
	\draw[->] (0,-0.2)--(0,1.2) node[left]{$\gamma$};
	\end{scope}

	\end{tikzpicture}
	\end{displaymath}
	\caption[Contributions of a segment of $\gamma$ to its winding number with respect to a line field]{Contributions of a segment of $\gamma$ to its winding number with respect to a line field   \! \begin{tikzpicture}[scale=2, baseline=-3.3ex]\begin{scope} 
			\draw[blue, ultra thick] (0,-0.2)--(0.25,-0.2);
			\end{scope}\end{tikzpicture}. The contribution is $-1$ on the left and $+1$ on the right.} \label{FigureWindingNumbers}
\end{figure}

\noindent By counting signed intersections of its derivative with the line field we can further define the \textbf{winding number of an arc} $\delta:[0,1]\rightarrow \Sigma$ which is transverse to a line field $\eta$ in some neighbourhood of its endpoints. It is invariant under homotopies which satisfy the additional assumption that all intermediate curves are nowhere parallel to $\eta$ in a neighbourhood of their end points. A short proof of this fact is given in Remark \ref{RemarkWindingNumberWellDefined}. Note that if $\eta$ is of ribbon type, then every arc has a winding number and every homotopy between arcs automatically satisfy the previous assumption. Indeed, we identify $\eta$ locally with a line field on the unit disc whose foliation consists of concentric circles with center $0$ and observe that (identifying $\delta$ locally with 
a path $[0,\varepsilon) \rightarrow \mathbb{R}^2$, $\delta(0)=0$)
\begin{displaymath}
\frac{d}{dt} \dg{\delta(t)}^2= 2 < \delta(t), \dot \delta(t)>
\end{displaymath}
\noindent is non-zero for all $t \neq 0$ in a neighbourhood of $0$. Here, $<-,->$ denotes the usual scalar product. Thus, $\dot \delta(t)$ is never tangent to the circle through $\delta(t)$ whenever $t$ is close to $0$. On the other hand, it is not difficult to see that the winding number of an arc is not invariant under homotopies of the line field.

\begin{rem}\label{RemarkWindingNumberWellDefined}
Suppose that $H: [0,1] \times [0,1] \rightarrow \Sigma$ is a homotopy between arcs such that for all $t \in [0,1]$, the arc $\gamma_t=H|_{\{t\} \times [0,1]}$ is nowhere parallel to $\eta$ in a neighbourhood of $\{0,1\}$. Let $\hat{H}: [0,1]^2 \rightarrow \mathbb{P}(T\Sigma)$ denote the associated derivative lift of $H$. By a small perturbation of $\hat{H}$ which leaves $\{0,1\} \times [0,1] \cup [0,1] \times \{0,1\}$ fixed we may assume that $\hat{H}$ and $\eta(\Sigma_{\reg})$ are transversal maps. The assumptions on $H$ imply that there exists a closed neighbourhood $U$ of $\punct \cup \partial \Sigma \subseteq \Sigma$ which is a disjoint collection of discs and annuli such that $\hat{H}([0,1]^2) \cap \eta(\Sigma_{\reg}) \subseteq W\coloneqq \eta(\Sigma_{\reg} \setminus U)$.  By compactness and the fact $H$ is a homotopy of arcs, there exists an open interval $I \subseteq [0,1]$ such that $X \coloneqq \hat{H}^{-1}(W)=\hat{H}^{-1}(\eta(\Sigma_{\reg})) \subseteq [0,1] \times I$. 
Since $X=\big(\hat{H}|_{[0,1] \times I}\big)^{-1}(W)$, $X$ is a $1$-dimensional manifold with boundary $\partial X = X \cap \big(\{0,1\} \times I\big)$. Replacing $U$ by a suitable smaller open neighbourhood of $\punct \cup \partial \Sigma$, we find that $X$ is closed in $[0,1]^2$ and hence compact showing that $X$ is a disjoint union of circles and intervals with end points on $\{0,1\} \times [0,1]$. This shows that the intersection numbers of $\dot \gamma_0$ and $\dot \gamma_1$ with $\eta(\Sigma_{\reg})$ agree.
\end{rem}

\begin{exa}\label{ExContraLoop}
The winding number of a counter-clockwise simple contractible loop is $2$. 
\end{exa}
\noindent Every boundary component $B \subseteq \partial \Sigma$  with its induced orientation gives rise to a simple loop whose winding number we denote by $\omega_\eta(B)$. Similarly, for every puncture $p \in \mathscr{P}$, we denote by $\omega_\eta(p)$ the winding number of any simple clockwise loop around $p$. 

\begin{exa}\label{ExampleWindingNumbersBoundaries} Let $\Gamma$ be a ribbon graph and let $(\Sigma, \eta)=(\Sigma_{\Gamma},\eta_{\Gamma})$ denote the associated graded punctured surface. Then, $\omega_{\eta}(p)=0$ for all punctures $p \in \mathscr{P}$ and for every component $B \subseteq \partial \Sigma$, 
\begin{equation}\label{EquationWindingNumberBoundary}\omega_{\eta}(B) = -l\end{equation}
 \noindent where $l$ denotes the perimeter of the corresponding face $F_B \in \BF{\Gamma}$, c.f.\ Figure \ref{FigureWindingNumberBoundary}.
 
 \begin{figure}[H]
  \begin{tikzpicture}
  
  \begin{scope}[rotate=18, scale=1.3]
  
  \def\r{1.5};
  \def\n{5};
  \def\angle{360 / \n};
  \def\factor{1.5};
 \draw[orange,very thick, opacity=100] ({1*\angle}:\r) \foreach \i in {2,..., \n} {
            -- ({\i*\angle}:\r)
        } -- cycle;
\foreach \i in {1, ..., \n} \filldraw[orange] ({\i*\angle}:\r) circle (3pt);
    \draw [orange, very thick] ({\angle}:\r)--({\angle}:{0.2});
     \filldraw[orange] ({\angle}:{0.2}) circle (3pt);

\begin{scope}        
   
\path[rounded corners, clip] ({1*\angle}:{\r*1.3}) \foreach \i in {2,..., \n} { -- ({\i*\angle}:{\r*1.3})
         } -- cycle;
         
\def\fac{0.8};
\def\dis{0.19}; 
\coordinate (D) at ({\dis-\dis*cos(72)},{-\dis*sin(72)});
\coordinate (E) at ({\dis*cos(144)-\dis*cos(72)},{\dis*sin(144)-\dis*sin(72)});

\coordinate (F) at ($ ({\angle}:{\fac*\r})+  (D) $);
\coordinate (G) at ($ ({\angle}:{0.15})+  (D) $);
\coordinate (H) at ($ ({\angle}:{0.15})+  (E) $);
\coordinate (J) at ($ ({\angle}:{\fac*\r})+  (E) $);

\draw[rounded corners, very thick, decoration={markings, mark=at position 0.2375 with {\arrow{<}}},
		postaction={decorate}] ({2*\angle}:{\r*\fac}) \foreach \i in {3,...,5} { --({\i*\angle}:{\r*\fac})} --(F)--(G)--(H)--(J)--cycle;

\begin{pgfinterruptboundingbox}
 \path[invclip] ({2*\angle}:{\r*\fac}) \foreach \i in {3,...,5} { --({\i*\angle}:{\r*\fac})}--(F)--(G)--(H)--(J)--cycle;
 \end{pgfinterruptboundingbox}

 \def\ct{9};
     \def\radius{0.75};
     \def\st{\radius / \ct};         
  \foreach \i in {1, ..., \n}
  {
     \begin{scope}[shift={({\i*\angle:\r})}]
    
     \foreach \j in {2,...,\ct}
     \draw[blue, opacity=0.65] (0,0) circle ({\st*\j});
     \end{scope}
  }      
  
  \begin{scope}

   \begin{scope}[shift={({\angle:0.2})}]
    
     \foreach \j in {2,...,5}
     \draw[blue, opacity=0.65] (0,0) circle ({\st*\j});
     \end{scope}
     
    \end{scope}

 \end{scope}

  \end{scope}
  \end{tikzpicture}
  \caption{A face of length $7$ in a ribbon graph and its line field $\eta_{\Gamma}$. Each corner of the face contributes $-1$  to the winding number of the associated boundary component.}
  \label{FigureWindingNumberBoundary}
 \end{figure}
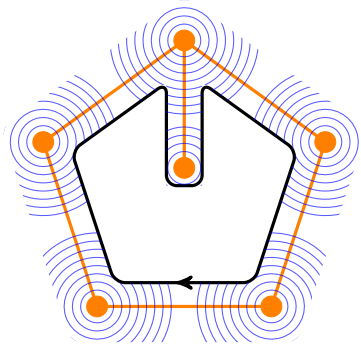
 
\end{exa}

\begin{rem}
If $p \in \gamma_1 \oInt \gamma_2$ is an oriented intersection at a puncture, then $\deg(p) + \deg(p^{\ast}) = \omega_\eta(p)$. 
\end{rem}

\noindent The Poincar\'e-Hopf theorem states that the winding numbers $\omega_\eta(B)$ and $\omega_\eta(p)$ are related to each other by the topology of $\Sigma$. It reads as follows.

\begin{thm}[Poincar\'e-Hopf Theorem] Let $\eta:\Sigma \rightarrow \mathbb{P}(T\Sigma)$ be a line field on an oriented compact surface $\Sigma$. Then,
$$ 
    \sum_{B \subseteq \partial\Sigma}{\omega_\eta(B)} = 2 \chi(\Sigma).
$$

\noindent In particular, if $(\Sigma, \punct)$ is a punctured surface and $\eta$ is a line field defined on $\Sigma_{\reg}$, then 
\begin{displaymath}
\sum_{B \subseteq \partial\Sigma}{\omega_\eta(B)} + \sum_{p \in \mathscr{P}}{\omega_\eta(p)} = 2 \chi\left(\Sigma \setminus \punct\right).
\end{displaymath}
\end{thm}

\subsubsection{Cycles of oriented intersections and their degrees} \ \medskip

\noindent Let  $(\Sigma, \eta)$ be a graded punctured surface and let $\cA$ be an $\eta$-graded arc system on $\Sigma$. Suppose that $(\gamma_i)_{i \in \mathbb{Z}_m}$ is a sequence of arcs in $\cA$ and suppose that $(p_i)_{i \in \mathbb{Z}_m}$ is a cycle of oriented intersections, by which we mean that for every $i \in \mathbb{Z}_m$, $p_i \in \gamma_i \oInt \gamma_{i+1}$ or $p_i \in \gamma_{i+1} \oInt \gamma_i$. The concatenation $\gamma$ of all arcs $\gamma_i$  (following the cyclic order) is an oriented piecewise smooth loop. There are several and non-equivalent ways to smooth out $\gamma$ to a loop since every $p_i$ is a puncture. However, since each $p_i$ is oriented they determine a canonical homotopy class of loops which agree with $\gamma$ outside a small neighbourhood of the intersections. Locally around $p_i$ such a smoothing is given as follows:
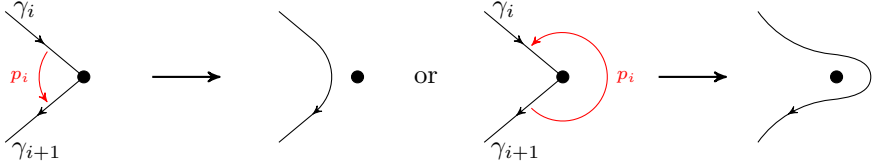
\begin{displaymath}
 \begin{tikzpicture}[scale=0.9]
 
 \begin{scope}
 \begin{scope}[el/.style = {inner sep=2pt, align=left}]
 \def\r{1.5};
 \def\ang{-5};
 \coordinate (A) at ({135-\ang}:{\r});
  \coordinate (B) at ({225+\ang}:{\r});
    	\filldraw (0,0) circle (2.5pt);
\draw[decoration={markings, mark=at position 0.5 with {\arrow{>}}},
		postaction={decorate}] (A)--(0,0) node[above, near start, anchor=south]{$\gamma_i$};
\draw[decoration={markings, mark=at position 0.5 with {\arrow{<}}},
		postaction={decorate}] (B)--(0,0) node[yshift=-9.5pt, xshift=4.5pt, near start]{$\gamma_{i+1}$};

\draw[red, ->] ({135-\ang+5}:0.65) arc ({135-\ang+5}:{225+\ang-5}:0.65) node[midway, left] {$\scriptstyle p_i$};
 \end{scope}
 \draw[thick, ->] (1,0)--(2,0);
  \begin{scope}[shift={(4,0)}]
 \def\r{1.5};
 \def\ang{-5};
 \coordinate (A) at ({135-\ang}:{\r});
  \coordinate (B) at ({225+\ang}:{\r});
    	\filldraw (0,0) circle (2.5pt);
\draw[decoration={markings, mark=at position 0.75 with {\arrow{>}}},
		postaction={decorate}, rounded corners=5ex] (A)--(0,0)--(B);
 \end{scope}
 \end{scope}

\draw (5,0) node{or};
\begin{scope}[shift={(7,0)}]
 \begin{scope}[el/.style = {inner sep=2pt, align=left}]
 \def\r{1.5};
 \def\ang{-5};
 \coordinate (A) at ({135-\ang}:{\r});
  \coordinate (B) at ({225+\ang}:{\r});
    	\filldraw (0,0) circle (2.5pt);
\draw[decoration={markings, mark=at position 0.5 with {\arrow{>}}},
		postaction={decorate}] (A)--(0,0) node[above, near start, anchor=south]{$\gamma_i$};
\draw[decoration={markings, mark=at position 0.5 with {\arrow{<}}},
		postaction={decorate}] (B)--(0,0) node[yshift=-9.5pt, xshift=4.5pt, near start]{$\gamma_{i+1}$};

\draw[red, ->] ({225+\ang+5}:0.65) arc ({225+\ang+5}:{360+135-\ang-5}:0.65) node[midway, right] {$\scriptstyle p_i$};
 \end{scope}
 \draw[thick, ->] (1.4,0)--(2.4,0);
  \begin{scope}[shift={(4,0)}]
 \def\r{1.5};
 \def\ang{-5};
 \coordinate (A) at ({135-\ang}:{\r});
  \coordinate (B) at ({225+\ang}:{\r});
    	\filldraw (0,0) circle (2.5pt);
\draw[decoration={markings, mark=at position 0.85 with {\arrow{>}}},
		postaction={decorate}, hobby] plot coordinates {(A) ({105}:0.35) (0.5,0) ({255}:0.35) (B)};
 \end{scope}
 \end{scope}

 \end{tikzpicture}
\end{displaymath}
\noindent Which one of the smoothings is applied depends on whether $p_i$ is an oriented intersection from $\gamma_i$ to $\gamma_{i+1}$ or vice versa. We simply refer to such a loop as above as a \textbf{smoothing} of the cycle $(p_i)_{i \in \mathbb{Z}_m}$.\medskip

\noindent The following lemma shows the relationship between the winding number of the smoothing and the degrees of the intersections $p_i$.

\begin{lem}\label{LemmaDegreeCondition}
Let $(\Sigma, \eta)$ be a graded punctured surface and let $\cA$ be an $\eta$-graded arc system on $\Sigma$. Suppose that $(p_i)_{i \in \mathbb{Z}_m}$ is a cycle of oriented intersections of arcs in $\cA$ and denote by $\gamma_{\text{sm}}$ its smoothing. Then,
\begin{displaymath}
 \sum_{i \in \mathbb{Z}_m} \sigma_i \cdot (\deg(p_i)-1) = \omega_\eta(\gamma_{\text{sm}}),
 \end{displaymath}
\noindent where 
\begin{displaymath}
 \sigma_i = \begin{cases} \phantom{-}1, & \text{if } p_i \in \gamma_{i} \oInt \gamma_{i+1}; \\[0.35em] -1, & \text{if } p_i \in \gamma_{i+1} \oInt \gamma_i. \end{cases}
\end{displaymath}
\end{lem}
\begin{proof}
We may assume that $\gamma_i$ is oriented in such a way that $p_i$ is its end point 
and regard  $\gamma_{\text{sm}}$ as a closed path $\gamma_{\text{sm}}:[0,1] \rightarrow \Sigma$ by choosing a base point $z=\gamma_{\text{sm}}(0)=\gamma_{\text{sm}}(1)$ in the interior of $\gamma_0$ (recall that $\gamma_{\text{sm}}$ agrees with the curves $\gamma_i$ outside of a small closed neighbourhood $U$ of the punctures). Any grading $\mathfrak{g}$ on $\gamma_i$, where $i \in \mathbb{Z}_m \setminus \{0\}$, determines a unique grading $\hat{\mathfrak{g}}$ on $\gamma_{\text{sm}}$ which agrees at the overlap of the two paths. Vice versa, every grading on $\gamma_{\text{sm}}$ induces a grading on each of the arcs $\gamma_i$, $i \neq 0$. Moreover, every grading $\mathfrak{g}$ on $\gamma_0$ determines two (and possibly distinct) gradings $\hat{\mathfrak{g}}^s$ and $\hat{\mathfrak{g}}^t$ on $\gamma_{\text{sm}}$ corresponding to the segment of $\gamma_0$ on either side of the base point $z$. Let $\mathfrak{g}_i$ denote the given grading of $\gamma_i \in \cA$ and denote by $\hat{\mathfrak{g}}_i$ the corresponding grading on $\gamma_{\text{sm}}$. Then, for $i, i+1 \neq 0$, one observes that 
\begin{equation}\label{EquationGradings}
 \hat{\mathfrak{g}}_{i+1}=\hat{\mathfrak{g}}_{i}[\sigma_i \cdot (1-\deg(p_i))],
\end{equation}

\noindent with $\sigma_i$ defined as in the assertion and likewise, $\hat{\mathfrak{g}}_1= \hat{\mathfrak{g}}_0^s[\sigma_0 \cdot (1-\deg(p_0))]$ and $\hat{\mathfrak{g}}_0^t=\hat{\mathfrak{g}}_{m-1}[\sigma_{m-1} \cdot (1-\deg(p_{m-1}))]$. To see this, we may assume that $p_i \in \gamma_i \oInt \gamma_{i+1}$ and choose a graded embedded path $u$ as in the definition of $\deg(p_i)$ in Section \ref{SectionDegreeIntersection}. That is, $u$ connects segments of $\gamma_i$ and $\gamma_{i+1}$ near $p_i$ and spans across the sector which corresponds to $p_i$. Locally the smoothing at $p_i$ can be regarded as the smoothing of the piecewise smooth path comprised of $\gamma_i$, $u$ and $\gamma_{i+1}$ at the intersections $(s_i, t_i)=q_i \in \gamma_i \oInt u$ and $q_{i+1} \in u \oInt \gamma_{i+1}$. Let $\hat{f}$ denote the grading which is induced on $\gamma_{\text{sm}}$ from $u$ in the same way as before. Then in order to prove \eqref{EquationGradings} it is sufficient to show
\begin{equation}\label{EquationSimpleCase}
\begin{array}{ccc}
\hat{\mathfrak{f}}=\hat{\mathfrak{g}}_i[1-\deg(q_{i})] & \text{and} & \hat{\mathfrak{g}}_{i+1}=\hat{\mathfrak{f}}[1-\deg(q_{i+1})]
\end{array}
 \end{equation}
 
 \noindent as this implies $\hat{\mathfrak{g}}_{i+1}=\hat{\mathfrak{g}}_i[2-\deg(q_i)-\deg(q_{i+1})]=\hat{\mathfrak{g}}_i[1-\deg(p_i)]$.
  The gradings $\mathfrak{g}_i$ and $\mathfrak{f}$ are uniquely determined by their associated paths $\mathfrak{g}_i(s_i)$ and $\mathfrak{f}(t_i)$ in the fiber $\mathbb{P}(T\Sigma)_{q_i}$ which connect $\eta(q_i)$ with $\dot \gamma_i(s_i)$ and $\dot u(t_i)$ respectively.
 Now, the first equation in \eqref{EquationSimpleCase} (and similar the second) follows from the fact that locally 
 the smoothing of $\gamma_i$ and $u$ at $q_i$ is the projection of a path which is homotopic to the path obtained by connecting $\dot \gamma_i$ and $\dot u$ via a \emph{clockwise} path in  $\mathbb{P}(TD)_{q_i}$ with endpoints $\dot \gamma_i(s_i)$ and $\dot u(t_i)$ whereas $\deg(q_i)$ is computed by connecting $\mathfrak{g}_i(s_i)$ with  $\overline{\mathfrak{f}(t_i)}$ via a \emph{counter-clockwise} path in the fiber.

 Let $\omega \in \mathbb{Z}$ be such that $\hat{\mathfrak{g}}_0^t[\omega]=\hat{\mathfrak{g}}_0^s$. Then, it follows from the definition of winding numbers that $\omega=\omega_{\eta}(\gamma_{\text{sm}})$. By the equations above, we also conclude that
\begin{displaymath}
\hat{\mathfrak{g}}_0^t= \hat{\mathfrak{g}}_{m-1}\big[\sigma_{m-1} \cdot (1-\deg(p_{m-1}))\big]  = \cdots = \hat{\mathfrak{g}}_0^s\bigg[\sum_{i \in \mathbb{Z}_m} \sigma_i \cdot \left(1-\deg(p_i)\right).\bigg]\end{displaymath}

\noindent Thus, $\sum_{i \in \mathbb{Z}_m} \sigma_i \cdot \left(1-\deg(p_i)\right)= -\omega_{\eta}(\gamma_{\text{sm}})$. This finishes the proof.
\end{proof}

\begin{cor}\label{CorollaryDegreeCondition}
Let $(p_i)_{i \in \mathbb{Z}_m}$ be as in Lemma \ref{LemmaDegreeCondition}. Then, 
\begin{displaymath}
 m + \sum_{i \in \mathbb{Z}_m} \deg(p_i) \equiv \omega_\eta(\gamma_{\text{sm}}) \mod 2.
\end{displaymath}
\end{cor}

\begin{cor}\label{CorollaryWindingEdges}
In the situation of Corollary \ref{CorollaryDegreeCondition} and under the additional assumption that $\eta$ is of ribbon type we also have  $$\sum_{i\in\mathbb{Z}_m}\omega_\eta(\gamma_i) \equiv  m + \omega_\eta(\gamma_{\text{sm}})\mod 2,$$ since $\dot \gamma_{\text{sm}}$ has one additional crossing with $\eta$ in a neighbourhood of each $p_i$ compared to $\bigcup\limits_{i\in\mathbb{Z}_m} \dot \gamma_i|_{(0,1)}$. Thus,
\begin{displaymath}
\sum_{i \in \mathbb{Z}_m} \omega_{\eta}(\gamma_i) \equiv \sum_{i \in \mathbb{Z}_m} \deg(p_i) \mod 2.
\end{displaymath}
\end{cor}

\subsection{Fukaya categories of marked surfaces} \ \medskip

\noindent Suppose that  $\cA$ is a full graded arc system on a graded surface $(\Sigma, \marked, \eta)$ with marked points and suppose that each boundary component of $\Sigma$ contains at least one marked point. For the given data, Haiden, Katzarkov and Kontsevich \cite{HaidenKatzarkovKontsevich} defined an $A_{\infty}$-category $\mathcal{F}=\mathcal{F}_{\cA}(\Sigma)$ and showed that the associated triangulated category
\begin{displaymath}
\Fuk(\Sigma)=\HH^0\left(\Tw \mathcal{F}\right),
\end{displaymath}

\noindent does not depend on $\cA$ up to equivalence.  The category $\Fuk(\Sigma)$ is referred to as the \textbf{Fukaya category} of the graded surface $(\Sigma, \marked, \eta)$. The precise definition of $\Fuk(\Sigma)$ in a slightly restricted context is given below.\medskip

\noindent
 Let $D$ be a disc with a set $\marked \subseteq \partial D$ of $n \geq 1$ marked points and let $\Sigma$ be a decorated surface. By a \textbf{marked disc} on $\Sigma$ we mean a continuous map $D \rightarrow \Sigma$ which is a smooth immersion on $D \setminus \marked$ and which sends points in $\marked$ to points in $\deco$ and boundary arcs of $D$ to arcs of $\cA$. We will assume that the edges of the disc are numbered in the clockwise order by $1,\dots,n \in \mathbb{Z}_n$. In particular, a marked disc whose corresponding sequence of arcs in $\cA$ is $\delta_1, \dots, \delta_n$ specifies a cyclic sequence of oriented intersections $p_i \in \delta_i \oInt \delta_{i+1}$. 

\begin{definition}\label{DefinitionFukayaCategory}
Suppose that  $\cA$ is a full graded arc system on a graded marked surface $(\Sigma, \marked, \eta)$ and suppose that each boundary component of $\Sigma$ contains a marked point. The strictly unital $A_{\infty}$-category $\cF=\cF_{\cA}(\Sigma)$ is a category with 
\begin{itemize}
    \item \textbf{Objects:} $\{X_\gamma\}_{\gamma\in\cA}$;
    \item \textbf{Morphisms:} Let $\gamma_1, \gamma_2 \in \cA$ be distinct. Then $\gamma_1 \oInt \gamma_2$ parametrizes a $\Bbbk$-basis of the vector space $\Hom_{\cF}(X_{\gamma_1}, X_{\gamma_2})$. The identity morphism and the set of self-intersections of $\gamma_1$ form a $\Bbbk$-basis of $\Hom_{\cF}(X_{\gamma_1},X_{\gamma_1})$.
    \item \textbf{Grading:} The degree of the morphism $a$ corresponding to a (self-)intersection $p$ is $\deg(p)$ and is denoted by $|a|$. 
    \item \textbf{Composition:} If $a \in \Hom_{\cF}(X_{\gamma_2}, X_{\gamma_3})$ corresponds to $p \in \gamma_2 \oInt \gamma_3$ and $b \in \Hom_{\cF}(X_{\gamma_1}, X_{\gamma_2})$ corresponds to $q \in \gamma_1 \oInt \gamma_2$, then $\mu_2(a,b)=(-1)^{|b|}c$, where $c$ corresponds to $pq \in \gamma_1 \oInt \gamma_3$ if defined and $c=0$ otherwise. All other compositions of basis elements vanish. 
    \item \textbf{Higher operations:} Let $D$ be a marked disc on $\Sigma$ and denote by $p_1, \dots, p_n$ its associated sequence of oriented intersections between arcs of $\cA$. The corresponding sequence of morphisms  $a_1,\dots,a_n$  will be called  a \textbf{disc sequence}. For any disc sequence  $a_1,\dots, a_n$, $$\mu_n(a_n,\dots,a_1b)=(-1)^{|b|}b \text{ for }a_1b\neq 0 \text{ and}$$ $$\mu_n(ba_n,\dots,a_1)=b \text{ for }ba_n\neq 0.$$ 
\noindent Higher operations vanish on all sequences of maps which are not of the form specified above.
\end{itemize}
\end{definition}

\begin{rem}It follows from Lemma \ref{LemmaDegreeCondition} that $ |a_1|+\dots+|a_n|=n-2$ for any disc sequence $a_1,\dots,a_n$ which shows that all operations in $\cF_{\cA}(\Sigma)$ have the correct degree. 
\end{rem}
\noindent Similar to the case of Brauer graph algebras, we can think of arcs of $\cA$ as vertices  and of morphisms as paths in a quiver.\medskip

\noindent For our purposes it is useful to extend Definition \ref{DefinitionFukayaCategory} to the case when $\partial\Sigma$ contains components without marked points. For example, this can be done in the following way. 
\begin{definition}\label{DefinitionExtendedFukayaCategories}[Extended definition of Fukaya categories]
Let $\cA$ be a full arc system on a graded marked surface $(\Sigma, \marked, \eta)$. Denote by $B_1, \dots, B_m \subseteq \partial \Sigma$ all boundary components with $B_i \cap \marked = \emptyset$ and for each $i \in [1,m]$,  let $q_i \in B_i$. Choose a full graded arc system $\cB \supseteq \cA$ of $(\Sigma, \marked \sqcup \{q_i\}_{i \in [1,m]}, \eta)$ and define $\mathcal{F}_{\cA}$ as the full $A_{\infty}$-subcategory of $\mathcal{F}_{\cB}$ which is spanned by the objects $\cA \subseteq \cB$.
\end{definition}

\noindent It is not difficult to verify that the $A_{\infty}$-category $\mathcal{F}_{\cA}$  is independent of the choice of the points $\{q_i\}$ and the arc system $\cB$.

\begin{definition}
A graded admissible arc system which does not admit any marked discs will be called \textbf{formal}. If $\cA$ is formal, then $A_{\infty}$-category $\mathcal{F}_{\cA}$ is formal, that is all its higher operations vanish. 
\end{definition}

\begin{rem}\label{RemarkFukayaGentle} If $\cA$ is formal and all morphisms in $\cF_{\cA}$ are concentrated in degree zero, then the endomorphism algebra of $\bigoplus_{\gamma \in \cA}X_{\gamma}$ is a \textit{gentle algebra} in the sense of \cite{AssemSkowronski} and every gentle algebra can be obtained in this way, see \cite{OpperPlamondonSchroll} or \cite{LekiliPolishchukGentle}. Thus, Definition \ref{DefinitionExtendedFukayaCategories} allows us to obtain all gentle algebras as idempotent subalgebras of endomorphism algebras of generators of Fukaya categories. In particular, for any gentle algebra $A$, $\Kb{A}$ is a full subcategory of $\H^0\Tw(\mathcal{F}_{\cB})$ of the homotopy category of a Fukaya category.
\end{rem}

\begin{rem}\label{RemarkSketchHKKProof}
\noindent The proof that $\Fuk(\Sigma)$ is essentially independent of the chosen arc system  uses the following ingredients \cite[Lemma 3.2, Proposition 3.2]{HaidenKatzarkovKontsevich}:

\begin{enumerate}
\item \label{ListPropertiesFukayacategories1} If $\cA$ is a full arc system and $\gamma \in \cA$ is an arc such that $\cA'=\cA \setminus \{\gamma\}$ is full, then there is a fully faithful inclusion $\mathcal{F}_{\cA'} \subseteq \mathcal{F}_{\cA}$ of $A_{\infty}$-categories which is a Morita equivalence.

\item \label{ListPropertiesFukayacategories2} Every two triangulations $\cA$ and $\cB$ of $(\Sigma,\marked)$ are related by a sequence of \textit{flips} each of which removes an arc from the given arc collection and adds a new arc. It therefore follows from \eqref{ListPropertiesFukayacategories1}  that $\mathcal{F}_{\cA}$ and $\mathcal{F}_{\cB}$ are related by a zig-zag of Morita equivalences. 

\item \label{ListPropertiesFukayacategories3} The construction of the Morita equivalence in \eqref{ListPropertiesFukayacategories1} is based on the following observation. If $\cA$ is full and $\cA'=\cA \setminus \{\gamma\}$ is full, then there exists a disc sequence $a_1,\dots,a_n$ involving objects $X_0, X_1, \dots, X_n=X_0$ in $\Tw(\cF_\cA)$ such that $X_0=X_{\gamma}$ and there exists a twisted complex $X_0'$ which is built from shifts of objects $X_1,\dots,X_{n-1}$ such that $X_0\simeq X_0'$ in $\H^0(\Tw\cF_\cA)$. It follows that $\cA$ and $\cA'$ generate the same triangulated subcategory of $\H^0(\Tw\cF_\cA)$ and hence the whole category $\H^0(\Tw\cF_\cA)$. In other words, the inclusion $\H^0(\Tw\cF_{\cA'}) \hookrightarrow \H^0(\Tw\cF_\cA)$ is an equivalence.
\end{enumerate}
\end{rem}

\noindent Analogues of \eqref{ListPropertiesFukayacategories1} and \eqref{ListPropertiesFukayacategories2} for Brauer graph categories will be essential in the proof of Theorem \ref{IntroTheoremCriterionDerivedEquivalence}.\medskip

\noindent For the sake of completeness we show that $\cF_{\cA}$ is still independent of $\cA$ if the ambient surface contains boundary components without marked points.

\begin{lem}\label{LemmaWellDefineFukaya}
Let $(\Sigma, \marked, \eta)$ be a marked surface. Let $\cA$ be a full arc system on $\Sigma$ and let $\gamma\in \cA$ be an arc such that $\cA\backslash\{\gamma\}$ is full. Then the natural inclusion $\cF_{\cA\backslash\{\gamma\}}\rightarrow \cF_\cA$ is a Morita equivalence. Moreover, the Morita equivalence class of $\cF_\cA$ is independent of $\cA$.
\end{lem}
\begin{proof}
Let $\cA$ be a full arc system on a graded marked surface $(\Sigma, \marked, \eta)$. Denote by $B_1, \dots, B_m \subseteq \partial \Sigma$ the boundary components with $B_i \cap \marked = \emptyset$ and for each $i \in [1,m]$,  let $q_i \in B_i$. Recall that the arc system $\cA$ cuts the surface $\Sigma$ into discs and annuli. The boundary of each annulus is a union of one of the $B_i$ and  a collection of arcs from $\cA$. Let us denote by $U_i$ the unique annulus  whose boundary contains $B_i$.

We choose a full graded arc system $\cB \supseteq \cA$ of $(\Sigma, \marked \sqcup \{q_i\}_{i \in [1,m]}, \eta)$  as follows: for each $i \in [1,m]$ let $\delta_i\in \cB\backslash \cA$ be an arc connecting $q_i\in B_i$ with a marked point from $\marked$ which lies in $\partial U_i$. This clearly produces a full arc system of  $(\Sigma, \marked \sqcup \{q_i\}_{i \in [1,m]}, \eta)$ since each annulus is cut into a disc along some $\delta_i$. The $A_{\infty}$-category $\mathcal{F}_{\cA}$ is  the full $A_{\infty}$-subcategory of $\mathcal{F}_{\cB}$ which is spanned by the objects $\cA \subseteq \cB$.
Let $\gamma \in \cA$ be such that $\cA\backslash \{\gamma\}$ is full. Then, $\cB\backslash\{\gamma\}$ is  full as well. The disc sequence which is used  to construct the Morita equivalence $\cF_{\cB\backslash\{\gamma\}}\rightarrow \cF_\cB$ (see Remark \ref{RemarkSketchHKKProof} \eqref{ListPropertiesFukayacategories3})   does not contain the objects $X_{\delta_i}$ since the corresponding marked disc does not contain the arcs $\delta_i$. Thus, the subcategories generated by the objects $\{X_\delta\}_{\delta\in \cA}$ and by the objects $\{X_\delta\}_{\delta\in \cA\backslash\{\gamma\}}$ in $\H^0\Tw(\cF_\cB)$ coincide which shows  that the inclusion $\cF_{\cA\backslash\{\gamma\}}\rightarrow \cF_\cA$ is a Morita equivalence. Note that if $\cA$ and $\cA'$ differ only by their gradings, we also get a Morita equivalence $\cF_{\cA'}\rightarrow \cF_\cA$.

Finally, let us consider any two full arc systems $\cA$, $\cA'$ on $(\Sigma,\marked)$. Both systems can be completed to triangulations of $(\Sigma,\marked)$ without changing their Morita equivalence classes. As any two triangulations of $(\Sigma,\marked)$ are connected by a sequence of edge flips (e.g.\ see \cite{Hatcher}) there exists a zig-zag of Morita equivalences, connecting $\cA$ and $\cA'$ up to a change of grading.
\end{proof}

\section{Brauer graph categories and trivial extensions of Fukaya categories}\label{SectionBrauerGraphCategories}

\noindent In this section we study trivial extensions of Fukaya categories and their ``coordinate-free'' 
counterparts which we call Brauer graph categories.

\subsection{Brauer graph categories}

\begin{definition}\label{DefinitionAdmissible}Let $\Sigma$ be a punctured surface.
A full arc system $\mathcal{A}$ on $\Sigma$ is \textbf{admissible} if for every puncture $p$, there exists an embedded path $c_p:[0,1] \rightarrow \left(\Sigma \setminus |\mathcal{A}|\right) \cup \{p\}$  such that $c_p(0)=p$ and $c_p(1) \in \partial \Sigma$.
\end{definition}

\noindent If $p$ is a puncture, a path $c_p$ as in Definition \ref{DefinitionAdmissible} is called a \textbf{cutting path} of $p$. One considers two cutting paths of $p$ as \textbf{homotopic} if there exists a homotopy between them all of whose intermediate paths are cutting paths of $p$. In other words, homotopies are allowed to move the end point of a cutting path freely within $\partial \Sigma$. In general there may exist many non-homotopic cutting paths for a puncture.

\begin{definition}\label{DefCP}
	Let $\mathcal{A}$ be an admissible arc system. A \textbf{cut} of $\mathcal{A}$ is a collection $\mathsf{C}=(c_p)_{p \in \mathscr{P}}$ of pairwise disjoint paths such that for every $p \in \mathscr{P}$ the path $c_p$ is a cutting path of $p$. The pair $(\cA, \mathsf{C})$ is called a \textbf{cutting pair}.
\end{definition}

\noindent We recall that any arc system $\cA$ can be naturally regarded as a ribbon graph $\Gamma=\Gamma_{\cA}$ whose vertices are punctures and whose edges are the arcs of $\mathcal{A}$. Let $p$ be a puncture, or equivalently, a vertex of $\Gamma$. Every cut $\sC$ of $\mathcal{A}$ determines an arrow $\alpha_p^\sC$ in $Q_{\Gamma}$ which corresponds to the unique pair of half-edges $(h,h^+) \in \BH{\Gamma}^2_p$ such that in a small neighbourhood of $p$ the curve $c_p$ belongs to the sector which lies between $h$ and $h^+$. For any path $\alpha$ of $\Bbbk Q_\Gamma$, we write $\alpha_p^\sC\in \alpha$, if $\alpha$ decomposes as $\alpha=u \alpha_p^\sC v$ for some paths $u$ and $v$. 

\begin{notation}
Let $B_{\Gamma}$ be a modified Brauer graph algebra and let $\gamma=\alpha \pi(\alpha)\cdots\pi^i(\alpha)$ be a non-trivial path. We will denote by $\gamma^*$ the path of $B_{\Gamma}$ which complements $\gamma$ to a maximal non-zero path, that is $\gamma\gamma^*=C_\alpha^{\seq(C_\alpha)}$.
\end{notation}

\noindent Note that  for any path $\gamma$, $\gamma^{\ast}\gamma=C_{\beta}^{\seq(C_\beta)}$ for some arrow $\beta$. If $\gamma=\alpha \pi(\alpha)\dots\pi^i(\alpha)$, then $\beta=\pi^{i+1}(\alpha)$.

\begin{definition}\label{DefinitionBrauerGraphCategory} Let $\cA$ be a graded arc system on a graded punctured surface $(\Sigma, \punct, \eta)$ with a line field of ribbon type. Let $\seq$ be a multiplicity function on $\punct$. The \textbf{Brauer graph category} $\mathbb{B}=\mathbb{B}(\cA, \seq)$ is the following $A_{\infty}$-category.
\begin{itemize}
    \item \textbf{Objects:} The set $\{X_\gamma\}_{\gamma\in \cA}$ is the set of objects of $\mathbb{B}$.
    
    \item \textbf{Morphisms:} Denote by $B_{\mathcal{A}}$ the modified Brauer graph algebra of $(\Gamma_{\mathcal{A}}, \seq, \omega)$, where $\omega(\gamma)\coloneqq \omega_\eta(\gamma)$ for all $\gamma\in \cA=\BE{\Gamma_{\cA}}$. An arc $\gamma$ corresponds to a vertex of the quiver $Q_{\Gamma_\cA}$. Given graded arcs $\gamma, \delta \in \mathcal{A}$, the set of equivalence classes of paths from the vertex $\gamma$ to the vertex $\delta$ in $B_{\mathcal{A}}$ is a $\Bbbk$-basis of $\Hom_{\mathbb{B}}(X_\gamma, X_\delta)$.
    
    \item \textbf{Gradings:} Every path is a homogeneous morphism whose degree equals the sum of the degrees of its arrows. An arrow $a \in \Hom_{\mathbb{B}}(X_\gamma, X_\delta)$ which corresponds to an intersection $p \in \gamma \oInt \delta$ at a puncture has degree $|a|\coloneqq \deg(p)$. The degree of the loop corresponding to a pair of half edges $(h,h)$ at a vertex with valency $1$ is $0$.
    
    \item \textbf{Composition:}  For classes of paths  $a \in \Hom_{\mathbb{B}}(X,Y)$, $b \in \Hom_{\mathbb{B}}(Y,Z)$ set  $$\mu_2(b,a)\coloneqq(-1)^{|a|}b a, $$
    where $ba$ denotes the associated product in $B_{\cA}$.
    
    \item \textbf{Higher operations:} Let $a_n,\dots,a_1$ be a disc sequence (recall that $a_i$ correspond to oriented intersections $p_i\in \delta_i \oInt \delta_{i+1}$), let $b$ be a path, then 

       \begin{equation}\label{Mu1}\mu(ba_n,\dots,a_1)=b \mbox{, for } ba_n\neq 0;\end{equation}
         \begin{equation}\label{Mu2}\mu(a_n,\dots,a_1b)=(-1)^{|b|}b \mbox{, for } a_1b\neq 0;\end{equation}
         \begin{equation}\label{HigherOps5}
         \begin{aligned} \mu(a_n,\dots, a_{r+1}, a_r(ba_r)^*,ba_r,a_{r-1},\dots, a_2)=(-1)^{\circ}a_1^* \mbox{, for } ba_r\neq 0,\\
         \text{where }\circ=|a_1|+|a_2|+\cdots+ |a_{r-1}|+|ba_r|+\omega_\eta(\delta_2)+\cdots+\omega_\eta(\delta_r).
         \end{aligned}
         \end{equation}
         \item[] Higher operations vanish for all sequences of elements which are not of the form above.
\end{itemize}
\end{definition}
\noindent Note that partially wrapped Fukaya categories and Brauer graph categories do not change if we replace the arcs in the underlying arc system by homotopic arcs with the corresponding grading. For Brauer graph categories this follows from Remark \ref{RemarkWindingNumberWellDefined}.
\begin{lem}
Let $\cA$ be an $\eta$-graded arc system on a graded punctured surface $(\Sigma, \punct, \eta)$ of ribbon type. Let $\seq$ be a multiplicity function on $\punct$. The Brauer graph category $\mathbb{B}=\mathbb{B}(\cA, \seq)$ is well-defined.
\end{lem}

\begin{proof}
Since $\eta$ is of ribbon type, the ideal of relations $I'_\Gamma$ is homogeneous and the degrees are well defined. The elements $a_1^*$ and $(ba_r)^*$ are defined, since $a_1$ and $a_r$ are non-trivial paths.  The only thing we need to check is whether the higher operations agree on elements which are identified under the relations in the ideal $I'_\Gamma$. Let $c_n,\dots,c_1$ be a sequence of composable elements of the form appearing in \eqref{Mu1}, \eqref{Mu2} or \eqref{HigherOps5} and let us assume that one of the elements $c_i$ is of the form $C_\alpha^{\seq(C_\alpha)}$ for some arrow $\alpha$. In the case of \eqref{Mu1} the only element that can be of this form is $ba_n$ since the arc system $\cA$ is admissible. That is, $ba_n=a_n^*a_n=(-1)^{\omega_\eta(\delta_n)}a_{n-1}a_{n-1}^*$, where $\delta_i$ denotes the arc corresponding to the domain of $a_i$ and the codomain of $a_{i-1}$. Let us compare the expressions $\mu(c_n,\dots,c_1)$ for the two different ways of interpreting $c_n$. We have
$$\mu(a_n^*a_n,\dots,a_1)=a_n^* \text{ and}$$
$$(-1)^{\omega_\eta(\delta_n)}\mu(a_{n-1}a_{n-1}^*,a_{n-1}\dots,a_1)=(-1)^{\omega_\eta(\delta_n)+\circ}a_n^*,$$
where $\circ=|a_n|+|a_1|+\cdots+ |a_{n-1}|+\omega_\eta(\delta_1)+\cdots+\omega_\eta(\delta_{n-1})$. The signs agree because $|a_1|+\cdots+ |a_{n}| \equiv \omega_\eta(\delta_1)+\cdots+\omega_\eta(\delta_{n}) \mod 2$ by Corollary \ref{CorollaryWindingEdges}. 

\noindent In the case of \eqref{Mu2} we get:  the only element that can have the form $C_\alpha^{\seq(C_\alpha)}$ is $a_1b$. That is, $a_1b=a_1a_1^*=(-1)^{\omega_\eta(\delta_2)}a_{2}^*a_{2}$. A comparison yields
\begin{align*}
    \mu(a_n,\dots,a_1a_1^*) & = (-1)^{|a_1^*|}a_1^*\; \text{   and}\\
     (-1)^{\omega_\eta(\delta_2)}\mu(a_n,\dots,a_2,a_2^*a_2) & = (-1)^{\omega_\eta(\delta_2) +|a_1|+|a_2^*a_2|+\omega_\eta(\delta_2)}a_1^*.
\end{align*}
\noindent The second formula follows from \eqref{HigherOps5} with $a_2^{\ast} a_2$ in the role of $ba_r$. Since $|a_i^{\ast}|=-|a_i|$ the signs agree.

\noindent Let us check the case of \eqref{HigherOps5}. The only elements that can have the form $C_\alpha^{\seq(C_\alpha)}$ are $a_r(ba_r)^*$ and $ba_r$. In the first case $b$ is the trivial path and $a_r(ba_r)^*=a_ra_r^*=(-1)^{\omega_\eta(\delta_{r+1})}a_{r+1}^*a_{r+1}$. In the second case $ba_r=a_{r}^*a_r=(-1)^{\omega_\eta(\delta_{r})}a_{r-1}a_{r-1}^*$. These two cases coincide up to a shift in indices. Let us compare the higher operations for the case $ba_r=a_{r}^*a_r=(-1)^{\omega_\eta(\delta_{r})}a_{r-1}a_{r-1}^*$:
$$ \mu(a_n,\dots, a_{r+1}, a_r,a_{r}^*a_r,a_{r-1},\dots, a_2)=(-1)^{\circ}a_1^* ,$$ $$
         \text{where }\circ=|a_1|+|a_2|+\cdots+ |a_{r-1}|+|a_r^*a_r|+\omega_\eta(\delta_2)+\cdots+\omega_\eta(\delta_r);
         $$
$$ (-1)^{\omega_\eta(\delta_{r})} \mu(a_n,\dots, a_r,a_{r-1}a_{r-1}^*,a_{r-1},\dots, a_2)=(-1)^{\omega_\eta(\delta_{r})+\circ}a_1^* ,$$ $$
         \text{where }\circ=|a_1|+|a_2|+\cdots+ |a_{r-1}|+\omega_\eta(\delta_2)+\cdots+\omega_\eta(\delta_{r-1}).
         $$         
The signs agree, since $|a_r^*a_r|=0$. The cases $r=2$ or $r-1=n$ were already dealt with when discussing \eqref{Mu1} and \eqref{Mu2}.         
\end{proof}

\noindent It is clear that $\bB=\mathbb{B}(\mathcal{A}, \seq)$ satisfies \eqref{EquationAInfinityConstraint} for $n=3$ and hence admits the structure of a graded $\Bbbk$-linear category with composition given by $g \circ f \coloneqq (-1)^{|f|}\mu^2_{\bB}(g,f)$ due to the vanishing of $\mu^1_{\bB}$. We delay the proof that $\mathbb{B}(\mathcal{A}, \mathbf{1})$ is an $A_{\infty}$-category to Section \ref{SectionTrivialExtensionsFukaya} where we show  that $\mathbb{B}(\mathcal{A},  \mathbf{1})$ is the trivial extension of a related Fukaya category which we construct by means of a cut. This can be seen as an $A_\infty$-counterpart of a theorem by Schroll \cite[Theorem 1.2.]{SchrollTrivialExtension} stating that a trivial extension of a gentle algebra is isomorphic to a Brauer graph algebra.
For a general multiplicity function $\seq$, we will see in Corollary \ref{CorollaryBGACategoryHigherMultWellDefined} that $\mathbb{B}(\mathcal{A},  \seq)$ arises as an orbit category of $\mathbb{B}(\mathcal{A}^{\seq},  \mathbf{1})$, where $\mathcal{A}^{\seq}$ is an arc system on a suitable branched cover of $\Sigma$ with branch locus $\mathscr{P}$.\medskip

\noindent Let $(\Gamma, \seq)$ be a Brauer graph and let $(\Sigma_{\Gamma}, \eta_{\Gamma})$ denote the associated graded surface. Since each arc $\gamma\in E(\Gamma)$ is transverse to the line field, we can grade $\gamma$ by taking the clockwise embedded (that is the ``shortest'') path from $\gamma^* \eta_{\Gamma}(t)$ to $\dot\gamma(t)$ for each $t\in(0,1)$. This provides a homotopy from $\gamma^* \eta_{\Gamma}$ to $\dot\gamma$ and defines a grading.\medskip

\noindent The next proposition shows that Brauer graph categories are a generalization of Brauer graph algebras.

\begin{prp}\label{PropositionBGAIsBrauerGraphCategory}Let $(\Gamma, \seq)$ be a Brauer graph regarded as an $\eta_{\Gamma}$-graded arc system on its corresponding graded surface $(\Sigma_{\Gamma}, \eta_{\Gamma})$ as explained above. Then $\mathbb{B}=\bB(\Gamma, \seq)$ is isomorphic to the category of indecomposable projective modules of $B=B_{(\Gamma, \seq)}$ as defined in Remark \ref{RemarkTwistedComplexesForFiniteDimensionalAlgebras}. In particular, $\H^0(\Tw \mathbb{B}) \cong \Kb{B}$.
\end{prp}
\begin{proof}
We consider the usual splitting $B = \bigoplus_{\gamma \in \BE{\Gamma}} P_\gamma$ of $B$ into indecomposable projective $B$-modules which correspond to its vertices $\BE{\Gamma}=(Q_{\Gamma})_0$ and denote by $\bA$ the associated category of projective $B$-modules from Remark \ref{RemarkTwistedComplexesForFiniteDimensionalAlgebras}. The desired isomorphism $\bB \rightarrow \bA$  maps an object $X_{\gamma} \in \Ob{\mathbb{B}}$, where $\gamma \in \BE{\Gamma}=(Q_{\Gamma})_0$, to $P_\gamma$ and a morphism $X_\gamma \rightarrow X_\delta$ which corresponds to a path $\alpha$ in $Q_{\Gamma}$ which starts at $\gamma$ and ends at $\delta$ to the  induced morphism $P_\gamma \rightarrow P_\delta$. It is immediate from the choice of grading on $E(\Gamma)$ that the degrees of all intersections of arcs at vertices of $\Gamma$ are zero, so all morphism in $\bB$ are of degree zero. Since all faces of $\Gamma$ contain a boundary component, the are no marked discs on $\Sigma_\Gamma$, and hence no higher operations in $\mathbb{B}$. Finally, all edges $\gamma\in \BE\Gamma$ are nowhere parallel to $\eta_{\Gamma}$. Hence, their winding numbers vanish and so the relations in the definition of the Brauer graph category are not modified by signs. \end{proof}
\subsection{Multiplicity-free Brauer graph categories as trivial extensions of Fukaya categories}\label{SectionTrivialExtensionsFukaya}

\subsubsection{Arc systems and marked surfaces induced by cuts} \ \medskip

\begin{definition}Let $\mathcal{A}$ be an admissible arc system on a punctured surface $(\Sigma, \deco, \eta)$ and let $\sC$ be a cut of $\cA$. The \textbf{cut surface} of $\sC$ is the marked surface $\Sigma_\mathsf{C}=(\Sigma, \deco_\mathsf{C})$, where $\deco_\mathsf{C} \subseteq \partial \Sigma$ denotes the set of all end points of the paths $c_p$. 
\end{definition}

\noindent Every cut of an admissible arc system $\mathcal{A}$ induces an arc system  of its cut surface by dragging the punctures along the cutting paths and  deforming the arcs of $\mathcal{A}$ accordingly. The newly created arc collection is a full arc system of $\Sigma_\mathsf{C}$ which we denote by $\mathcal{A}_\mathsf{C}$. The arcs of $\cA$ and $\cA_\mathsf{C}$ are in bijection. Alternatively, and more conveniently, one may choose a simple path $\delta_B:S^1 \rightarrow \Sigma$ for every boundary component $B$ of $\Sigma$ which is homotopic to $B$ (considered as a path) and such that $\delta_B \cap |\cA|$ consists of all $p \in \deco$ with $c_p(1) \in B$. Then $\Sigma$ is homeomorphic to the subsurface $\Sigma' \subseteq \Sigma$ such that $\Sigma \setminus \Sigma'$ is a collection of annuli bounded by $\partial \Sigma$ and all the loops $\delta_B$. We may choose a  homeomorphism $\Sigma' \rightarrow \Sigma$ which maps $p$ to $c_p(1)$ and which extends the inclusion $\Sigma' \setminus T \hookrightarrow \Sigma$, where $T$ is a small collar neighbourhood of $\partial \Sigma'$ in $\Sigma'$. 
The homeomorphism allows us to deform any line field $\eta$ on $\Sigma$ ``along the cutting paths'' and induce a line field $\eta_{\mathsf{C}}$ on $\Sigma_{\mathsf{C}}$. In other words, we can think of the triple $(\Sigma_{\mathsf{C}}, \deco_C, \eta_{\mathsf{C}})$ as of $(\Sigma', \deco, \eta|_{\Sigma'})$. The construction ensures two useful properties of $\eta_{\mathsf{C}}$. First, there exists a canonical bijection between gradings of $\cA$ and gradings of $\cA_{\mathsf{C}}$ and second, if $p$ is an oriented intersection of arcs in $\cA$ which corresponds to an oriented intersection of arcs in $\cA_{\mathsf{C}}$ with their respective gradings, then the two intersections are of the same degree. Subsequently, whenever $\cA$ is graded we assume implicitly that $\cA_{\mathsf{C}}$ is equipped with the corresponding  grading under the above bijection.

\subsubsection{Isomorphisms between Brauer graph categories and trivial extensions} \ \medskip

\noindent As indicated before we prove the following relationship between Fukaya categories and Brauer graph categories.

\begin{prp}\label{PropositionTrivialExtensionIsomorphism}
Let $(\mathcal{A}, \mathsf{C})$ be a cutting pair on a punctured graded surface $(\Sigma, \eta)$ of ribbon type. Suppose that $\cA$ is $\eta$-graded. Then, there exists an isomorphism
\begin{displaymath}
\begin{tikzcd}[ampersand replacement=\&] \Phi:\triv\left(\mathcal{F}_{\mathcal{A}_{\mathsf{C}}}\right) \arrow{r}{\simeq} \& \mathbb{B}(\cA, \mathbf{1}) \end{tikzcd}\end{displaymath}
\noindent between the underlying graded $\Bbbk$-linear categories under which the $A_{\infty}$-operations in Definition \ref{DefinitionBrauerGraphCategory} correspond to the $A_{\infty}$-operations of $\triv(\mathcal{F}_{\cA_{\mathsf{C}}})$. In particular, $\mathbb{B}(\cA, \mathbf{1})$ is an $A_{\infty}$-category.
\end{prp}

\noindent Before we give a proof of the proposition, let us describe $\Phi$ in the easiest case of ordinary algebras. The general definition is similar but involves a coherent choice of signs  whose existence is a more subtle issue. Suppose that $\cA$ is a graded admissible formal arc system on a graded punctured surface $\Sigma$ such that  $\mathbb{B}=\mathbb{B}(\cA, \mathbf{1})$ is concentrated in degree zero and let $\mathsf{C}$ be a cut of $(\Sigma, \cA)$. Then, after identifying $\cA$ and $\cA_{\sC}$ canonically as sets, $\Phi$ maps $X_\gamma \in \Ob{\cF_ {\cA_{\sC}}}$ to $X_{\gamma} \in \Ob{\mathbb{B}}$ and a morphism associated to an oriented intersection to the corresponding path in $B_{\Gamma_{\cA}}$. If further $g$ is a morphism in $\mathcal{F}_{\cA_{\sC}}$ corresponding to a path $\beta$ of $B_{\Gamma_{\cA}}$ and $f_{\beta}$ denotes the functional dual to $g$ with respect to the basis in Definition \ref{DefinitionFukayaCategory}, then  $\Phi(f_{\beta})=\beta^{\ast}$. This is easily seen to define an isomorphism as all signs which arise from winding numbers of arcs and their gradings vanish. 

\begin{proof}[Proof of Proposition \ref{PropositionTrivialExtensionIsomorphism}]
Set $\mathcal{F}\coloneqq \mathcal{F}_{\mathcal{A}_{\mathsf{C}}}$, $\mathbb{T}\coloneqq \triv(\mathcal{F})$ and $\mathbb{B} \coloneqq \mathbb{B}(\mathcal{A}, \mathbf{1})$. We further denote by $B_{\cA}$ the modified Brauer graph algebra of $(\Gamma_{\cA}, \mathbf{1}, \omega)$, where $\omega(\gamma)=\omega_\eta(\gamma)$ for all $\gamma \in \BV{\Gamma_{\cA}}=\cA$.
The vector space $\Hom_{\mathcal{F}}(X_\gamma,X_\delta)$ has a basis consisting of paths $\gamma\rightarrow \delta$ corresponding to the oriented intersections and a trivial path in case $\gamma=\delta$. For a path $a:\gamma \rightarrow \delta$, we denote by $f_a \in \mathbb{D}\Hom_{\mathcal{F}}(X_\gamma,X_\delta)$ the corresponding element of the dual basis.  Elements of these two bases for appropriate $\Hom$-spaces in $\cF$ form a basis of $\Hom_\mathbb{T}(X_\gamma,X_\delta)$. We will make all the computations on the elements of this basis.\medskip

\noindent \textbf{Computation of $\mu_{\mathbb{T}}$:} Let us first express the multiplication in $\mathbb{T}$ in terms of the paths in $\cF$ and their duals. For any paths $a,b$ in $ \cF$ we have $\mu_2(a,b)=(-1)^{|b|}ab$. We recall that $\dg{a}=|a|-1$. For any paths  $a,c$ and $d$  such that $c=da$ we have:
\begin{equation}\label{FormulaMult1}
    \mu_2(a,f_c)=(-1)^{\dg{a}+\dg{f_c}}f_c(\mu_2(-,a))=(-1)^{\dg{a}+\dg{f_c}}(-1)^{|a|}f_d=(-1)^{|f_c|}f_d.
\end{equation}
Likewise, for paths $l,m$ and $b$ such that $l=bm$ we get: 
\begin{equation}\label{FormulaMult2}\mu_2(f_l,b)=(-1)^{\dg{f_l}+\dg{b}}f_l(\mu_2(b,-))=(-1)^{\dg{f_l}+\dg{b}}(-1)^{|m|}f_m=f_m.
\end{equation}
The last equality follows from $|f_l|+|b|+|m|=0$. Multiplication is zero on all other pairs of elements from the basis.

Next, we compute the higher operations in $\mathbb{T}$. If $a_1, \dots, a_n$ is a disc sequence and $b$ is an element in $\cF$ such that $a_1b$ is defined, then  $\mu_\cF(a_n,\dots,a_1b)$ agrees with $\mu_\mathbb{T}(a_n,\dots,a_1b)$ and the same holds for $\mu(ba_n,\dots,a_1)$. If $l > 2$ and $\mu_{\mathbb{T}}(u_l, \dots, u_1) \neq 0$ but the tuple $(u_l, \dots, u_1)$ is not of the above form, then one of its entries must be  a functional. We may assume that $(u_{m}, \dots, u_1)=(x_m, \dots, x_{i+1}, f, x_{i-1}, \dots, x_1)$ where $f$ is a functional and each $x_j$ is a path in $\cF$. In that case, $\mu_{\mathbb{T}}(x_m, \dots x_{i+1},f,x_{i-1},\dots, x_1)\neq 0$ if and only if there exists a morphism $x$ in $\cF$ such that $f(y)\neq 0$, where $y = \mu_\cF(x_{i-1},\dots, x_1,x,x_m, \dots x_{i+1})$. Thus, $x_{i+1}', \dots, x_m, x, x_1, \dots, x_{i-1}'$ must be a disc sequence such that either $x_{i-1}=bx_{i-1}'$ or $x_{i+1}=x_{i+1}'b$ for some $b$ and $y=\pm b$. Now $f=f_b$ and $\mu(x_m, \dots x_{i+1},f,x_{i-1},\dots, x_1)=(-1)^\circ f_x$, where $\circ$ is the sum of $\dg{x_1} + \dots + \dg{x_{m}} +\dg{f}$ and the sign which appears in the definition of $\mu_{\cF}$. Using that $|f_b|=-|b|$ and hence $(-1)^{|f_b|+|b|}=1$, we can describe the result of the higher multiplications in $\mathbb{T}$ as follows:

\begin{align}\label{HigherOps3'}\mu_{\mathbb{T}}(a_n,\dots, a_{r+1},f_b,ba_r,\dots, a_2) & =(-1)^{\dg{a_2}+\cdots+\dg{a_n}-1} \cdot f_{a_1}, & \mbox{where } ba_r\neq 0; \\
  \label{HigherOps4'}\mu_{\mathbb{T}}(a_n,\dots, a_{2},f_b)&=(-1)^{\dg{a_2}+\cdots+\dg{a_n}+\dg{f_b}} \cdot f_{ba_1}, & \mbox{where } ba_1\neq 0;\\
  \label{HigherOps5'}\mu_{\mathbb{T}}(a_n,\dots, a_{r}b,f_b,a_{r-1},\dots, a_2)&=(-1)^{\dg{a_2}+\cdots+\dg{a_n}+\dg{b}} \cdot f_{a_1}, & \mbox{where } a_{r}b\neq 0;\\
  \label{HigherOps6'}\mu_{\mathbb{T}}(f_b,a_{n-1},\dots, a_{1})&=(-1)^{\dg{a_1}+\cdots+\dg{a_{n-1}}-1} \cdot f_{a_nb}, & \mbox{where } a_nb\neq 0.\end{align}
 \smallskip
\noindent \textbf{Definition of a strict functor $\mathbb{T} \rightarrow \mathbb{B}$}: For a path $a=a_1\dots a_{l}\in Q_{\Gamma_\cA}$ which is a part of a cycle around a puncture $p$ set
\begin{equation}\label{EquationDeltaCut}
\Delta_\sC(a) \coloneqq \begin{cases} |a_{i+1}\dots a_{l}|, &  \text{if }\alpha_p^\sC= a_i; \\  0, & \text{if }\alpha_p^\sC\notin a.\end{cases}\end{equation}

\noindent In other words, $\Delta_\sC(a) \neq 0$ only if $a$ is a path in $\cF$. By Lemma \ref{LemmaDeltaP} below there is a choice of $\Delta_p\in \{0,1\}$ for each $p\in \punct$ such that the following holds: if $\gamma \in \cA=\left(Q_{\Gamma_{\cA}}\right)_0$ is an arc with end points $p, q \in \punct$ and $\alpha, \beta \in \left(Q_{\Gamma_{\cA}}\right)_1$ denotes the pair of distinct arrows which end at $\gamma$, then
\begin{equation}\label{MagicFormulaDelpap} \Delta_\sC(C_\alpha) + \Delta_p\equiv \Delta_\sC(C_\beta) + \Delta_q  + \omega_{\eta}(\gamma)\mod 2.\end{equation}

\noindent Set $\Delta_p(a):=\Delta_p$, if $\alpha_p^\sC\in a$ and $\Delta_p(a)\coloneqq 0$, otherwise. Define $$\Delta(a) \coloneqq \Delta_\sC(a)+\Delta_p(a).$$

Let us now define a strict $A_{\infty}$-functor $\Phi: \mathbb{T} \rightarrow \mathbb{B}$. Define
$\Phi(\gamma) \coloneqq \gamma$ for all $\gamma \in \mathcal{A}_{\mathsf{C}}$, where $\gamma$ denotes the corresponding arc in $\mathcal{A}$. 
For a path $a$ in $\cF$, set $\Phi(a) \coloneqq a$, and $\Phi(f_a)\coloneqq (-1)^{\Delta(a^*)} a^*$ in case $a$ is a non-trivial path when viewed as a path in the Brauer graph category. For the trivial path $e_\gamma$ at the vertex $\gamma$ set $\Phi(f_{e_\gamma})\coloneqq (-1)^{\Delta(C_\alpha)} C_\alpha=(-1)^{\Delta(C_\beta)} C_\beta$, where $\alpha,\beta$ are arrows ending at $e_\gamma$, \eqref{MagicFormulaDelpap} implies that $\Phi(f_{e_\gamma})$ is well-defined. The degrees of $a$ in $\cF$ and $\mathbb{B}$ agree. Moreover, $|aa^*|=0$ since $\eta$ is of ribbon type which implies $|a^*|=-|a|=|f_a|$. Consequently, $\Phi$ preserves the degrees of homogeneous elements. \medskip

\noindent \textbf{Comparison of compositions:} For paths $a,b$ in $\cF$ we have $\mu_\mathbb{B}(\Phi(a),\Phi (b))=\Phi\left(\mu_\mathbb{T}(a,b)\right)$ as $\mu_\mathbb{T}(a,b)$ is a morphism in $\cF$. Next, let $c,d$ and $a$ be paths such that $c=da$. Then $d^*=ac^*$ and hence $\Delta(c^*)=\Delta(d^*)$. By \eqref{FormulaMult1} we derive 
\begin{displaymath}
\begin{aligned}
\mu_\mathbb{B}\left(\Phi(a),\Phi(f_c)\right) & = (-1)^{\Delta(c^*)} \cdot \mu_\mathbb{B}(a,c^*) &  & = (-1)^{\Delta(c^*)+|c^*|} \cdot d^* \\ 
& = (-1)^{\Delta(c^*)+|c^*|+\Delta(d^*)} \cdot \Phi(f_d) & & = (-1)^{|f_c|} \cdot \Phi(f_d)\\ & =\Phi\left( \mu_\mathbb{T}(a,f_c)\right).\end{aligned}
\end{displaymath}

\noindent In case $d$ is a trivial path and $a=c$ we can use $C_\alpha=aa^*$ instead of $d^*$ to guarantee $\Delta(c^*)=\Delta(d^*)$ in the computation. In case $a,d$ and $c$ are all the same trivial path we can chose the same cycle $C_\alpha$ as both $d^*$ and $c^*$ in the computation.

\noindent Similarly, for $l,b$ and $m$ such that $l=bm$ we have $m^*=l^*b$. Hence, $\Delta_\sC(m^*)=|b|+\Delta_\sC(l^*)$ and thus $\Delta(m^*)=|b|+\Delta(l^*)$ which implies that
\begin{displaymath}
\begin{aligned}
\mu_\mathbb{B}\left(\Phi(f_l),\Phi( b)\right)=(-1)^{\Delta(l^*)} \cdot \mu_\mathbb{B}(l^*,b)
= (-1)^{\Delta(l^*)+|b|} \cdot m^* = \Phi( f_m)= \Phi\left(\mu_\mathbb{T}(f_l,b)\right).
\end{aligned}
\end{displaymath}

\noindent As before, in case $m$ is a trivial path and $l=b$ we can use $C_\alpha=l^*l$ instead of $m^*$ to guarantee $\Delta_\sC(m^*)=|b|+\Delta_\sC(l^*)$ in the computation. In case $l,b$ and $m$ are all the same trivial path we can chose the same cycle $C_\alpha$ as both $m^*$ and $l^*$ in the computation.

\noindent Since $\mu_\mathbb{T}(f_a,f_b)=0=\mu_\mathbb{B}(a^*,b^*)$ for all paths $a, b$ in $\cF$, we conclude that $\Phi$ commutes with $\mu^2$.\smallskip

\noindent \textbf{Comparison of higher operations:} Let $a_n, \dots, a_1$ be a disc sequence and let $\delta_1, \dots, \delta_n \in \cA$ denote the arcs such that for each $i \in \{1,\dots, n\}$, $a_i$ corresponds to an intersection $p_i \in \delta_i \oInt \delta_{i+1}$. We want to show that $\Phi$ commutes with $\mu^n$ on  sequences of the form 
\begin{displaymath}
a_n,\dots, a_{r+1},f_b,ba_r,\dots, a_2.
\end{displaymath}

\noindent If $ba_r\neq 0$, then $b^\ast=a_r(ba_r)^\ast$. In case $b$ is a trivial path we will use $b^\ast=a_r(ba_r)^\ast=a_ra_r^*$ as the definition of $b^*$. We compute
\begin{align*}
\mu_\mathbb{B}\left(\Phi(a_n),\dots,\Phi(f_b),\Phi(ba_r),\dots, \Phi(a_2)\right) & =(-1)^{\Delta(b^*)} \cdot  \mu_\mathbb{B}(a_n,\dots,a_r(ba_r)^*, ba_r,\dots,  a_2) \\ & =(-1)^{\Delta(b^*)+\circ} \cdot a_1^*\end{align*}

\noindent where $\circ=|a_1| + \cdots +|a_{r-1}| + |ba_r| + \sum_{i=2}^{r} \omega_{\eta}(\delta_i)$. On the other hand, \eqref{HigherOps3'} gives us
\begin{displaymath}
\Phi(\mu_{\mathbb{T}}(a_n, \dots, a_{r+1}, f_b, ba_r, \dots, a_2) = \Phi \left((-1)^{\dg{a_2}+\cdots+\dg{a_n}-1}f_{a_1}\right)=(-1)^\star \cdot a_1^*,
\end{displaymath}
\noindent where $\star \coloneqq \dg{a_2}+\cdots+\dg{a_n}-1+\Delta(a_1^*)$. As $\sum_{i=1}^n {\dg{a_i}} \equiv 0 \mod 2$ by Lemma \ref{LemmaDegreeCondition}, $\star \equiv  |a_1| + \Delta(a_1^{\ast}) \mod 2$. Our goal is to show that $(-1)^{\Delta(b^*) + \circ}=(-1)^{\star}$. In the subsequent computation we write $C_d$ for the cycle $C_{d_1}$ whenever $d=d_1\dots d_{l}$ and $\pi(d)$ for $\pi(d_l)$.  Due to \eqref{MagicFormulaDelpap} and setting $\delta_{n+1}\coloneqq \delta_1$ we find that
\begingroup
\addtolength{\jot}{0.3em}
\begin{align}\label{DeltaBBB} \Delta_{p_r}(b^*) & \equiv \Delta_{p_r}(a_r^*) \nonumber \\ & \equiv \Delta_\sC(C_{a_{r}})+\Delta_\sC(C_{\pi(a_{r+1})})+\Delta_{p_{r+1}}(a_{r+1}^*) + \omega_{\eta}(\delta_{r+1}) \nonumber \\ & \equiv \Delta_\sC(C_{a_{r}})+\Delta_\sC(C_{\pi(a_{r+1})})+\Delta_\sC(C_{a_{r+1}})+\Delta_\sC(C_{\pi(a_{r+2})})+\Delta_{p_{r+2}}(a_{r+2}^*) + \sum_{i=r+1}^{r+2} \omega_{\eta}(\delta_i) \nonumber \\ &  \vdotswithin{= } \nonumber \\ & \equiv \Delta_\sC(C_{a_r}) + \textstyle{\sum_{i=r+1}^{n}}|a_i| +  \Delta_\sC(C_{\pi(a_{1})})+\Delta_{p_1}(a_{1}^*) + \sum_{i=r+1}^{n+1}{\omega_{\eta}(\delta_i)} \mod 2. 
\end{align}
\endgroup

\noindent The last equality follows from $\Delta_\sC(C_{\pi(a_{i})})+\Delta_\sC(C_{a_{i}})\equiv |a_i|$. Since $\Delta_\sC(b^{\ast}) + |b| = \Delta_\sC(C_{a_r})$ and $\Delta_\sC(C_{a_1}) + |a_1| = \Delta_\sC(C_{\pi(a_1)})$, it follows that 

\begin{displaymath}
\Delta(b^*) \equiv \Delta_\sC(b^\ast) + \Delta_{p_r}(b^\ast) \equiv |b| + \Delta(a_1^{\ast}) + |a_1| + \sum_{i=r+1}^n |a_i| + \sum_{j=r+1}^{n+1} \omega_{\eta}(\delta_j) \mod 2.
\end{displaymath}
Finally, $\Delta(b^{\ast}) + \circ \equiv \star \mod 2$ then follows from $\sum_{i=1}^n |a_i| \equiv n \mod 2$  and $\sum_{i=1}^n \omega_{\eta}(\delta_i) \equiv n \mod 2$. This identifies \eqref{HigherOps5} and  \eqref{HigherOps3'}.

Next,  in order to compare the signs in  \eqref{HigherOps5} and  \eqref{HigherOps5'}, we need to compare the  signs $\Delta(b^{\ast}) + \circ$ and $\star$  in front of $a_1^{\ast}$. Let us consider a sequence
\begin{displaymath}
a_n,\dots, a_{r+1},a_rb,f_b,\dots, a_2.
\end{displaymath}
Here $b^*=(a_rb)^*a_r$ and we use this as the definition of $b^*$ in case $b$ is a trivial path. Now $\Delta(b^{\ast})=\Delta_\sC(C_{a_r})+|a_r|+\Delta_{p_r}(b^{\ast})$. 
$$\star= \dg{a_2}+\cdots+\dg{a_n} +\dg{b} +\Delta(a_1^{\ast})  \text{ and}$$
$$\circ=  |a_1| + \dots +|a_{r-1}|+|(a_rb)^*a_r|+\sum_{j=2}^{r} \omega_{\eta}(\delta_j).$$
Using \ref{DeltaBBB}, we get:
\begin{align*} \Delta(b^*)+ \circ +\star & \equiv \Delta_\sC(C_{a_r})+|a_r|+\Delta_\sC(C_{a_r}) + \sum_{i=r+1}^{n}|a_i| +  \Delta_\sC(C_{\pi(a_{1})})+\Delta_{p_1}(a_{1}^*) + \sum_{i=r+1}^{n+1}\omega_{\eta}(\delta_i)\\ & \phantom{\equiv} \; \;\phantom{\Delta_\sC(C_{a_r})}  +  \sum_{i=1}^{r-1}|a_i|+|(a_rb)^*a_r|+\sum_{j=2}^{r} \omega_{\eta}(\delta_j) +\sum_{i=2}^{n}\dg{a_i} +\dg{b}+ \Delta(a_1^{\ast})\\ & \equiv \sum_{i=1}^{n}|a_i| + \sum_{j=1}^{n} \omega_{\eta}(\delta_j) +\Delta_\sC(C_{\pi(a_{1})})+\Delta_{p_1}(a_{1}^*) +\Delta(a_{1}^*) +|b^*| +\sum_{i=2}^{n}\dg{a_i} +\dg{b} \\ & \equiv 0 \mod 2. \end{align*}

The last equality follows from  $\Delta_\sC(C_{\pi(a_1)})=|a_1|+\Delta_\sC(a_1^*)$.
Thus, $\Phi$ commutes with $\mu^n$ in this case as well.
 Next, we consider a sequence of morphisms of the form
\begin{displaymath}
a_n,\dots, a_2, f_b\end{displaymath}

\noindent with $ba_1\neq 0$. By \eqref{Mu2}, \eqref{HigherOps4'} and $b^*=a_1(ba_1)^*$ it is sufficient to show that  
\begin{displaymath}
\Delta(b^*)+|(ba_1)^*|=\dg{a_2} +\cdots+\dg{a_{n}}+\dg{f_b} +  \Delta((ba_1)^*),
\end{displaymath}

\noindent This follows from the equations $\Delta_\sC((ba_1)^*)= \Delta_\sC(b^*)$ and $\Delta_{p_1}((ba_1)^*)=\Delta_{p_1}(b^*)$ as well as $\sum_{i=1}^n\dg{a_i}=0\mod 2$.
Finally, we consider a sequence of the form
\begin{displaymath}
f_b,a_{n-1}\dots, a_1\end{displaymath}

\noindent  with $a_nb\neq 0$. By \eqref{Mu1}, \eqref{HigherOps6'}, and $b^*=(a_nb)^*a_n$, we have to prove
\begin{displaymath}
\Delta(b^*) =\dg{a_1} +\cdots+\dg{a_{n-1}}-1+ \Delta((a_nb)^*)\end{displaymath}
\noindent which follows again from $\sum_{i=1}^n\dg{a_i}=0\mod 2$ and the fact that $\Delta_\sC((a_nb)^*) + |a_n|= \Delta_\sC(b^*)$ as well as $\Delta_{p_n}((a_nb)^*)= \Delta_{p_n}(b^*)$.

Higher operations vanish in both $\mathbb{T}$ and $\mathbb{B}$ on all  sequences which are not of any of the above forms. Thus, $\Phi$ is an $A_\infty$-functor. It is clear that it induces isomorphisms between the corresponding morphism spaces in $\mathbb{T}$ and $\mathbb{B}$.
\end{proof}

\noindent The following lemma was used in the proof of Proposition \ref{PropositionTrivialExtensionIsomorphism}.

\begin{lem}\label{LemmaDeltaP}
Let $(\cA, \sC)$ be a graded cutting pair on a punctured graded surface $(\Sigma, \eta)$ of ribbon type. Then, there is a family $(\Delta_p)_{p \in \punct}$ with $\Delta_p \in \{0,1\}$ such that for all edges $\gamma \in \cA=\BE{\Gamma_{\cA}}$ the following holds: if $\alpha, \beta$ denotes the pair of distinct arrows  in $Q_{\Gamma_{\cA}}$ which end at $\gamma$ and $p$ and $q$ denote the end points of $\gamma$, then
\begin{equation}\label{EquationDeltaConstraint} \Delta_{\sC}(C_{\alpha}) + \Delta_{p} \equiv  \Delta_{\sC}(C_{\beta}) + \Delta_{q} +\omega_\eta(\gamma) \mod 2.\end{equation}
\end{lem}
\begin{proof} Set $\Gamma=\Gamma_{\cA}$. Choose any vertex $v \in \BV{\Gamma}=\punct$ and set $\Delta_v \coloneqq 0$. For any vertex $u \in \BV{\Gamma}$, choose a path  $\gamma^u= \gamma_{l_u}^u \cdots \gamma_1^u$ in $\Gamma$ from $v$ to $u$ with oriented edges $\gamma_i^u \in \cA=\BE{\Gamma}$. Define $\Delta_u$ as the residue modulo $2$ of
\begin{displaymath}
\sum_{h \in A_{\gamma^u}}\Delta_{\sC}(C_{\alpha_h}) + \sum_{j=1}^{l_u}\omega_\eta(\gamma_j^u),\end{displaymath}

\noindent where $A_{\gamma^u}$ denotes the set of all half-edges in $\Gamma$ along  $\gamma^u$ and where $\alpha_h$ denotes the arrow in $Q_{\Gamma}$ which corresponds to the pair $(h^-,h)$. We claim that $\Delta_u$ is well-defined, i.e.\ does not depend on the choice of $\gamma^u$. Suppose that $\gamma'$ is another path which starts in $v$ and ends in $u$. Then, $\gamma^u \overline{\gamma'}$ is closed and starts at $v$. It therefore suffices to show that $\sum_{h \in A_{\delta}}\Delta_{\sC}(C_{\alpha_h}) + \sum_{j=1}^{l}\omega_\eta(\delta_j) \equiv 0 \mod 2$  for all closed paths $\delta=\delta_{l} \cdots \delta_1$ which start at $v$. Let us assume that $\delta$ is closed and based at $v$. We observe that for consecutive half-edges $h\in \delta_{i}$ and $h'\in \delta_{i+1}$, such that $s(h)=s(h')$, $\Delta_{\sC}(C_{\alpha_h})+\Delta_{\sC}(C_{\alpha_{h'}})=\deg(p_i)$, where $p_i$ is the unique oriented intersection of $\overline{h} \in \cA$ and $\overline{h'} \in \cA$ whose  sector does not contain the cutting path $c_{p_i}$ and which is induced by an oriented intersection of $h$ and $h'$. 
The collection of all these oriented intersections defines a cycle $(p_i)_{i \in \mathbb{Z}_{l}}$. Corollary \ref{CorollaryDegreeCondition} and Corollary \ref{CorollaryWindingEdges} imply
\begin{displaymath}
\sum_{h \in A_{\delta}}\Delta_{\sC}(C_{\alpha_h}) +  \sum_{i=1}^{l}\omega_\eta(\delta_i) \equiv \sum_{i \in \mathbb{Z}_l}{\deg(p_i)} +\sum_{i=1}^{l}\omega_\eta(\delta_i) \equiv 0 \mod 2.
\end{displaymath} 
\noindent  This shows that $\Delta_u$ is well-defined. Finally, \eqref{EquationDeltaConstraint} follows from the definition of $\Delta_p$ and $\Delta_q$ by choosing $\gamma^p$ and $\gamma^q$ such that $\gamma^p=\gamma\gamma^q$ for a suitable orientation of $\gamma$.\end{proof}

\section{Brauer graph categories of branched covers and their orbit categories}\label{SectionBranchedCovers}
\noindent We have proved in Section \ref{SectionBrauerGraphCategories} that any Brauer graph category with trivial multiplicities on a graded surface $\Sigma$ with an arc system $\cA$ is the trivial extension of a  category $\mathcal{F}_{\cA_{\mathsf{C}}}$ for a suitable cut $\mathsf{C}$ of $(\Sigma, \cA)$. We prove a similar statement for non-trivial multiplicities by showing that such Brauer graph categories arise as orbit categories of trivial extensions associated with suitable branched covers of $\Sigma$.
\subsection{Branched covers and orbit categories of Brauer graph categories}

\subsubsection{A short reminder on branched covers}\ \medskip

\noindent We recall a few basic facts about branched covers on surfaces. 

\begin{definition}Let $d \in \mathbb{Z}$ and let $\Sigma$ be a punctured surface with punctures $\mathscr{P}$.
A \textbf{$d$-fold branched cover} of $\Sigma$ is a punctured surface $(\tilde{\Sigma}, \tilde{\mathscr{P}})$ together with a surjective map $\Theta:\tilde{\Sigma} \rightarrow \Sigma$ such that $\Theta^{-1}(\mathscr{P})=\tilde{\punct}$ and which restricts to a $d$-fold covering map $\tilde{\Sigma}_{\reg} \rightarrow \Sigma_{\reg}$. 
\end{definition}
\noindent In particular, locally around each puncture, any $d$-fold branched cover $\Theta$ is assembled from maps $\mathbb{C} \rightarrow \mathbb{C}$ of the form $z \mapsto z^n$. We recall that by our conventions all punctured surfaces are connected. In particular, $\tilde{\Sigma}$ is connected. By convention we will also assume that  $\tilde{\Sigma}$ is equipped with the unique orientation such that the restriction of $\Theta$ to $\tilde{\Sigma}_{\reg}$ is orientation-preserving.

\begin{definition}Let $\Theta:\tilde{\Sigma} \rightarrow \Sigma$ be a branched covering.
A \textbf{deck transformation} of $\Theta$ is a diffeomorphism $f: \tilde{\Sigma} \rightarrow \tilde{\Sigma}$ such that $\Theta \circ f= \Theta$.
\end{definition}

\noindent Deck transformations preserve the set of punctures and the orientation (the proof is the same as for non-branched coverings). The group of all deck transformations of a branched cover $\Theta$ will be denoted by $\Deck(\Theta)$. The group of deck transformations acts on the fiber over any point and a branched cover $\Theta$ is called \textbf{regular} if this action is transitive. In particular, any regular $d$-fold branched cover of a punctured surface $\Sigma$ is locally of the form $z \mapsto z^n$, where  $n|d$.\medskip

\noindent In Section \ref{SectionHigherMultiplicityBranchedCover} we define a regular branched cover of $\Sigma$ for every admissible arc system and every multiplicity function. The orders of ramification of this cover will be given by the multiplicities.\ \medskip

\noindent The following seems to be well-known but was included in lack of a suitable reference.

\begin{prp}
The restriction $f \mapsto f|_{\tilde{\Sigma}_{\reg}}$ induces an isomorphism between $\Deck(\Theta)$ and the group of deck transformations of the cover $\tilde{\Sigma}_{\reg} \rightarrow \Sigma_{\reg}$.
\end{prp}
\begin{proof}
Let $g$ be a deck transformation  of $\Theta|_{\tilde{\Sigma}_{\reg}}$. Let $p \in \punct$ and let $p \in D$ be an open disc neighbourhood such that for $D_p=D\setminus \{p\}$, $U \coloneqq \Theta^{-1}(D_p) \cong \bigsqcup_{i=1}^m D_p$. Then, $g$ restricts to a diffeomorphism $U \rightarrow U$ and we may choose coordinates such that on each connected component $W$ of  $U$, $\Theta|_W$ corresponds to a map $\mathbb{C}^{\times} \rightarrow \mathbb{C}^{\times}$, $z \mapsto z^n$. The assertion follows from the fact that  the map $g|_{W}$ (considered as a diffeomorphism $\mathbb{C}^{\times} \rightarrow \mathbb{C}^\times$) must be of the form $z \mapsto \zeta \cdot z$ where $\zeta$ is an $n$-th root of unity. Hence $g|_W$ admits a unique smooth extension to a map $h: \mathbb{C} \rightarrow \mathbb{C}$ which is compatible with $\Theta|_W$.
\end{proof}

\subsubsection{Action of deck transformations on Brauer graph categories}\ \medskip

\noindent We consider group actions on Brauer graph categories which arise from branched covers. Suppose we are given a graded punctured surface $(\Sigma, \eta)$ of ribbon type together with an $\eta$-graded admissible arc system $\cA$ and a branched cover $\Theta: \tilde{\Sigma} \rightarrow \Sigma$. We make the following observations:
\begin{enumerate}

    \item Let $\tilde{\cA}$ denote the set of all lifts of arcs in $\mathcal{A}$ along $\Theta$. Then $\tilde{\cA}$ is an admissible arc system on $\tilde{\Sigma}$ as can be seen from lifting the cutting paths for $\cA$. Moreover, $\Deck(\Theta)$ acts freely on the set $\tilde{\cA}$.

    \item The line field $\eta$ lifts to a unique $\Deck(\Theta)$-invariant line field $\tilde{\eta}$ of ribbon type on $\tilde{\Sigma}$ and the $\eta$-gradings on the arcs of $\mathcal{A}$ can be lifted canonically to $\tilde{\eta}$-gradings of the arcs in $\tilde{\mathcal{A}}$.

\end{enumerate}

\noindent Our goal is to define an action of $\Deck(\Theta)$ on the Brauer graph category $\mathbb{B}(\tilde{\cA}, \mathbf{1})$ in the following sense.
  
\begin{definition}\label{DefinitionLeftAction}Let $\mathbb{A}$ be an $A_{\infty}$-category. A \textbf{free left action} of a group $G$ on $\mathbb{A}$ is a free left action of $G$ on $\Ob{\mathbb{A}}$ together with a $\Bbbk$-linear action of $G$ on $\bigoplus_{X,Y \in \Ob{\mathbb{A}}} \Hom_\mathbb{A}(X, Y)$, which maps a homogeneous element  $a \in \Hom_\mathbb{A}(X,Y)$ to a homogeneous element $g.a \in \Hom_\mathbb{A}(g.X, g.Y)$ of the same degree, such that all operations $\mu_{\mathbb{A}}^m$ are $G$-equivariant, i.e.\ for all $g \in G$ and all $m \geq 1$,
	\begin{equation}\label{EquationOperationsFreeAction}
		g.\mu_{\mathbb{A}}^m(a_m, \dots, a_1) = \mu_{\mathbb{A}}^m(g.a_m , \dots, g.a_1).
	\end{equation}
\end{definition}

\begin{lem}\label{LemmaBranchedCoverInducesCategoricalAction} Let $(\Sigma, \eta)$ be a surface of ribbon type. Let $\Theta: \tilde{\Sigma} \rightarrow \Sigma$ be a  branched cover with the line field $\tilde{\eta}$ and let  $\tilde{\mathcal{A}}$ be the lift of an admissible, $\eta$-graded arc system $\cA$ with its canonical $\tilde{\eta}$-grading. Then, $\mathbb{B}(\tilde{\mathcal{A}}, \mathbf{1})$  admits a free left action of $\Deck(\Theta)$.
\end{lem}
\begin{proof}Set $\mathbb{B}=\mathbb{B}(\tilde{\mathcal{A}}, \mathbf{1})$. The group $\Deck(\Theta)$ acts freely on $\tilde{\mathcal{A}}$ and hence on $\Ob{\mathbb{B}}=\{X_{\delta}\}_{\delta \in \tilde{\cA}}$.
Given arcs $\gamma, \delta \in \tilde{\mathcal{A}}$, a homogeneous basis of $\Hom_\mathbb{B}(X_\gamma, X_\delta)$ is in bijection  with paths in the modified Brauer graph algebra $B_{\tilde{\cA}}$ with trivial multiplicities associated with the ribbon graph $\Gamma_{\tilde{\mathcal{A}}}$. Moreover, every such path corresponds to an oriented intersection of $\gamma$ and $\delta$, a trivial path (in case $\gamma=\delta$) or a full cycle $C_\alpha=(-1)^{\omega_{\tilde{\eta}}(\gamma)}C_\beta$, where $\alpha$ and $\beta$ are two arrows which end at $\gamma$ and correspond to half-edges $h$ and $h'$. Every deck transformation $f$ preserves the orientation and hence induces a bijection between $\gamma \oInt \delta$ and  $(f.\gamma) \oInt (f.\delta)$. It maps $\gamma$ to $f.\gamma$,  $h$ to $f.h$ and $h'$ to $f.h'$. Thus, it induces a bijection between trivial paths and between full cycles sending $C_\alpha$ to $C_{f.\alpha}$, where $f.\alpha$ corresponds to $f.h$.  Since $\tilde{\eta}$ is $\Deck(\Theta)$-invariant it follows that $f$ induces a graded $\Bbbk$-linear isomorphism $\Hom_{\bB}(X_\gamma, X_\delta) \rightarrow \Hom_{\bB}(X_{f.\gamma}, X_{f. \delta})$ of degree $0$ which is compatible with composition. The relations from the definition of the modified Brauer graph algebra are equivariant under this action since $\omega_{\tilde{\eta}}(\gamma)=\omega_{\tilde{\eta}}(f.\gamma)$. Finally, \eqref{EquationOperationsFreeAction} is a consequence of the fact that $(a_m, \dots, a_1)$ is a disc sequence if and only if $(g.a_m, \dots, g.a_1)$ is a disc sequence.
\end{proof}

\subsubsection{Orbit categories of group actions}\ \medskip

\noindent  An $A_{\infty}$-structure can be ``pushed down'' along any free group action to form a new $A_{\infty}$-category classically known as the \textit{orbit category}.

\begin{definition}
 The \textbf{orbit category} of a free left action by a  group $G$ on $\mathbb{A}$ is an $A_{\infty}$-category $\mathbb{A}/G$ such that
$\Ob{\mathbb{A}/G}=\Ob{\mathbb{A}}$ and for all $X,Y \in \Ob{\mathbb{A}}$,
$$
\Hom_{\mathbb{A}/G}(X, Y) \coloneqq \bigoplus_{g \in G} \Hom_\mathbb{A}(X, g.Y).
$$

\noindent The $A_{\infty}$-operations on $\mathbb{A}/G$ are defined as follows: let 
$$a_m \otimes \cdots \otimes a_1 \in \Hom_{\mathbb{A}/G}(X_{m-1}, X_{m}) \otimes \cdots\otimes \Hom_{\mathbb{A}/G}(X_0, X_1)$$
\noindent be such that for each $i \in [1,m]$, $a_i \in \Hom(X_{i-1}, g_i. X_i)$ with $g_i \in G$.
Then,
\begin{displaymath}
\mu_{\mathbb{A}/G}(a_m, \dots, a_1) \coloneqq \mu_{\mathbb{A}}\big( (g_1 \cdots g_{m-1}).a_m, \dots,(g_1g_2).a_3, g_1.a_2 ,a_1 \big) \in \Hom_{\mathbb{A}}(X_0, (g_1 \cdots g_m). X_m).
\end{displaymath}
\end{definition}

\begin{lem}
$\mathbb{A}/G$ is an $A_{\infty}$-category.
\end{lem}
\begin{proof}
With the notation $u_a^b= g_a g_{a+1} \cdots g_b$ for $ a \leq b$ one has $u_a^b \cdot u_{b+1}^c= u_a^c$. Then,

	\begin{equation}
	\begin{split}
	 & \mu_{\mathbb{A}/G}\big( a_m, \dots, a_{i+1}, \mu_{\mathbb{A}/G}(a_{i}, \dots, a_{j}), a_{j-1}, \dots, a_1\big)  \\  \stackrel{\phantom{\eqref{EquationOperationsFreeAction}}}{=} \, & \mu_{\mathbb{A}/G}(a_m, \dots, a_{i+1} , \, \mu_{\mathbb{A}}\big( u_j^{i-1}.a_i, \dots, g_j.a_{j+1}, a_j\big), \, a_{j-1}, \dots, a_1\Big) \\  
	 \stackrel{\phantom{\eqref{EquationOperationsFreeAction}}}{=} \, & \mu_{\mathbb{A}}\Big( u_1^{m-1}.a_m, \dots, u_1^i. a_{i+1}, \, u_1^{j-1}.\mu_{\mathbb{A}}\big( u_j^{i-1}.a_i, \dots, g_j.a_{j+1}, a_j\big), \,  u_1^{j-2}.a_{j-1}, \dots, a_1\Big) \\
	  \stackrel{\eqref{EquationOperationsFreeAction}}{=} \, & \mu_{\mathbb{A}}\Big( u_1^{m-1}.a_m, \dots, u_1^i. a_{i+1}, \, \mu_{\mathbb{A}}\big( u_1^{i-1}.a_i, \dots, u_1^{j-1}.a_j\big), \, u_1^{j-2}.a_{j-1}, \dots, a_1\Big).
	  \end{split}
	\end{equation}
	
	\noindent Since $\dg{g.f}=\dg{f}$ for all homogeneous morphisms $f$ and all $g \in G$, the $A_{\infty}$-constraints follow from the constraints for $\mathbb{A}$.
\end{proof}

\noindent Two objects $X, Y$ in $\mathbb{A}/G$ are isomorphic if there exists $g \in G$ and an isomorphism $X \rightarrow g.Y$ in $\mathbb{A}$. One concludes:

\begin{cor}\label{CorollarySkeletonMoritaEquivalence}
Let $\bA$ be an $A_{\infty}$-category with a free left action of a  group $G$ and let $S$ be a  complete set of representatives of the $G$-orbits of $\Ob{\mathbb{A}}$. Then, the inclusion $S \hookrightarrow \mathbb{A}/G$ of the corresponding full $A_{\infty}$-subcategory is a Morita equivalence.
\end{cor}

\subsection{Multiplicity-free covers of Brauer graph algebras} \ \medskip

\noindent We recall the construction of a multiplicty-free cover of a Brauer graph from \cite{GreenSchrollSnashall} (see also \cite{Asashiba}). 
It associates to any Brauer graph $(\Gamma, \seq)$ a multiplicity-free Brauer graph. We will see that the ribbon surface of the new Brauer graph is a branched cover of the original ribbon surface which is ramified over every vertex $v$ of $\Gamma$ with $\seq(v)\neq 1$.  \medskip

\noindent In what follows, we identify a multiplicity function with the corresponding function $\seq:\BH{\Gamma} \rightarrow \mathbb{Z}$ which is determined by the formula 
$$\seq(h)=\seq(s(h)).$$

\noindent The number $\overline{\seq}$ shall henceforth denote the least common multiple of all values of $\seq$. By a \textbf{cut} of $\Gamma$, we mean a subset $\sC \subseteq \BH{\Gamma}$ such that  $\sC \cap \BH{\Gamma}_v$ consists of a single element for all $v \in \BV{\Gamma}$.
\begin{definition}Let $(\Gamma, \seq)$ be a Brauer graph and let $\mathsf{C}$ be a cut of $\Gamma$.

\begin{itemize}
    \item Let $W=W_\sC: \BH{\Gamma} \rightarrow \mathbb{Z}_{\overline{\seq}}$ denote the  function defined by 
\[W(h) \coloneqq \begin{rightalignedcases}
   \overline{\left(\frac{\overline{\seq}}{\seq(h)}\right)}, & \text{if } h \in \sC; \\
   \overline{0}, & \text{otherwise}.
\end{rightalignedcases}\]

      \item Denote by $\Gamma^{\seq}=\Gamma^{\seq}_\sC$ the ribbon graph which is defined by the following data:
      \begin{enumerate} 
      
      \item $\BV{\Gamma^{\seq}}$ consists of all pairs $(v, g+U_v)$, where $v\in \BV{\Gamma}$ and $g + U_v$ is a coset of the subgroup $U_v \subseteq \mathbb{Z}_{\overline{\seq}}$ which is generated by the residue class of $\frac{\overline{\seq}}{\seq(v)}$;
          \item $\BH{\Gamma^{\seq}} = \BH{\Gamma} \times \mathbb{Z}_{\overline{\seq}} \coloneqq \{h_g \, | \, h \in \BH{\Gamma}, g \in \mathbb{Z}_{\overline{\seq}}\}$;
          \item $\iota(h_g) \coloneqq \iota(h)_g$ and $\sigma(h_g) \coloneqq \sigma(h)_{g+W(h)}$;
          \item $s(h_g) \coloneqq (s(h), g +U_{s(h)})$. 
      \end{enumerate}
\end{itemize} 

\noindent Note that $s(h_g)=s(\sigma(h_g))$. The Brauer graph $(\Gamma^{\seq}, \mathbf{1})$ is called a \textbf{multiplicity-free cover} of $(\Gamma, \seq)$.
\end{definition}

\noindent The function $W_{\mathsf{C}}$ is referred to as the \textbf{weight function}. We collect a few observations about the multiplicity-free cover.

\begin{enumerate}
    \item There exists a canonical projection 
    \begin{equation}\label{EquationMapPi}\Theta:\Gamma^{\seq} \rightarrow \Gamma.\end{equation} 
    \noindent It maps a vertex $(v, g + U_v)$ to the vertex $v$ and a half-edge $h_g$ to the half-edge $h$. In particular, $\Theta$ commutes with the functions $\sigma$, $\iota$ and $s$ and hence is a morphism of ribbon graphs. It induces a map $\BF{\Gamma^{\seq}} \rightarrow \BF{\Gamma}$ which maps a face $(h_1, \dots, h_{2m}) \in \BF{\Gamma^{\seq}}$ to the unique face $(l_1, \dots, l_{2d}) \in \BF{\Gamma}$ such that $d \, | \, m$ and such that $(\Theta(h_1), \dots, \Theta(h_{2m}))=(l_1, \dots, l_{2d}, l_1, \dots, l_{2d}, \dots,l_1, \dots, l_{2d})$.
    \item The ribbon graph $\Gamma^{\seq}$ carries a canonical $\mathbb{Z}_{\overline{\seq}}$-action, c.f.\ \cite[Lemma 3.7.]{GreenSchrollSnashall}. An element $g \in \mathbb{Z}_{\overline{\seq}}$ acts by the maps
    \begin{displaymath}
    \begin{tikzcd}[row sep=0.5ex]
        \BV{\Gamma^{\seq}} \arrow{r} & \BV{\Gamma^{\seq}}, & & \BH{\Gamma^{\seq}} \arrow{r} & \BH{\Gamma^{\seq}}, \\
        (v, f + U_v)  \arrow[mapsto]{r} & (v, (g+f) + U_v), & & h_f \arrow[mapsto]{r} & h_{g+f}.
    \end{tikzcd}\end{displaymath}
\noindent The functions $\iota, \sigma$ and $s$ are $\mathbb{Z}_{\overline{\seq}}$-equivariant and hence $g$ defines an automorphism of $\Gamma^{\seq}$. Note that  since $\Gamma^{\seq}$ is connected,
any ribbon graph automorphism $f: \Gamma^{\seq} \rightarrow \Gamma^{\seq}$ such that $\Theta \circ f=\Theta$ is of the above form. Similar to the case of deck transformations this follows from the observation that every such automorphism is determined by its value on a single half-edge.
\end{enumerate}    

\begin{cor}
The action of $\mathbb{Z}_{\overline{\seq}}$ on $\BE{\Gamma^{\seq}}$ is free and $U_v$ is the stabilizer of $(v, f+ U_v) \in \BV{\Gamma^{\seq}}$.\end{cor}
\noindent Note that $U_v$ is of order $\seq(v)$ and that $|H_{(v,f+U_v)}|=|H_v|\cdot \seq(v)$ for all $f\in \mathbb{Z}_{\overline{\seq}}$.

\begin{exa}\label{ExampleMultiplicityFreeCover}
We consider the multiplicity-free cover of the triangle from Example \ref{FigureExampleBGA}. The non-trivial entries of the weight function are indicated by arrows between half-edges, i.e.\ an arrow corresponding to a pair $(h, h^+)$ of half-edges has label $W(h)$.
\begin{displaymath}
\begin{tikzpicture}

\def\r{1};
\def\thickness{3pt}; 
\def\factor{1.3}; 
\def\rtwo{0.75}; 
\def\posi{0.75}; 
\def\dist{-7pt};

 \foreach \i in {1,2,3}
 \draw ({\factor*\r*cos(330-\i*120)},{\factor*\r*sin(330-\i*120)}) node[black]{$\scriptstyle \i$};
 
 \foreach \i in {1,2,3}
 \draw ({\r*cos(330-\i*120)},{\r*sin(330-\i*120)})--({\r*cos(210-\i*120)},{\r*sin(210-\i*120)});

 \draw[fill=white] ({\r*cos(-30)},{\r*sin(-30)}) circle (\thickness);
 \filldraw[black] ({\r*cos(90)},{\r*sin(90)}) circle (\thickness);
 \filldraw[orange] ({\r*cos(210)},{\r*sin(210)}) circle (\thickness);
 
  \foreach \i in {2,3}
  {
\pgfmathparse{int(Mod(int(6/ \i), 6))};
       \let\theIntINeed\pgfmathresult;
 \draw[->] ({\r*cos(330-\i*120)+\rtwo*cos(180-\i*120)},{\r*sin(330-\i*120)+\rtwo*sin(180-\i*120)}) arc ({180-\i*120}:{210-\i*120+270}:{\rtwo}) node[pos=\posi, label={[label distance=\dist]{180-\i*120+\posi*300}:$\overline{  \theIntINeed}$}] {};
 }
\end{tikzpicture}
\end{displaymath}
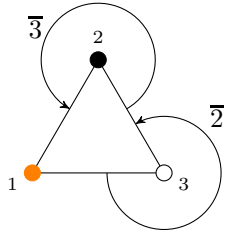
\noindent   Figure \ref{FigureMultipicityFreeCover} shows the two ``halfs'' of its multiplicity-free cover $\Gamma^\seq$. The ribbon surface of $\Gamma^{\seq}$ is a torus.  The three elements in the fiber of the vertex $\bullet$ are labelled by $\{\overline{0}, \overline{3}\}, \{\overline{1}, \overline{4}\}$ and $\{\overline{2}, \overline{5}\}$ and are obtained by identification of pairs of vertices with the same labels on the left and on the right hand side of Figure \ref{FigureMultipicityFreeCover}. 
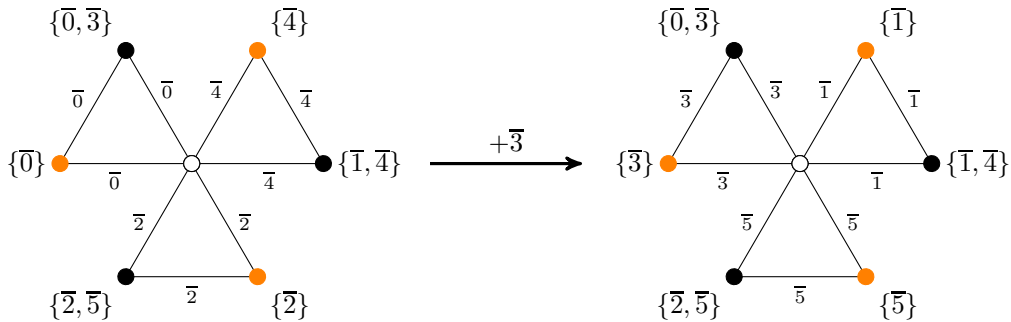
\begin{figure}[h]
\begin{displaymath}
 \begin{tikzpicture}
 
\def\r{1};
\def\thickness{3pt}; 
\def\factor{1.3}; 
\def\rtwo{0.75}; 
\def\posi{0.75}; 
\def\dist{-7pt};

\begin{scope}

\foreach \j in {1,2,3}
{
\begin{scope}[rotate around={\j*120:({\r*cos(-30)},{\r*sin(-30)})}]


\foreach \i in {2}
{
        \pgfmathparse{int(Mod(int(6-\j), 6))};
       \let\CosetOne\pgfmathresult;
       \pgfmathparse{int(Mod(int(\CosetOne+3), 6))};
       \let\CosetTwo\pgfmathresult;
  \draw ({\r*cos(330-\i*120)},{\r*sin(330-\i*120)}) node[label={[label distance=-2pt]{-120+\i*120+\j*120}:$\{\overline{\CosetTwo}, \overline{\CosetOne} \}$}] {};
 }

\foreach \i in {1}
{
        \pgfmathparse{int(Mod(int((\j-1)*2+2), 6))};
       \let\CosetOne\pgfmathresult;
  \draw ({\r*cos(330-\i*120)},{\r*sin(330-\i*120)}) node[anchor=center, label={[label distance=-2pt]{60+\i*120+\j*120}:$\{\overline{\CosetOne} \}$}] {};
 }
 
 \foreach \i in {1,2, 3}
 {
  \pgfmathparse{int(Mod(int((\j-1)*2+2), 6))};
       \let\lab\pgfmathresult;
 \draw ({\r*cos(330-\i*120)},{\r*sin(330-\i*120)})--({\r*cos(210-\i*120)},{\r*sin(210-\i*120)}) ;

 }
 
 
 \draw[fill=white] ({\r*cos(-30)},{\r*sin(-30)}) circle (\thickness);
 \filldraw[black] ({\r*cos(90)},{\r*sin(90)}) circle (\thickness);
 \filldraw[orange] ({\r*cos(210)},{\r*sin(210)}) circle (\thickness);
 \end{scope} 
 }
 
\begin{scope}[shift={({\r*cos(-30},{\r*sin(-30)})}]
\foreach \i in {0, 1, 2}
{ 
\pgfmathparse{int(Mod(int(4+2*\i), 6))};
       \let\lab\pgfmathresult;
       \def\off{12};
       \def\fac{0.6};
 \draw ({sqrt(3)*\fac*\r*cos(60+\i*120+\off)},{sqrt(3)*\fac*\r*sin(60+\i*120+\off)}) node[anchor=center]{$\scriptstyle \overline{\lab}$};
 
 \pgfmathparse{int(Mod(int(2*\i), 6))};
       \let\lab\pgfmathresult;
 \draw ({sqrt(3)*\fac*\r*cos(120+\i*120-\off)},{sqrt(3)*\fac*\r*sin(120+\i*120-\off)}) node[anchor=center]{$\scriptstyle \overline{\lab}$};

 \pgfmathparse{int(Mod(int(2*\i+4), 6))};
       \let\lab\pgfmathresult;
 \draw ({sqrt(3)*\r*cos(30+\i*120)},{sqrt(3)*\r*sin(30+\i*120)}) node[anchor=center]{$\scriptstyle \overline{\lab}$};
}
\end{scope}
\end{scope}


\begin{scope}[shift={(8,0)}]

\draw[very thick, ->] (-4,{\r*sin(-30)})--(-2,{\r*sin(-30)}) node[pos=0.5, above]{$+ \overline{3}$};
\foreach \j in {1,2,3}
{
\begin{scope}[rotate around={\j*120:({\r*cos(-30)},{\r*sin(-30)})}]


\foreach \i in {2}
{
        \pgfmathparse{int(Mod(int(6-\j-3), 6))};
       \let\CosetOne\pgfmathresult;
       \pgfmathparse{int(Mod(int(\CosetOne+3), 6))};
       \let\CosetTwo\pgfmathresult;
  
  \draw ({\r*cos(330-\i*120)},{\r*sin(330-\i*120)}) node[label={[label distance=-2pt]{-120+\i*120+\j*120}:$\{\overline{\CosetOne}, \overline{\CosetTwo} \}$}] {};
 }

\foreach \i in {1}
{
        \pgfmathparse{int(Mod(int((\j-1)*2+2-3), 6))};
       \let\CosetOne\pgfmathresult;
  \draw ({\r*cos(330-\i*120)},{\r*sin(330-\i*120)}) node[anchor=center, label={[label distance=-2pt]{60+\i*120+\j*120}:$\{\overline{\CosetOne} \}$}] {};
 }
 
 \foreach \i in {1,2, 3}
 {
 \draw ({\r*cos(330-\i*120)},{\r*sin(330-\i*120)})--({\r*cos(210-\i*120)},{\r*sin(210-\i*120)}) ;

 }

 \draw[fill=white] ({\r*cos(-30)},{\r*sin(-30)}) circle (\thickness);
 \filldraw[black] ({\r*cos(90)},{\r*sin(90)}) circle (\thickness);
 \filldraw[orange] ({\r*cos(210)},{\r*sin(210)}) circle (\thickness);
 \end{scope} 
 }
 
\begin{scope}[shift={({\r*cos(-30},{\r*sin(-30)})}]
\foreach \i in {0, 1, 2}
{
\pgfmathparse{int(Mod(int(4+2*\i-3), 6))};
       \let\lab\pgfmathresult;
       \def\off{12};
       \def\fac{0.6};
 \draw ({sqrt(3)*\fac*\r*cos(60+\i*120+\off)},{sqrt(3)*\fac*\r*sin(60+\i*120+\off)}) node[anchor=center]{$\scriptstyle \overline{\lab}$};
 
 \pgfmathparse{int(Mod(int(2*\i-3), 6))};
       \let\lab\pgfmathresult;
 \draw ({sqrt(3)*\fac*\r*cos(120+\i*120-\off)},{sqrt(3)*\fac*\r*sin(120+\i*120-\off)}) node[anchor=center]{$\scriptstyle \overline{\lab}$};

 \pgfmathparse{int(Mod(int(2*\i+4-3), 6))};
       \let\lab\pgfmathresult;
 \draw ({sqrt(3)*\r*cos(30+\i*120)},{sqrt(3)*\r*sin(30+\i*120)}) node[anchor=center]{$\scriptstyle \overline{\lab}$};
}
\end{scope}
\end{scope}

 \end{tikzpicture}
\end{displaymath}
    \caption{Vertices are labelled by cosets in $\mathbb{Z}_6$; an edge with label $g$ represents an edge $e_g$.}
    \label{FigureMultipicityFreeCover}
\end{figure}
The cyclic order on the edges around the vertex $\{\overline{0},\overline{3}\}$ (and analogously for other vertices in the fiber) is given by the following cyclic order on their respective ends: $(\{\overline{0}\}, \{\overline{0}, \overline{2}, \overline{4}\}, \{\overline{3}\}, \{\overline{1}, \overline{3}, \overline{5}\})$. The cyclic order around all other vertices can be read off from Figure \ref{FigureMultipicityFreeCover}. The involution corresponding to $\overline{3} \in \mathbb{Z}_6$ interchanges the left and the right hand side of Figure \ref{FigureMultipicityFreeCover} and the element $\overline{2} \in \mathbb{Z}_6$ acts by rotation around $\circ$-vertices.
\end{exa}

\subsection{Brauer graph categories of higher multiplicity via multiplicity-free covers}\label{SectionHigherMultiplicityBranchedCover} \ \medskip

\noindent Suppose that $\mathcal{A}$ is an admissible arc system on a punctured surface $\Sigma$ and let $\seq$ be a multiplicity function. Set $\Gamma \coloneqq \Gamma_{\mathcal{A}}$. Then any cut $\sC$ of $\cA$ determines a cut of $\Gamma$ and hence a multiplicity-free cover $\Gamma^{\seq}=\Gamma^{\seq}_{\sC}$. This in turn gives rise to a punctured surface $\Sigma^{\seq}$ by gluing discs to all those boundary components on the ribbon surface of $\Gamma^{\seq}$ which correspond to the preimage of a false face under the map $\BF{\Gamma^{\seq}} \rightarrow \BF{\Gamma}$. In what follows we identify $\Gamma$ with $|\mathcal{A}| \subseteq \Sigma$ and $\Gamma^{\seq}$ with $|\mathcal{A}^{\seq}| \subseteq \Sigma^{\seq}$, where $\mathcal{A}^{\seq}$ denotes the admissible arc system on $\Sigma^\seq$ corresponding to $\Gamma^{\seq}$.

\begin{lem}\label{LemmaBranched}
Let $\Sigma, \Sigma^{\seq}$, $\Gamma$ and $\Gamma^{\seq}$ be as above. Then, for a suitable smooth structure on $\Sigma^{\seq}$,
the projection $\Theta: \Gamma^{\seq} \rightarrow \Gamma$ from \eqref{EquationMapPi} extends to a regular branched cover $\Theta_{\Sigma}:\Sigma^{\seq} \rightarrow \Sigma$ with $\Deck(\Theta_{\Sigma}) \cong  \mathbb{Z}_{\overline{m}}$.
\end{lem}
\begin{proof}
Every face $F=(h_1, \dots, h_{2m}) \in \BF{\Gamma^{\seq}}$ defines a continuous map $u_{F}: S^1 \rightarrow \Gamma^{\seq}$ which maps every $m$-th root of unity to a vertex of $\Gamma^{\seq}$ and which maps the open segment on the circle between two consecutive roots of unity homeomorphically onto the arc $\overline{h_{2j}} \subseteq |\cA|$ with its end points removed. In the same way, every face $G \in \BF{\Gamma}$ defines a map $u_G: S^1 \rightarrow \Gamma$. Note that for every $F \in \BF{\Gamma^{\seq}}$, there exists $d=d_{F} \geq 1$ such that $u_{\Theta(F)}(z^d)=\Theta \circ u_{F}(z)$ for all $z \in S^1$ for a suitable choice of the maps $u_G: S^1 \rightarrow \Gamma$. Note that $d_F=1$ whenever $F$ is a false face. This is due to the fact that the weight function $W_{\mathsf{C}}$ is trivial on all half-edges which do not belong to $\mathsf{C}$: given $h \in \BH{\Gamma}$ such that the pair $(h,h^+)$ is a part of a false face and $g \in \mathbb{Z}_{\overline{\seq}}$ we have $h_g^+=(h^+)_g$. Hence, any false face $F$ has exactly $\overline{\seq}$ preimages in  $\BF{\Gamma^{\seq}}$ with the same perimeter as $F$. Regarding the maps $u_{F}$ as attaching maps, we have
\begin{displaymath}
\Sigma^\seq \cong \Gamma^{\seq} \, {\cup}_{\, {\scriptstyle\bigsqcup} u_U} \Bigg( \bigsqcup_{U} S^1 \times [0,1] \Bigg) \cup_{\, {\scriptstyle\bigsqcup} u_{F}} \Bigg( \bigsqcup_{F} D^2  \Bigg) 
\end{displaymath} 
\noindent where $F \in \BF{\Gamma^{\seq}}$ ranges over the set of preimages of all false faces of $\Gamma$ under $\Theta$ and $U$ is indexed by all other faces of $\Gamma^{\seq}$. Similarly, we identify $\Sigma$ with 
\begin{displaymath}
\phantom{\Sigma^\seq \cong \quad}  \Gamma  \, \cup_{\, {\scriptstyle\bigsqcup} u_V} \Bigg( \bigsqcup_{V} S^1 \times [0,1] \Bigg) \cup_{\, {\scriptstyle\bigsqcup} u_G} \Bigg( \bigsqcup_{G} D^2  \Bigg)
\end{displaymath} 
\noindent where $G \in \BF{\Gamma^{\seq}}$ is indexed by the false faces of $\Gamma$ and $V$ is indexed by all other faces. Then, $\Theta_{\Sigma}: \Sigma^\seq \rightarrow \Sigma$ extends $\Theta: \Gamma^{\seq} \rightarrow \Gamma$. It maps a point $(z, t) \in S^1 \times [0,1]$ in the annulus corresponding to a face $U$ to the point $(z^{d_{U}}, t)$ in the annulus which corresponds to $V=\Theta(U)$ and maps every disc corresponding to a preimage $F$ of a false face one-to-one onto the disc corresponding to $G=\Theta(F)$.

By construction, $\Theta_{\Sigma}$ is a topological branched cover whose group of deck transformations is isomorphic to the group of ribbon graph automorphisms $f:\Gamma^{\seq} \rightarrow \Gamma^{\seq}$ such that $\Theta \circ f=\Theta$. Thus, $\Deck(\Theta_{\Sigma}) \cong \mathbb{Z}_{\overline{m}}$. By \cite[Proposition 4.40]{LeeSmoothManifolds}, there exists a unique smooth structure on $\Sigma^{\seq}$ such that $\Theta_{\Sigma}$ is smooth.
\end{proof}

\begin{rem}The statement of Lemma \ref{LemmaBranched} remains true for any smooth structure if one allows for a deformation of the embedding $\kappa: \Gamma^{\seq} \hookrightarrow \Sigma^{\seq}$. Since $\Sigma^{\seq}$ is a compact surface, all its self-homeomorphisms are isotopic to a diffeomorphism (e.g. see \cite{HatcherTorusTrick} for a simple proof). Applied to the identity it shows that $\Theta_{\Sigma}$ may be deformed by a diffeomorphism $f$ (isotopic to the identity) to be smooth for any previously chosen smooth structure. After deforming $\kappa$ via $f$, the deformation of $\Theta_{\Sigma}$ extends $\Theta$.
\end{rem}

\noindent Now we can consider $\cA^{\seq}$ as lifts of arcs in $\cA$ along the branched cover $\Theta_{\Sigma}$. By Lemma \ref{LemmaBranchedCoverInducesCategoricalAction}, $\mathbb{Z}_{\overline{\seq}}$ acts freely on the category $\mathbb{B}=\mathbb{B}(\mathcal{A}^{\seq}, \mathbf{1})$. Its orbit category $\mathbb{B}/ \mathbb{Z}_{\overline{\seq}}$ contains a Morita equivalent subcategory, which we denote by  $\mathbb{B}^{\seq}$, whose objects are in bijection with $\mathcal{A}$ (Corollary \ref{CorollarySkeletonMoritaEquivalence}). The following proposition is analogous to \cite[Theorem 3.9.]{GreenSchrollSnashall} and shows that $\mathbb{B}^{\seq}$ is nothing but the category $\mathbb{B}(\mathcal{A}, \seq)$.

\begin{prp}\label{PropsitionOrbitIsBGC}
Let $\mathcal{A}$ be a graded admissible arc system on a graded surface of ribbon type, let $\seq$ be a multiplicity function and let $\mathbb{B}^{\seq}$ denote the $A_{\infty}$-category defined above. Then, there exists an isomorphism
$$\begin{tikzcd} \mathbb{B}^{\seq} \arrow{r}{\simeq} &  \mathbb{B}(\mathcal{A}, \seq)\end{tikzcd}$$
\noindent of the underlying $\Bbbk$-linear graded categories under which the $A_{\infty}$-operations in Definition \ref{DefinitionBrauerGraphCategory} correspond to the $A_{\infty}$-operations of $\mathbb{B}^{\seq}$.
\end{prp}
\begin{proof}Let $B_{\tilde{\cA}}$ denote the modified Brauer graph algebra of $(\Gamma^{\seq}, \mathbf{1})$ and let $B_{\cA}$ denote the modified Brauer graph algebra associated with $(\Gamma, \seq)$.
Let us construct a morphism of graded $\Bbbk$-linear categories  $\varphi: \bB^{\seq} \rightarrow \bB(\cA, \seq)$.

Every object $X \in \Ob{\bB^{\seq}}$ corresponds uniquely to an arc $\gamma_X\in \tilde{\Gamma}$. On objects $\varphi$ is defined by sending $X \in \Ob{\bB^{\seq}}$ to the object $X_{\Theta(\gamma_X)}$. 
For all $Y \in \Ob{\bB^{\seq}}$, the graded vector space $\Hom_{\bB^{\seq}}(X,Y)$ possesses a homogeneous basis which is in bijection with non-zero paths in $B_{\tilde{\cA}}$  from the vertex $\gamma_X \in \BV{Q_{\Gamma^{\seq}}}$ and to all vertices $g.\gamma_Y$, where $g \in \mathbb{Z}_{\overline{\seq}}$. 
Note that every path from $\gamma_X$ to $g.\gamma_Y$ is a subpath of a maximal path around some vertex $(v,f+U_v)$ incident to $\gamma_X$ and starting at $\alpha_h$ for a half-edge $h\in \gamma_X$. Thus, there are non-zero paths from $\gamma_X$ to $g.\gamma_Y$ if and only if $g\in U_v$ and there is a path from $\Theta(\gamma_X)$ to $\Theta(\gamma_Y)$ in $\bB(\cA, \seq)$. For a path $\alpha: \gamma_X \rightarrow g.\gamma_Y$ going around vertex $v$ and starting at $\alpha_h$  define $\varphi(\alpha)$ to be the path of the same length  $\varphi(\alpha): \Theta(\gamma_X) \rightarrow \Theta(\gamma_Y)$, going around $\Theta(v)$ and starting at $\alpha_{\Theta(h)}$. Since $|U_v|=\seq(v)$, the map $\varphi$ assigns to a morphism $X_{\gamma} \rightarrow X_{\gamma}$ which corresponds to a maximal path $C_{\alpha}$ of $B_{\tilde{\cA}}$ the morphism $X_{\Theta(\gamma)} \rightarrow X_{\Theta(\gamma)}$ which corresponds to a maximal path $C_{\Theta(\alpha)}^{\seq(v)}$ of $B_{\cA}$.  We observe that $\varphi$ is compatible with compositions and preserves the relations of the modified Brauer graph algebras due to $\omega_{\tilde{\eta}}(\gamma)=\omega_{\eta}(\Theta(\gamma))$. Equality of the two winding numbers follows from the fact that $\Theta$ is a local diffeomorphism away from the punctures. Finally, due to our choice of gradings on $\tilde{\cA}$, it is apparent that $\varphi$ is compatible with the given gradings on $\bB$ and $\bB(\cA, \seq)$ and we conclude  that $\varphi$ induces an isomorphism of graded $\Bbbk$-linear categories.

In order to show that $\varphi$ is compatible with the higher operations defined on both categories it suffices to show that $\varphi$ maps disc sequences to disc sequences and that every disc sequence has a preimage which is a disc sequence. This follows from the fact that $\Theta|_{\tilde{\Sigma}_{\reg}}$ is an immersion and the fact that every continuous map from a marked disc $D$ to $\Sigma$ which is a smooth immersion outside a finite number of points on its boundary lifts to a map $D \rightarrow \tilde{\Sigma}$ with the same properties.
\end{proof}

\begin{cor}\label{CorollaryBGACategoryHigherMultWellDefined}The operations in Definition \ref{DefinitionBrauerGraphCategory} turn the $\Bbbk$-linear graded category $\mathbb{B}(\mathcal{A}, \seq)$ into an $A_{\infty}$-category.
\end{cor}

\section{Morita equivalences and elementary moves}\label{SectionMoritaEquivalences}

\noindent In this section we prove that the Morita equivalence class of a Brauer graph category of a graded punctured surface is independent of the chosen arc collection, see Corollary \ref{CorollaryBGAIndependentOfArcSystem}. The following preliminary result serves as the foundation for our proof. In this section, all admissible arc systems are understood to be on a graded surface of ribbon type.

\begin{prp}\label{Prop_AddingAnArcMorita}Let $\mathcal{A}$ be a graded admissible arc system and let $\gamma \in \mathcal{A}$ be such that $\mathcal{B}=\mathcal{A} \setminus \{\gamma\}$ is full. Then, $\mathbb{B}(\mathcal{B}, \seq)$ is a full subcategory of  $ \mathbb{B}(\mathcal{A}, \seq)$ and the inclusion is a Morita equivalence.
\end{prp}
\begin{proof}

The first part of the assertion is clear and for the second part it is sufficient to show that the induced functor between the respective homotopy categories of twisted complexes is essentially surjective. 
Let $(\cA,\sC)$ be a cutting pair. Since $\cB$ is full, $(\cB,\sC)$ is a cutting pair as well. Let   $\cF_\cA$ and $\cF_\cB$ denote the corresponding $A_\infty$-categories. We can consider the following commutative diagram of $A_\infty$-functors.
\begin{displaymath}
 \begin{tikzcd}
       {\cF}_{\cB} \arrow{d}{\Psi_\cB} \arrow{r}{\iota_\cF} & {\cF}_{\cA}  \arrow{d}{\Psi_\cA} \\
        \bB(\cB, \seq) \arrow{r}{\iota_\bB} & \bB(\cA, \seq). \end{tikzcd}
\end{displaymath}
Here $\iota_\cF$ and $\iota_\bB$ denote the natural inclusions and $\Psi_\cA$ and $\Psi_\cB$ are inclusions, which send an object $X_\delta$ corresponding to an arc $\delta$ to the object corresponding to the same arc and any morphism corresponding to a path $\delta_1\rightarrow \delta_2$ to the corresponding path $\delta_1 \rightarrow \delta_2$ which does not contain $\alpha_\sC^p$ for any $p\in \punct$. Since disc sequences for $\cF_{\cB_\sC}$ correspond to disc sequences for $\mathbb{B}(\mathcal{B}, \seq)$ and disc sequences for $\cF_{\cA_\sC}$ correspond to disc sequences for $\mathbb{B}(\mathcal{A}, \seq)$, this defines strict $A_\infty$-functors. Note that in case $\seq=\mathbf{1}$ the functor $\psi$ is just the inclusion of an $A_\infty$-category into its trivial extension from Lemma  \ref{LemmaInclusionsTrivialExtension}.
The diagram above induces a commutative square of exact functors $\H^0(\Tw\iota_\cF), \H^0(\Tw\iota_\bB), \H^0(\Tw\Psi_\cB)$ and $\H^0(\Tw\Psi_\cA)$ between the respective homotopy categories of twisted complexes. Since $\H^0(\Tw\Psi_\cA)$ is a functor, it preserves isomorphisms and it follows from Remark \ref{RemarkSketchHKKProof} and Lemma \ref{LemmaWellDefineFukaya} that in $\H^0(\Tw\mathbb{B}(\mathcal{A}, \seq))$, the object $X_\gamma$ is isomorphic to a twisted complex on the remaining generators $X_\delta$, where $\delta\neq \gamma$.   Thus, the essential image of the inclusion $\H^0(\Tw\iota_\bB): \H^0(\Tw\mathbb{B}(\mathcal{B}, \seq)) \rightarrow \H^0(\Tw\mathbb{B}(\mathcal{A}, \seq))$ contains all the objects $X_\delta$, $\delta\in\cA$. Since these generate the triangulated category $\H^0(\Tw\mathbb{B}(\mathcal{A}, \seq))$ and since the essential image of $\H^0(\Tw\iota_\bB)$ is a triangulated subcategory, we conclude that $\Psi_\cA$ is essentially surjective and hence a Morita equivalence.\end{proof}
\noindent The previous proposition motivates the following definition.

\begin{definition}
	Let $\mathcal{A}, \mathcal{B}$ be two admissible arc systems. We say that $\mathcal{B}$ is obtained from $\mathcal{A}$ by an \textbf{elementary move} if either $\mathcal{B}=\cA \sqcup \gamma$ or $\mathcal{B}=\cA \backslash\{\gamma\}$.
\end{definition}

\noindent Proposition \ref{Prop_AddingAnArcMorita} implies that Brauer graph categories whose arc systems are related by a sequence of elementary moves are Morita equivalent. We emphasize a few easy observations:

\begin{itemize}
    \item Every admissible arc system is connected to a maximal admissible arc system by a sequence of elementary moves.
    \item  If $(\cA, \mathsf{C})$ is a cutting pair and $\cA$ is maximal, then every flip of edges in the associated triangulation of $\Sigma_{\mathsf{C}}$ corresponds to a two-step sequence of elementary moves given by removing and adding an arc. The sequence transforms $\cA$ into a new maximal admissible arc system for which $\mathsf{C}$ is a cut. By \cite{Hatcher}, every two triangulations of $(\Sigma_{\sC}, \cA_{\sC})$ are related by a sequence of flips. Note that a triangulation in \cite{Hatcher} does not contain boundary arcs. However, one obtains a triangulation in the above sense from a triangulation in the sense of \cite{Hatcher} simply by adding boundary arcs and this process gives rise to a bijection. 
    \item If $\cA$ and $\cB$ are graded arc systems which differ only by their respective grading, then the Brauer graph categories $\mathbb{B}(\cA, \seq)$ and $\mathbb{B}(\cB, \seq)$ are Morita equivalent, since a shift of the grading of an arc corresponds to a shift in the category of twisted complexes.
\end{itemize}

\noindent Suppose now that $(\cA, \mathsf{C})$ is a cutting pair with $\mathsf{C}=(c_p)_{p \in \punct}$. Then for every boundary component $B \subseteq \partial \Sigma$, the map $p \mapsto c_p(1)$ induces a bijection from $\punct_\mathsf{C}(B) \coloneqq \{ p \in \punct \, | \, c_p(1) \in B\}$ to $\deco_{\mathsf{C}} \cap B$ which turns $\punct_{\mathsf{C}}(B)$ into a cyclically ordered set. By passing back and forth between a maximal cutting pair and its associated triangulation of a marked surface, the above observations show the following.

\begin{cor}\label{CorollaryTriangulationConnectedByElementaryMoves} Suppose that $(\cA, \mathsf{C})$ and $(\cB, \mathsf{D})$ are cutting pairs  of $\Sigma$ such that for all boundary components $B \subseteq \partial\Sigma$, $\punct_{\mathsf{C}}(B)=\punct_{\mathsf{D}}(B)$ as cyclically ordered sets. Then, $\cA$ and $\cB$ are connected by a sequence of elementary moves.
\end{cor}

\begin{prp}
	Let $\mathcal{A}, \mathcal{B}$ be two admissible arc systems on a punctured surface $\Sigma$. Then, there exists a sequence of elementary moves which transforms $\mathcal{A}$ into $\mathcal{B}$.
\end{prp}
\begin{proof}
We may assume that $\cA$ and $\cB$ are maximal admissible and choose a cut $\mathsf{C}=(c_p)_{p\in \mathscr{P}}$ of $\cA$ and a cut $\mathsf{D}=(d_p)_{p\in \mathscr{P}}$ of $\cB$. By Corollary \ref{CorollaryTriangulationConnectedByElementaryMoves}, it suffices to transform $(\cA, \mathsf{C})$ by a sequence of elementary moves into a cutting pair $(\cA^{\ast}, \mathsf{C}^{\ast})$ such that for each connected component $B \subseteq \partial \Sigma$, $\punct_{\mathsf{C}^\ast}(B)=\punct_{\mathsf{D}}(B)$ as cyclically ordered sets.
    Our first goal is to show that we can transform $(\cA, \mathsf{C})$ into a pair $(\hat{\cA}, \hat{\mathsf{C}})$ such that $\punct_{\hat{\mathsf{C}}}(B)=\punct_{\mathsf{D}}(B)$ as \emph{sets}. We proceed by induction over the expression
    \begin{displaymath}
    \Delta_{\punct}(\mathsf{C}, \mathsf{D})=\sum_{B\subseteq \partial \Sigma} \Big | \punct_{\mathsf{C}}(B)  \, \triangle \,  \punct_{\mathsf{D}}(B)\Big|.\end{displaymath}
    
   \noindent where for a pair of sets $U$ and $V$, $U \triangle  V= (U \cup V) \setminus (U \cap V)$ denotes their symmetric difference. Note that $\Delta_{\punct}(\mathsf{C}, \mathsf{D})= 0$ if and only if $\punct_{\mathsf{C}}(B)=\punct_{\mathsf{D}}(B)$ for all components $B \subseteq \partial \Sigma$. In what follows, we describe a series of elementary moves which replaces $(\cA, \mathsf{C})$ by a cutting pair $(\cA'', \mathsf{C}'')$ such that $\Delta_{\punct}(\mathsf{C}'', \mathsf{D}) < \Delta_{\punct}(\mathsf{C}, \mathsf{D})$. If $\Delta_{\punct}(\mathsf{C}, \mathsf{D}) \neq 0$, then there exists a boundary component $B_0$ of $\Sigma$ such that $\punct_{\mathsf{C}}(B_0) \triangle \punct_{\mathsf{D}}(B_0) \neq \emptyset$. Choose $p \in \punct_{\mathsf{C}}(B_0) \triangle \punct_{\mathsf{D}}(B_0)$. Since $\bigcup_{B \subseteq \partial \Sigma}\punct_{\mathsf{C}}(B)=\punct=\bigcup_{B \subseteq \partial \Sigma}\punct_{\mathsf{D}}(B)$, there exists a component $B_1 \neq B_0$ such that $p \in \punct_{\mathsf{C}}(B_1) \triangle \punct_{\mathsf{D}}(B_1)$.  Set $q \coloneqq c_p(1)$ and let $j \in \{0,1\}$ be such that $q \in B_j$. There exists a path $f_q:[0,1] \rightarrow \Sigma_\sC$  such that $f_q(0)=q$ and $f_q(1) \in B_{1-j}\backslash \deco_\sC$  and which crosses each arc of $\cA_\sC$ transversally and at most once. Now, we can flip the arcs of the triangulation $\cA_\sC$ which cross $f_q$ in the order in which they appear on $f_q$, starting from the crossing nearest to $q$. Each flip reduces the number of crossings and we obtain a triangulation $\cA'_\sC$ of $\Sigma_{\sC}$ whose arcs meet $f_q$ only at the boundary point $f_q(1)$:
    
     
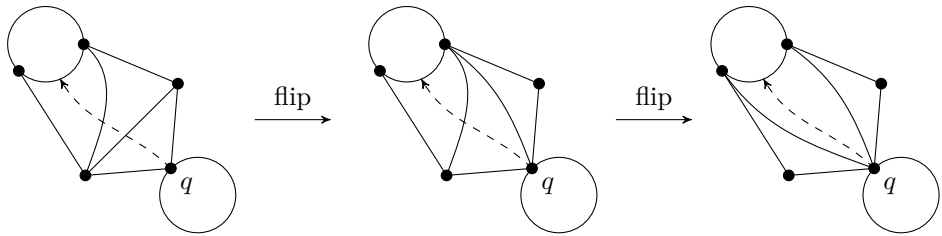
\begin{displaymath}
    \begin{tikzpicture}
    
\begin{scope}
    \draw (0,0) circle (0.5);
    \draw (-2,2) circle (0.5);
      
    \coordinate (A) at (135:0.5);
    \coordinate (B) at ($ (-2,2) + (0:0.5) $);
    \coordinate (C) at ($ (-2,2) + (225:0.5) $);
    \def\r{1.5};
    \def\alp{10};
    \coordinate (D) at ({90+\alp}:{\r});
    \coordinate (E) at ({180-\alp}:{\r}); 
    \coordinate (F) at ($ (-2,2) + (292.5:0.5) $);
     
    \filldraw (A) circle (2pt) node[anchor=north west] {$q$};
    \filldraw (B) circle (2pt);
    \filldraw (C) circle (2pt);
    \filldraw (D) circle (2pt);
    \filldraw (E) circle (2pt);
                    
    \draw (B)--(D)--(A)--(E)--(C);
    \draw (E)--(D);
    \draw (B)  to[out=-40,in=70] (E);
    \draw[dashed,->] (A) to[out=135,in=-60] (F);
    \end{scope} 
    
\draw[->] (0.75,1)--(1.75,1) node[midway, above] {$\text{flip}$};
\begin{scope}[shift={(4.75,0)}]
    \draw (0,0) circle (0.5);
    \draw (-2,2) circle (0.5);
      
    \coordinate (A) at (135:0.5);
    \coordinate (B) at ($ (-2,2) + (0:0.5) $);
    \coordinate (C) at ($ (-2,2) + (225:0.5) $);
    \def\r{1.5};
    \def\alp{10};
    \coordinate (D) at ({90+\alp}:{\r});
    \coordinate (E) at ({180-\alp}:{\r}); 
    \coordinate (F) at ($ (-2,2) + (292.5:0.5) $);
     
    \filldraw (A) circle (2pt) node[anchor=north west] {$q$};
    \filldraw (B) circle (2pt);
    \filldraw (C) circle (2pt);
    \filldraw (D) circle (2pt);
    \filldraw (E) circle (2pt);
                    
    \draw (B)--(D)--(A)--(E)--(C);
    \draw (A) to[out=110, in=-35] (B);
    \draw (B)  to[out=-40,in=70] (E);
    \draw[dashed,->] (A) to[out=135,in=-60] (F);
    
    \draw[->] (0.75,1)--(1.75,1) node[midway, above] {$\text{flip}$};    
    \end{scope}

\begin{scope}[shift={(9.25,0)}]
    \draw (0,0) circle (0.5);
    \draw (-2,2) circle (0.5);
      
    \coordinate (A) at (135:0.5);
    \coordinate (B) at ($ (-2,2) + (0:0.5) $);
    \coordinate (C) at ($ (-2,2) + (225:0.5) $);
    \def\r{1.5};
    \def\alp{10};
    \coordinate (D) at ({90+\alp}:{\r});
    \coordinate (E) at ({180-\alp}:{\r}); 
    \coordinate (F) at ($ (-2,2) + (292.5:0.5) $);
     
    \filldraw (A) circle (2pt) node[anchor=north west] {$q$};
    \filldraw (B) circle (2pt);
    \filldraw (C) circle (2pt);
    \filldraw (D) circle (2pt);
    \filldraw (E) circle (2pt);
                    
    \draw (B)--(D)--(A)--(E)--(C);
    \draw (A) to[out=110, in=-35] (B);
    \draw (A)  to[out=160,in=-55] (C);
    \draw[dashed,->] (A) to[out=135,in=-60] (F);
    \end{scope} 
    
    \end{tikzpicture}
    \end{displaymath}
    
    \noindent By performing the same sequence of flips on the arcs of $\cA$ we obtain a maximal admissible arc system $\cA'$ of $\Sigma$ such that $(\cA^{\prime })_{\sC}=\cA'_\sC$ and an embedded path $c''$ which connects $p$ with a point on $B_{1-j}$ and which crosses at most one arc $\delta \in \cA'$, namely the arc corresponding to the boundary arc in $\cA'_\sC$ which contains $f_q(1)$. The arc system $\cA'':=\cA'\backslash \{\delta\}$ is admissible and has a cut $\sC'' \coloneqq \left(\mathsf{C} \setminus \{c_p\}\right)  \sqcup \{c''\}$. Indeed, $\cA''$ is full as $\cA'$ cuts $\Sigma$ into discs and annuli and removing $\delta$ results in gluing an annulus with a neighbouring disc  containing the path $c''$. By construction, $\Delta_{\punct}(\mathsf{C}'', \mathsf{D}) < \Delta_{\punct}(\mathsf{C}, \mathsf{D})$. We proceed by induction and obtain a series of elementary moves which transforms $(\cA, \mathsf{C})$ into a cutting pair $(\hat{\cA}, \hat{\mathsf{C}})$ such that $\punct_{\hat{\mathsf{C}}}(B)=\punct_{\mathsf{D}}(B)$ as sets for all $B \subseteq \partial \Sigma$. However, the cyclic orders of these sets might not agree. Using the algorithm from above with $f_q$ connecting points on the same boundary component we can change the cyclic order on $\punct_{\hat{\mathsf{C}}}(B)$ to agree with $\punct_{\mathsf{D}}(B)$ for every component $B$. This yields the desired pair $(\cA^\ast, \mathsf{C}^\ast)$ and finishes the proof.
\end{proof}

\begin{cor}\label{CorollaryBGAIndependentOfArcSystem}
Let $\cA$ and $\cB$ be graded admissible arc systems. Then, for every multiplicity function $\seq$, $\mathbb{B}(\cA, \seq)$ and  $\mathbb{B}(\cB, \seq)$ are Morita equivalent.
\end{cor}

\noindent The previous corollary does not state in which way a change of the multiplicity function or the line field affects the Morita equivalence class of the associated Brauer graph category. It turns out that the Morita equivalence class is not affected as long as the two multiplicity functions have the same multi-set of values. We shall see in the next section that this phenomenon is a consequence of the fact that the Morita equivalence class of a Brauer graph category only depends on the \emph{orbit} of the line field under the action of the mapping class group.
	
\section{Orbits of line fields under the mapping class group action and derived equivalence classification}\label{SectionOrbitsLineFields}

\noindent We discuss the action of diffeomorphisms of a surface on homotopy classes of its line fields and invariants of the resulting orbits due to Lekili-Polishchuk \cite{LekiliPolishchukGentle} which extend earlier results of Kawazumi \cite{Kawazumi}. Matching the invariants of orbits  under the  action of the mapping class group with derived invariants of Brauer graph algebras finally leads to a proof of Theorem \ref{IntroTheoremCriterionDerivedEquivalence}.

\subsection{Action of mapping class groups on line fields and multiplicities of Brauer graph algebras}\ \medskip

\noindent Being defined as the intersection number of $\eta(\Sigma)$ and $\dot{\gamma}$ inside $\mathbb{P}(T\Sigma)$, the winding number function $\omega_{\eta}$ of a line field $\eta$ (as a function on the set of loops) is an invariant of its homotopy class. In fact, one can show that two line fields are homotopic if and only if their corresponding winding number functions agree, see \cite{Chillingworth}.
More generally, we are interested in the question when two line fields on a surface $\Sigma$ are not necessarily homotopic but their homotopy classes can be transformed into each other by a diffeomorphism of $\Sigma$. The natural group of diffeomorphisms one considers for this purpose is the following.

\begin{definition}\label{DefinitionMCG}Let $(\Sigma, \deco)$ be a decorated surface. The \textbf{mapping class group $\MCG(\Sigma, \deco)$} consists of the isotopy classes of all orientation-preserving diffeomorphisms $f:\Sigma \rightarrow \Sigma$ such that $f(\deco)=\deco$.
\end{definition}
\noindent Two   diffeomorphisms $f, g: \Sigma \rightarrow \Sigma$ as in Definition \ref{DefinitionMCG} are considered \textbf{isotopic} if there exists a smooth homotopy $H: [0,1] \times \Sigma \rightarrow \Sigma$, such that $H(0, -)=f(-)$ and $H(1, -)=g(-)$ and for all $t \in [0,1]$, $H(t, -)$ is a diffeomorphism. Moreover, for every $p \in \deco$, the path $H(-,p):[0,1] \rightarrow \Sigma$ is required to be constant. In other words, isotopies leave marked points fixed. 

\begin{exa}\label{ExampleHalfTwist} By definition, every mapping class $f \in \MCG(\Sigma, \deco)$ induces a permutation of the set of punctures $\punct$. A transposition is realized by a so-called \textit{half-twist}. Suppose that $\varepsilon:[0,1] \rightarrow \Sigma$ is an embedded path which connects punctures $p$ and $q$ such that $\varepsilon\big((0,1)\big) \subseteq \Sigma_{\reg}$. A \textbf{half-twist} $T_{\varepsilon} \in \MCG(\Sigma, \deco)$ about $\varepsilon$ permutes $p$ and $q$ by a rotation and agrees with the identity outside of a neighbourhood of $\varepsilon$, see Figure \ref{FigureHalfTwists}.  
\end{exa}

	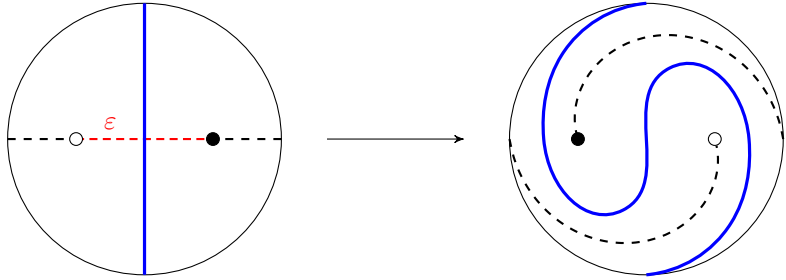
\begin{figure}[h]
			\centering
			\begin{tikzpicture}[scale=0.3]
			\draw (0,0) circle (6);

			\draw[thick, dashed] (-6,0)--(-3,0);
			\draw[thick, dashed] (6,0)--(3,0);

			\draw[->] (8,0)--(14,0);
			
			\draw ({0+22},0) circle (6);
			
			
			\draw[densely dashed, thick , red] (-3,0)--(3,0) node[pos=0.25, above] {$\varepsilon$};
			
			\filldraw (3, 0) circle(8pt);
			\filldraw[white] (-3, 0) circle(8pt);
			\draw[black] (-3, 0) circle(8pt);

			\draw[scale=1,domain=0:1,smooth, dashed,variable=\t, thick] plot ({22+(-6+3*\t)* cos(deg((6-6+3*\t)/3 * pi))},{(-6+3*\t)* sin(deg((6-6+3*\t)/3 * pi))});
			\draw[scale=1,domain=0:1,smooth, dashed,variable=\t, thick] plot ({22+(6-3*\t)* cos(deg((6-6+3*\t)/3 * pi))},{(6-3*\t)* sin(deg((6-6+3*\t)/3 * pi))});        
			
			\filldraw ({-3+22}, 0) circle(8pt);
			\filldraw[white] ({3+22}, 0) circle(8pt);
			\draw[black] ({3+22}, 0) circle(8pt);

			\draw[blue, very thick] (0,6)--(0,-6);
			\begin{scope}[rotate around={90:(22,0)}]
			\draw[blue, very thick] ({22+(-6+3*0)* cos(deg((6-6+3*0)/3 * pi))},{(-6+3*0)* sin(deg((6-6+3*0)/3 * pi))}) to[curve through ={({22+(-6+3*0.25)* cos(deg((6-6+3*0.25)/3 * pi))},{(-6+3*0.25)* sin(deg((6-6+3*0.25)/3 * pi))}) . . ({22+(-6+3*0.5)* cos(deg((6-6+3*0.5)/3 * pi))},{(-6+3*0.5)* sin(deg((6-6+3*0.5)/3 * pi))}) . . ({22+(-6+3*0.9)* cos(deg((6-6+3*0.9)/3 * pi))},{(-6+3*0.9)* sin(deg((6-6+3*0.9)/3 * pi))}) ..  ({22+(6-3*(1-0.1))* cos(deg((6-6+3*(1-0.1))/3 * pi))},{(6-3*(1-0.1))* sin(deg((6-6+3*(1-0.1))/3 * pi))}) .. ({22+(6-3*(1-0.5))* cos(deg((6-6+3*(1-0.5))/3 * pi))},{(6-3*(1-0.5))* sin(deg((6-6+3*(1-0.5))/3 * pi))}) .. ({22+(6-3*(1-0.75))* cos(deg((6-6+3*(1-0.75))/3 * pi))},{(6-3*(1-0.75))* sin(deg((6-6+3*(1-0.75))/3 * pi))})
			}] ({22+(6-3*(1-1))* cos(deg((6-6+3*(1-1))/3 * pi))},{(6-3*(1-1))* sin(deg((6-6+3*(1-1))/3 * pi))});
			\end{scope}
			

			\end{tikzpicture}
			
			\caption{The action of a half-twist on curves. It acts as the identity outside of a neighborhood of $\varepsilon$.}
			\label{FigureHalfTwists}
		\end{figure}	

\noindent The mapping class group acts naturally on homotopy classes of line fields via pullback.

\begin{definition}Let $f \in \MCG(\Sigma, \deco)$ and let $\eta: \Sigma_{\reg} \rightarrow \mathbb{P}(\operatorname{T}\Sigma)$ be a line field. The \text{pullback} of $\eta$ along $f$ is the line field $f^{\ast}\eta \coloneqq \mathbb{P}\operatorname{T}(f)^{-1} \circ \eta \circ f$, where $\mathbb{P}\operatorname{T}(f): \mathbb{P}(\operatorname{T}\Sigma) \rightarrow \mathbb{P}(\operatorname{T}\Sigma)$ denotes the map which is induced by the differential of $f$. \begin{displaymath}
 \begin{tikzcd}
        \mathbb{P}(T\Sigma) & & \mathbb{P}(T\Sigma) \arrow[swap]{ll}{\mathbb{P}\operatorname{T}(f)^{-1}} \\
        \Sigma_{\reg} \arrow{u}{f^{\ast}\eta} \arrow[swap]{rr}{f} & & \Sigma_{\reg} \arrow{u}[swap]{\eta}
 \end{tikzcd}
\end{displaymath}
\end{definition}

\begin{lem}\label{LemmaHalfTwistsPreserveHomotopyClasses}
Let $\eta$ be a line field on a punctured surface $(\Sigma, \deco)$ and let $T_{\varepsilon}: \Sigma \rightarrow \Sigma$ be a half-twist about a simple path $\varepsilon$ connecting punctures $p$ and $q$. If the winding numbers of $p$ and $q$ agree, then $T_{\varepsilon}^{\ast}\eta \simeq \eta$.
\end{lem}
\begin{proof}
It is sufficient to show that the winding number functions of $\theta \coloneqq T_{\varepsilon}^{\ast}\eta$ and $\eta$ agree. For every loop $\gamma$,  $\omega_{\theta}(\gamma)=\omega_{\eta}(T_{\varepsilon}(\gamma))$. Thus, we have to show that $T_{\varepsilon}$ preserves the winding number of every loop. Let $\delta$ denote a loop as in Figure \ref{FigureSpecialCurveHalfTwist}.

The loop $\delta$ is homotopic to  the concatenation of $\varepsilon$ with a clockwise primitive loop around $p$ followed by $\overline{\varepsilon}$ and a counter-clockwise primitive loop around $q$. The associated homology class $[\dot \delta] \in H_1(\mathbb{P}(T\Sigma), \mathbb{Z})$ coincides with $[\dot \delta_p] - [ \dot \delta_q]$, where $\delta_{x}$ denotes a clockwise simple loop around a puncture $x$. In particular, $\omega_{\eta}(\delta)= \omega_{\eta}(\delta_p) - \omega_{\eta}(\delta_q)=0$ by our assumptions. Assuming that $\gamma$ and $\varepsilon$ are transverse the assertion now follows from the observation that the homology class of the tangent curve $\dot{\big(T_{\varepsilon}(\gamma)\big)}=\mathbb{P}\operatorname{T}(T_{\varepsilon})(\dot \gamma)$ can be expressed as a sum 
\begin{displaymath}[\dot \gamma] + \sum_{x \in \gamma \cap \varepsilon} \pm \left[\dot{\delta}\right],\end{displaymath}
\noindent where the exact signs depend on the index of the respective intersection between $\varepsilon$ and $\gamma$. The vanishing of $\omega_{\eta}(\delta)$ implies that $\omega_{\eta}(T_\epsilon(\gamma))=\omega_{\eta}(\gamma)$ for any loop $\gamma$.
\end{proof}
	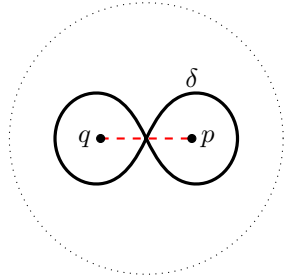
\begin{figure}[H]
			\centering
			\begin{tikzpicture}[scale=0.3]
			
			
			\draw[red, thick, dashed] (-2,0)--(2,0);
			

			\draw[dotted] ({0},0) circle (6);
			\filldraw ({-2}, 0) circle( 5pt);
			\filldraw ({2}, 0) circle(5pt);

			\draw (-2,0) node[left]{$q$};
			\draw (2,0) node[right]{$p$};


			\begin{scope}[closed hobby]
			\draw[very thick] plot coordinates {(2, 2) (4,0) (2,-2) (0,0) (-2,2) (-4,0) (-2,-2) (0,0)  } ;
			\draw (2,2.7) node{$ \delta$};
			\end{scope}
			
			\end{tikzpicture}
			
			\caption{The path $\varepsilon$ (dashed) and the loop $\delta$.}
			\label{FigureSpecialCurveHalfTwist}
		\end{figure}
\noindent 
Every homotopy between line fields $\eta$ and $\theta$ induces a bijection between the possible $\eta$-gradings and $\theta$-gradings on any arc system $\mathcal{A}$ which preserves the degrees of all oriented intersections. However, since the winding number of an arc is not invariant under homotopies of the line field, the bijection does not induce an isomorphism of the corresponding Brauer graph categories right away. The next lemma shows that the problem disappears since $\Bbbk$ is algebraically closed.
\begin{lem}\label{LemmaRootsOfUnityBaseChange}
Let $\eta_1, \eta_2$ be homotopic line fields on $(\Sigma, \punct)$ of ribbon type and let $\mathcal{A}$ be an admissible arc system  endowed with $\eta_i$-gradings corresponding to each other under a homotopy between $\eta_1$ and $\eta_2$. Let $\mathcal{A}_{\eta_i}$ denote $\mathcal{A}$ equipped with its $\eta_i$-grading. Then for each multiplicity function $\seq$, $\mathbb{B}(\mathcal{A}_{\eta_1}, \seq)$ and $\mathbb{B}(\mathcal{A}_{\eta_2}, \seq)$ are Morita equivalent.
\end{lem}
\begin{proof}
By Corollary \ref{CorollaryBGAIndependentOfArcSystem} we may and will assume that $\mathcal{A}$ is formal. First, we claim that there exists a tuple $(\sigma_p)_{p \in \punct}$, $\sigma_p \in \{0,1\}$, such that for any $\gamma \in \mathcal{A}$ with end points $p, q \in \punct$, $\sigma_p + \sigma_q \equiv \omega_{\eta_1}(\gamma) + \omega_{\eta_2}(\gamma) \mod 2$. The proof is similar to Lemma \ref{LemmaDeltaP}. Choose $p_0 \in \punct$ and set $\sigma_{p_0}=0$. For each path $\gamma^p=\gamma_l^p \cdots \gamma_1^p$ in $\Gamma=\Gamma_{\mathcal{A}}$ from $p_0$ to $p$ with oriented edges $\gamma_i^p \in \mathcal{A}=\BE{\Gamma}$, let $\sigma_p \in \{0,1\}$ denote the unique element such that
\begin{equation}\label{EquationComparison}
\sigma_p \equiv \sum_{i=1}^l \omega_{\eta_1}(\gamma_i^p) + \sum_{i=1}^l\omega_{\eta_2}(\gamma_i^p) \mod 2.
\end{equation}
\noindent Then $\sigma_p$ is independent of the chosen path $\gamma^p$ since given a closed path $\delta=\delta_r \cdots \delta_1$ based at $p_0$, Corollary \ref{CorollaryWindingEdges} and the assumption $\eta_1 \simeq \eta_2$ imply
\begin{displaymath}
\sum_{i=1}^r \omega_{\eta_1}(\delta_i) + \sum_{i=1}^r \omega_{\eta_2}(\delta_i) \equiv 
\big(r + \omega_{\eta_1}(\delta_{\text{sm}})\big) + \big(r + \omega_{\eta_2}(\delta_{\text{sm}})\big) \equiv 0 \mod 2.
\end{displaymath}
\noindent Let $\mathsf{C}$ be an arbitrary cut of $\mathcal{A}$. Then an isomorphism $\mathbb{B}(\mathcal{A}_{\eta_1}, \seq) \rightarrow \mathbb{B}(\mathcal{A}_{\eta_2}, \seq)$ can be given as follows. It maps all objects to themselves and it maps an arrow $a$ in $B_{\mathcal{A}_{\eta_1}}$ to itself unless $a=\alpha_p^{\mathsf{C}}$ for some $p \in \punct$ (as defined in the paragraph succeeding Definition \ref{DefCP}), in which case $a$ is mapped to $\lambda a$ where $\lambda \in \Bbbk$ is such that $\lambda^{\seq(p)}=(-1)^{\sigma_p}$. The existence of $\lambda$ is guaranteed by the assumption that $\Bbbk$ is algebraically closed. 
\end{proof}

\begin{prp}\label{PropositionSameOrbitsSameCategories}
Let $\eta_1, \eta_2$ be line fields on $(\Sigma, \punct)$ of ribbon type and let $\mathcal{A}_i$ be an $\eta_i$-graded admissible arc system. If the homotopy classes of $\eta_1$ and $\eta_2$ lie in the same $\MCG(\Sigma, \punct)$-orbit, then for all multiplicity functions $\seq_1, \seq_2$ with the same multi-sets of multiplicities, $\mathbb{B}(\mathcal{A}_1, \seq_1)$ and $\mathbb{B}(\mathcal{A}_2, \seq_2)$ are Morita equivalent.
\end{prp}
\begin{proof}
 Every $f \in \MCG(\Sigma, \punct)$ induces a canonical bijection between $\eta$-graded arc systems  and $f^{\ast}\eta$-graded arc systems which induces an isomorphism on the level of Brauer graph categories.
Suppose now that $\mathcal{A}$ is an $\eta$-graded arc system and that $\mathcal{B}$ is a $f^{\ast}\eta$-graded arc system. Let $\seq$ be a multiplicity function. Then, it follows from the bijection between arc systems above and Corollary \ref{CorollaryBGAIndependentOfArcSystem} that the Brauer graph categories $\mathbb{B}(\mathcal{A}, \seq)$ and $\mathbb{B}(\mathcal{B}, \seq \circ f)$ are Morita equivalent.

Now the proof of the assertion follows from this special case. By assumption, there exists $f \in \MCG(\Sigma, \punct)$ such that $f^{\ast}\eta_1 \simeq \eta_2$. Then, permuting punctures via half-twists and replacing the line field by a homotopic one, we may assume that $\seq_1 \circ f= \seq_2$  by Lemma \ref{LemmaHalfTwistsPreserveHomotopyClasses}. The assertion finally follows from Lemma \ref{LemmaRootsOfUnityBaseChange}.
\end{proof}

\subsection{A complete set of invariants for orbits}\ \medskip

\noindent Building on work of Kawazumi \cite{Kawazumi}, Lekili and Polishchuk \cite{LekiliPolishchukGentle} classified  $\MCG(\Sigma, \deco)$-orbits of line fields in terms of invariants. For a line field $\eta$, one defines an element $\sigma(\eta) \in \mathbb{Z}_2$ as follows:
\begin{displaymath}
\sigma(\eta)\coloneqq  \begin{cases} 0, & \text{if }\eta \text{ is orientable}; \\ 1, & \text{otherwise.} \end{cases} 
\end{displaymath}

\noindent We recall that a line field is orientable if it lifts to a vector field. It is clear that $\sigma(\eta)$ is an invariant under the action of $\MCG(\Sigma, \deco)$. Here is a description of $\sigma(\eta)$  in terms of winding numbers we will use later:
\begin{lem}[{\cite[Lemma 1.1.4.]{LekiliPolishchukGentle}}]\label{LemmaCharacterizationOrientable}
A line field $\eta$ is orientable, i.e.\ $\sigma(\eta)=0$, if and only if $\omega_{\eta}(\gamma)$ is even for every loop $\gamma$.
\end{lem}
\noindent  Further invariants of the $\MCG(\Sigma, \deco)$-orbit of $\eta$ are 
\begin{itemize}
    \item the multi-set $\Omega_{\eta}^{\partial}$ consisting of the winding numbers of all boundary components and the multi-set $\Omega_{\eta}^{\punct}$ consisting of the winding numbers of all punctures;
    \item the greatest common divisor $\gcd(\eta) \geq 0$ of all integers $\omega_{\eta}(\gamma)$, where $\gamma \subseteq \Sigma$ is a non-separating simple loop.
\end{itemize}

\noindent Under certain conditions one can define an additional invariant $\Arf(\eta) \in \mathbb{Z}_2$ which is expressed as the \textit{Arf invariant} of a certain quadratic form defined by $\eta$. This invariant will turn out to be irrelevant for our discussion of Brauer graph algebras and we refer the interested reader to \cite{LekiliPolishchukGentle} for the definition of $\Arf(\eta)$.\medskip

\noindent With all invariants in place, we can finally state the classification result for $\MCG(\Sigma, \deco)$-orbits of line fields. 

\begin{thm}[{\cite[Theorem 1.2.4.]{LekiliPolishchukGentle}}]\label{TheoremClassificationOrbits}
Let $\eta, \eta'$ be line fields on $\Sigma_{\reg}$ and let $g$ denote the genus of $\Sigma$. The homotopy classes of $\eta$ and $\eta'$ are in the same $\MCG(\Sigma, \deco)$-orbit if and only if $\Omega_{\eta}^{\partial}=\Omega_{\eta'}^{\partial}$ and $\Omega_{\eta}^{\punct}=\Omega_{\eta'}^{\punct}$ as multi-sets and moreover one of the following conditions holds:

\begin{enumerate}
    \item $g=0$;
    \item $g=1$ and $\gcd(\eta)=\gcd(\eta')$;
    \item $g \geq 2$ and $\sigma(\eta)=\sigma(\eta')$, additionally if $\sigma(\eta)=0=\sigma(\eta')$ and $b \equiv 2 \mod 4$ for all $b \in \Omega_{\eta}^{\partial}\cup \Omega_{\eta}^{\punct}$, then $\Arf(\eta)=\Arf(\eta')$. 
\end{enumerate}
\end{thm}
\noindent Note that if $\eta$ is of ribbon type, then $b \equiv 0 \mod 4$ for all $b \in \Omega_\eta^{\punct}$ and it follows immediately that line fields $\eta, \eta'$ of ribbon type on a punctured surface of genus at least $2$ lie in the same orbit if and only if $\Omega_{\eta}^{\partial}=\Omega_{\eta'}^{\partial}$ and $\sigma(\eta)=\sigma(\eta')$.\medskip

\noindent In their paper, Lekili and Polishchuk consider compact surfaces without punctures. However, their result carries over to the punctured case as stated. To see this, one replaces $(\Sigma, \punct)$ with a pair $(\Sigma_{\circ}, \marked_\circ)$, where   $\Sigma_{\circ} \subseteq \Sigma$ denotes a subsurface which is obtained by removing an open disc around every puncture and $\marked_{\circ}$ contains a single point on each of the new boundary components (see \cite[Corollary 1.2.6]{LekiliPolishchukGentle} for the case with marked points on the boundary). If $\eta_1$ and $\eta_2$ satisfy one of the conditions in Theorem \ref{TheoremClassificationOrbits}, it follows that $\eta_1|_{\Sigma_{\circ}}$ and $\eta_2|_{\Sigma_{\circ}}$ lie in the same $\MCG(\Sigma_{\circ}, \marked_{\circ})$-orbit. In the end, one uses the fact that every diffeomorphism $\Sigma_{\circ} \rightarrow \Sigma_{\circ}$ which permutes the boundary components corresponding to punctures can be extended to a diffeomorphism $\Sigma \rightarrow \Sigma$ and that the homotopy class of a line field on a punctured disc is determined by the winding number of a simple loop.

\subsection{Orbit invariants for line fields of ribbon type}\ \medskip

\noindent We translate some of the invariants of a ribbon graph into invariants of the line field $\eta_{\Gamma}$ on its ribbon surface.
\begin{lem}\label{LemmaBipartiteOrientable}
Let $\Gamma$ be a ribbon graph. Then,
$
 \sigma(\eta_\Gamma)=\sigma(\Gamma).
$
 In other words,
$\eta_\Gamma$ is orientable if and only if $\Gamma$ is bipartite.
\end{lem}
\begin{proof}
The surface $\Sigma_{\Gamma} \setminus \BV{\Gamma}$ deformation retracts onto a graph $\hat{\Gamma}$ which is obtained by blowing up each vertex $v \in \BV{\Gamma}$ to a circle $S^1_v$ and attaching the half-edges from $\BH{\Gamma}_v$ to $S^1_v$ according to the cyclic order on $\BH{\Gamma}_v$. The attaching point of $h$  on $S^1_{s(h)}$ is denoted by $z_h$. Thus, $\pi_1(\hat{\Gamma}) \cong \pi_1(\Sigma_{\Gamma})$ and hence there exists a bijection between homotopy classes of oriented loops on $\Sigma_{\Gamma}$ and equivalence classes of \textit{closed reduced walks} in $\Gamma$. By a closed reduced walk we mean a cyclic sequence $(h_1, h_2, d_1, h_3, h_4, d_2, \dots, h_{2m}, d_m)$ consisting of half-edges $h_i$ and integers $d_i \in \mathbb{Z}$ such that for all $i \in [1,m]$, $h_{2i}=\iota(h_{2i-1})$ and $s(h_{2i+1})=s(h_{2i})$. In addition we assume that $h_{2i+1}\neq h_{2i}$ whenever $d_i=0$. Two walks are considered equivalent if they agree after a rotation of their entries by $3l$ steps for some $l \in \mathbb{Z}$. Given a walk $(h_1, h_2, d_1, \dots, h_{2m}, d_m)$  its associated loop is obtained by joining each pair of edges $\overline{h_{2i+1}}$ and $\overline{h_{2i}}$, where $1 \leq i \leq m$ and indices are considered modulo $2m$, by a path from $z_{h_{2i}}$ to $z_{h_{2i+1}}$ in $S^1_v$, where $v=s(h_{2i+1})=s(h_{2i})$. More specifically, this connecting path corresponds to $d_i \in \mathbb{Z} \cong \pi_1(S^1_v, z_{h_{2i}})$ under the bijection between homotopy classes of all paths from $z_{h_{2i}}$ to $z_{h_{2i+1}}$ in $S^1_v$ and $\pi_1(S^1_v, z_{h_{2i}})$ given by concatenating a loop based at $z_{h_{2i}}$ with a path $\epsilon$ between the two points which we fixed in advance (e.g. a shortest clockwise path). In case $z_{h_{2i}}=z_{h_{2i+1}}$ we require $\epsilon$ to be constant.

We may assume that the circles $S^1_v$ are everywhere parallel to $\eta_{\Gamma}$. It follows that the parity of the winding number of a loop $\gamma$ which is represented by a walk $(h_1, h_2, d_1, \dots, h_{2m}, d_m)$ is $m$, c.f. Figure \ref{FigureWindingNumberBoundary} on page \pageref{FigureWindingNumberBoundary}. Thus, $\omega(\gamma)$ is even for all loops $\gamma$ if and only if all closed walks in $\Gamma$ have even length. The latter condition is equivalent to bipartivity of $\Gamma$ and by Lemma \ref{LemmaCharacterizationOrientable}, the former condition is equivalent to the orientability of $\eta_\Gamma$.
\end{proof}
	
\begin{lem}\label{LemmaGCD} Let $\eta$ be a line field on a punctured surface $\Sigma$ of genus $g \geq 1$. If $\eta$ is of ribbon type, then
\begin{displaymath}\gcd(\eta)=\begin{cases} 1, & \text{if }\sigma(\eta)=1; \\ 2, & \text{if }\sigma(\eta)=0.  \end{cases}\end{displaymath}

\end{lem}
\begin{proof}Let $p \in \Sigma$ be a puncture and let $\gamma:S^1\rightarrow \Sigma \setminus\punct$ be any non-separating simple loop. Choose any embedded closed path $\delta \subseteq \Sigma_{\reg}$ with base point $\gamma(z)$ for some $z \in S^1$, which bounds a closed disc $U \subseteq \Sigma$ such that $U \cap \mathscr{P}=\{p\}$ and $U \cap \gamma = \{\gamma(z)\}$. Regarding $\gamma$ as a closed path based at $\gamma(z)$, define a new non-separating simple loop $\gamma'$ as the concatenation of $\gamma$ and $\delta$ (this may require replacing $\delta$ by its inverse path). After deforming $\gamma'$ we may assume that $\gamma$ and $\gamma'$ are disjoint and bound a unique subsurface $V$ of genus $0$ such that $\partial V= \gamma \sqcup \gamma'$ and $V \cap \mathscr{P}= \{p\}$. An application of the Poincar\'e-Hopf Theorem to the punctured surface $(V, \{p\})$ and the line field $\eta|_{V}$ shows that \begin{displaymath}
 \omega_{\eta}(\gamma) + \omega_{\eta}(\gamma') = 2 \chi(V \setminus \{p\})= -2,\end{displaymath}

\noindent for suitable orientations of $\gamma$ and $\gamma'$. Hence, $\gcd(\eta) > 0$ and $\gcd(\eta) \mid 2$ which shows that $\gcd(\eta) \in \{1,2\}$. In particular, $\gcd(\eta)=2$ if all winding numbers are even, that is if $\sigma(\eta)=0$. On the other hand, suppose that $\gcd(\eta)=2$ but $\omega_{\eta}(\gamma)$ is odd for a loop $\gamma$. By \cite{Johnson}, $\omega_{\eta}$ induces a homomorphism $\omega_{\eta}^{(2)}:\operatorname{H}_1(\Sigma, \mathbb{Z}_2) \rightarrow \mathbb{Z}_2$ such that $\omega_{\eta}^{(2)}([\gamma]) \equiv \omega_{\eta}(\gamma) \mod 2$. By assumption, $\omega_{\eta}^{(2)} \neq 0$. As $\gcd(\eta)=2$ and $\operatorname{H}_1(\Sigma, \mathbb{Z}_2)$ is generated by boundary and non-separating loops, we conclude that there exists a boundary component $B_0$ such that $\omega_{\eta}(B_0)$ is odd. 
It follows from the argument in \cite[Lemma 2.6]{Kawazumi} that $\gcd(\eta) \mid \omega_{\eta}(B)+2$ for all boundary components $B$. For convenience let us recall it here: for any non-separating loop $\alpha$ and every boundary component $B$ one finds a non-separating loop $\alpha'$ such that $\alpha, \alpha'$ and $B$ bound a subsurface which is diffeomorphic to a pair of pants $P$. Another application of Poincar\'e-Hopf to $P$ shows that $\omega_{\eta}(\alpha')=\omega_{\eta}(\alpha)+\omega_{\eta}(B)+2$ for suitable orientations of the three loops and hence $\gcd(\eta) \mid \omega_{\eta}(B) +2$. Thus, $\gcd(\eta)$ is odd, which contradicts $\gcd(\eta)=2$. Hence, all winding numbers of $\eta$ are even and Lemma \ref{LemmaCharacterizationOrientable} shows that $\eta$ is orientable. 
\end{proof}	

\noindent The following is a consequence of Theorem \ref{TheoremClassificationOrbits} together with Example \ref{ExampleWindingNumbersBoundaries}, Lemma \ref{LemmaBipartiteOrientable} and Lemma \ref{LemmaGCD}.

\begin{cor}\label{CorollaryOrbitsLineFieldsRibbonGraph}Let $\Gamma, \Gamma'$ be ribbon graphs such that $\Sigma \coloneqq \Sigma_{\Gamma} \cong \Sigma_{\Gamma'}$ as punctured surfaces. Let $g$ denote the genus of $\Sigma$. Then, the homotopy classes of  $\eta_{\Gamma}$ and $\eta_{\Gamma'}$ lie in the same $\MCG(\Sigma, \punct)$-orbit if and only if the multi-sets of perimeters of faces of $\Gamma$ and $\Gamma'$ agree and one of the following conditions is satisfied:
\begin{enumerate}
\setlength\itemsep{1ex}
\item $g=0$;
\item $g \geq 1$ and $\sigma(\Gamma)=\sigma(\Gamma')$.
\end{enumerate}
\end{cor}

\noindent We can finally give a proof of Theorem \ref{IntroTheoremCriterionDerivedEquivalence}.
\begin{thm}\label{TheoremClassificationBGAs}
Let $(\Gamma, \seq)$ and  $(\Gamma', \seq')$ be Brauer graphs and assume that the corresponding Brauer graph algebras $B$ and $B'$ are not local. Then, $B$ and $B'$ are derived equivalent if and only if all of the following conditions are satisfied:
\begin{enumerate}
   	\item $\Gamma$ and $\Gamma'$ have the same number of vertices, edges and faces, i.e.\ $\Sigma_{\Gamma} \cong \Sigma_{\Gamma'}$;
   	  \item $\Gamma$ and $\Gamma'$ have identical multi-sets of perimeters of faces and the multi-sets of entries of $\seq$ and $\seq'$ coincide;
    \item $\sigma(\Gamma)= \sigma(\Gamma')$.
\end{enumerate}
\end{thm}
\begin{proof}
By Theorem \ref{TheoremDerivedInvariantsBrauerGraphAlgebra}, all conditions are satisfied if $B$ and $B'$ are derived equivalent. The converse follows from Proposition \ref{PropositionSameOrbitsSameCategories}, Corollary \ref{CorollaryOrbitsLineFieldsRibbonGraph} and the fact that an equivalence between the categories of perfect complexes over algebras $B$ and $B'$ implies an equivalence between the corresponding derived categories \cite{RickardMoritaTheory}.
\end{proof}

\vspace{-12pt}
\subsection{Concluding remarks} \ \medskip

\noindent We finish this paper with a couple of remarks on derived equivalences of Brauer graph algebras. For the undefined terminology of silting mutation see \cite{AiharaIyama}.

\begin{rem}
As can be seen from Theorem \ref{TheoremClassificationBGAs},  two Brauer graph algebras $B_1, B_2$ with identical underlying ribbon graphs $\Gamma$ are derived equivalent if their multi-sets of multiplicities agree. On the topological level, this is a consequence of Lemma \ref{LemmaHalfTwistsPreserveHomotopyClasses} which shows that the homotopy class of a line field of ribbon type is invariant under half-twists. From an algebraic perspective an explanation can be given in terms of mutation of tilting complexes and spherical twists. Suppose that $\seq_1$ and $\seq_2$ are the multiplicity functions of $B_1$ and $B_2$ and suppose that they differ only at two vertices, say $v_1, v_2$, which are connected by an edge $h$. Then, a tilting complex $T$ over $B_1$ providing a derived equivalence between $B_1$ and $B_2$ can be constructed as follows: $T$ is a silting mutation of $B_1$ with respect to the additive subcategory generated by $P_h$. That is, $T=P_h \oplus\sum_{h'\in E(\Gamma), h'\neq h} T_{h'}$ with $T_{h'}$ defined from the following triangle
\begin{equation}\label{EquationApproximation}
\begin{tikzcd}
T_{h'}\arrow{r} & P  \arrow{r}{f_{h'}} & P_{h'}\arrow{r} & T_{h'}[1],
\end{tikzcd}
\end{equation}
\noindent where $f_{h'}$ is the minimal right $\operatorname{add}(P_{h})$-approximation of $P_{h'}$, where $\operatorname{add}(P_h)$ denotes the (small) additive closure of $P_h$. This is a tilting complex, since the category $\Kb{B_1}$ is $0$-Calabi-Yau. In case $\seq(v_1)=1=\seq(v_2)$, the derived equivalence induced by the
complex $T$ is an auto-equivalence and \eqref{EquationApproximation} is isomorphic to the defining triangle of $T_{P_h}(P_{h'})[-1]$, where $T_{P_h}$ denotes the spherical twist associated to the $0$-spherical object $P_h$.

Let us sketch the idea of the proof for the simplest case, when there are no loops or multiple edges at the vertices $v_1$ and $v_2$. The summands of $T$ are of the form:
$$T_{h'}=\begin{cases}
P_{h'}[-1], \text{ if } h' \text{ and } h \text{ do not share a common vertex,}\\
P_h\xrightarrow{\alpha_{hh'}}P_{h'}, \text{ if } h' \text{ and } h \text{ share a common vertex, } h\neq h',\\
P_h, \text{ if } h=h'.
\end{cases}$$
Here $\alpha_{hh'}$ denotes the shortest non-trivial path from $h$ to $h'$ and in the second case $P_h$ is concentrated in degree $0$. In order to define a homomorphism $f$ from $B_2$ to $E=\End_{\mathsf{D}^b(B_1)}(T)$ it is enough to define $f$ on the idempotents corresponding to the vertices of $B_2$ and on arrows. Since the edges of the Brauer graphs corresponding to $B_1$ and $B_2$ are identified we can define $f(e_{h'})=id_{T_{h'}}.$ By abuse of notation we will consider $h$ as a half-edge in order to define $f$ on arrows (then $h$, $\iota(h)$ are the half-edges comprising the edge with endpoints $v_1,v_2$ that we are considering). On the arrows $\alpha_{hh^{+}}$, $\alpha_{gg^{+}}$, $\alpha_{h^{-}h}$ ($g\neq h,h^{-}$), corresponding to sectors between edges incident to the vertices $v_1$ and $v_2$ we define
 \begin{displaymath}
 \begin{tikzcd}
        P_{h} \arrow{d}{\lambda\alpha_{\iota(h)\iota(h)}} \\
        P_{h} \arrow{r}{\alpha_{hh^{+}}} & P_{h^{+}},
 \end{tikzcd}\quad\quad
\begin{tikzcd}
        P_{h}\arrow[swap]{d}{\operatorname{id}} \arrow{r}{\alpha_{hg}}  & P_{g} \arrow{d}{\alpha_{gg^{+}}}   \\
        P_{h} \arrow{r}{\alpha_{hg^{+}}} & P_{g^{+}},
 \end{tikzcd}\quad\quad
\begin{tikzcd}
        P_{h}\arrow[swap]{d}{\operatorname{id}} \arrow{r}{\alpha_{hh^{-}}}  & P_{h^{-}}    \\
        P_{h}   &{ }& { }
 \end{tikzcd}
\end{displaymath}
\noindent where $\lambda^{\seq(s(\iota(h)))}=-1$. On the arrows $\alpha_{kk^{+}}$, $\alpha_{\iota(g)\iota(g)^{+}}$, $\alpha_{\iota(g)^{-}\iota(g)}$ ($k\neq \iota(g),\iota(g)^{-}$), corresponding to the sectors between edges at vertices $v\neq v_1,v_2$, where $\{g,\iota(g)\}$ connects $v$ and $v_1$ or $v_2$, we define: 
 \begin{displaymath}
 \begin{tikzcd}
        P_{k}[-1] \arrow{d}{\alpha_{kk^+}} \\
        P_{k^+}[-1],
 \end{tikzcd}\quad\quad
\begin{tikzcd}
        P_{h} \arrow{r}{\alpha_{hg}} & P_{g} \arrow{d}{\alpha_{\iota(g)\iota(g)^+}} \\
       & P_{\iota(g)^+},
\end{tikzcd}\quad\quad
\begin{tikzcd}
       & P_{\iota(g)^-} \arrow{d}{\alpha_{\iota(g)^-\iota(g)}}  \\
        P_{h} \arrow{r}{\alpha_{hg}}  & P_{g}.  
 \end{tikzcd}
\end{displaymath}
It is straightforward to check that this assignment defines a homomorphism $B_2 \rightarrow E$. The map $f$ is injective, since it maps socle elements of $B_2$ to non-zero elements and it is surjective, since the dimensions of $B_2$ and $E$ coincide. 
The isomorphism $E \cong B_2$ can alternatively be deduced directly from the proof of Proposition \ref{Prop_AddingAnArcMorita}. If $T_h$ denotes the half-twist along $h$ (see Example \ref{ExampleHalfTwist}), then $T_h(\Gamma)$ is obtained from $\Gamma$ by a series of edge removals and inclusions, see Figure \ref{FigureHalfTwistMutation}. From this perspective the presence of $\lambda$ in the definition of $f$ is a result of Lemma \ref{LemmaRootsOfUnityBaseChange} and Corollary \ref{CorollaryWindingEdges} which forces the winding numbers of new edges to be odd.
\end{rem}

\begin{figure}[H]
\centering
\begin{tikzpicture}
\def\ang{5};
\def\radi{0.75};
\def\thick{2pt};
\def\off{20};
\begin{scope}[scale=1, shift={(0,0)}]
\filldraw (0,0) circle (\thick); 
\filldraw (2,0) circle (\thick);
\draw (0,0)--(2,0);
    \foreach \i in {1,...,3} 
    {
        \draw (0,0)-- ({\i*360/5+\off}:\radi);
        \filldraw ({\i*360/5+\off}:\radi) circle (\thick);
    }
    \begin{scope}[shift={(2,0)}, xscale=-1]
        \foreach \i in {1,...,3} 
        {
            \draw (0,0)--({\i*360/5+\off}:\radi);
            \filldraw ({\i*360/5+\off}:\radi) circle (\thick);
        }
    \end{scope}
        \draw[red, ->] ({\ang}:{\radi*0.5}) arc ({\ang}:{360/5+\off-\ang}:{\radi*0.5});
        \begin{scope}[shift={(2,0)}]
           \draw[red, ->] ({180+\ang}:{\radi*0.5}) arc ({180+\ang}:{180+360-3*360/5-\off-\ang}:{\radi*0.5});
           \end{scope}
    \end{scope}
 \draw[->, thick] (3.25,0)--(3.75,0);   
\begin{scope}[scale=1, shift={(5,0)}]
\filldraw (0,0) circle (\thick); 
\filldraw (2,0) circle (\thick);
\draw (0,0)--(2,0);
  \draw[dashed] (0,0)-- ({1*360/5+\off}:\radi);
    \coordinate (A) at (2,0);
     \coordinate (B) at ({1*360/5+\off}:\radi);
   \draw (A) to[out=140, in=25] (B);
        \filldraw ({1*360/5+\off}:\radi) circle (\thick);
    \foreach \i in {2,...,3} 
    {
        \draw (0,0)-- ({\i*360/5+\off}:\radi);
        \filldraw ({\i*360/5+\off}:\radi) circle (\thick);
    }
    \begin{scope}[shift={(2,0)}, xscale=-1]
     \draw[dashed] (0,0)--({3*360/5+\off}:\radi);
     \coordinate (A) at (2,0);
     \coordinate (B) at ({3*360/5+\off}:\radi);
     \draw (A) to[out=-130, in=-15] (B);
  
            \filldraw ({3*360/5+\off}:\radi) circle (\thick);
        \foreach \i in {1,...,2} 
        {
            \draw (0,0)--({\i*360/5+\off}:\radi);
            \filldraw ({\i*360/5+\off}:\radi) circle (\thick);
        }
    \end{scope}
        \draw[red, ->] ({\ang}:{\radi*0.5}) arc ({\ang}:{2*360/5+\off-\ang}:{\radi*0.5});
        \begin{scope}[shift={(2,0)}]
           \draw[red, ->] ({180+\ang}:{\radi*0.5}) arc ({180+\ang}:{180+360-2*360/5-\off-\ang}:{\radi*0.5});
           \end{scope}
    \end{scope}    
 \draw[->, thick] (8.25,0)--(8.75,0);  
\begin{scope}[scale=1, shift={(10,0)}]
\filldraw (0,0) circle (\thick); 
\filldraw (2,0) circle (\thick);
\draw (0,0)--(2,0);
  \draw[dashed] (0,0)-- ({2*360/5+\off}:\radi);
    \coordinate (A) at (2,0);
     \coordinate (B) at ({1*360/5+\off}:\radi);
       \coordinate (C) at ({2*360/5+\off}:\radi);
  \draw (A) to[out=140, in=25] (B);
    \draw (A) to[out=160, in=35] (C);
        \filldraw ({1*360/5+\off}:\radi) circle (\thick);
         \filldraw ({2*360/5+\off}:\radi) circle (\thick);
    \foreach \i in {3} 
    {
        \draw (0,0)-- ({\i*360/5+\off}:\radi);
        \filldraw ({\i*360/5+\off}:\radi) circle (\thick);
    }
    \begin{scope}[shift={(2,0)}, xscale=-1]
     \draw[dashed] (0,0)--({2*360/5+\off}:\radi);
     \coordinate (A) at (2,0);
     \coordinate (B) at ({3*360/5+\off}:\radi);
     \coordinate (C) at ({2*360/5+\off}:\radi);
     
     \draw (A) to[out=-130, in=-15] (B);
     \draw (A) to[out=-150, in=-65] (C);
            \filldraw ({3*360/5+\off}:\radi) circle (\thick);
             \filldraw ({2*360/5+\off}:\radi) circle (\thick);
        \foreach \i in {1} 
        {
            \draw (0,0)--({\i*360/5+\off}:\radi);
            \filldraw ({\i*360/5+\off}:\radi) circle (\thick);
        }
    \end{scope}
        \draw[red, ->] ({\ang}:{\radi*0.5}) arc ({\ang}:{3*360/5+\off-\ang}:{\radi*0.5});
        \begin{scope}[shift={(2,0)}]
           \draw[red, ->] ({180+\ang}:{\radi*0.35}) arc ({180+\ang}:{180+360-1*360/5-\off-\ang}:{\radi*0.35});
           \end{scope}
    \end{scope}    
 \draw[->, thick] (3.25,-2.5)--(3.75,-2.5);  
\begin{scope}[scale=1, shift={(5,-2.5)}]
\filldraw (0,0) circle (\thick); 
\filldraw (2,0) circle (\thick);
\draw (0,0)--(2,0);
    \coordinate (A) at (2,0);
    \coordinate (B) at ({1*360/5+\off}:\radi);
    \coordinate (C) at ({2*360/5+\off}:\radi);
    \coordinate (D) at ({3*360/5+\off}:\radi);
    \draw (A) to[out=140, in=25] (B);
    \draw (A) to[out=160, in=35] (C);
 
    \coordinate (E) at (0.8,0.23);
    \draw[hobby] plot coordinates {(A) (E) ({135}:{0.4*\radi}) (D)};
        \filldraw ({1*360/5+\off}:\radi) circle (\thick);
         \filldraw ({2*360/5+\off}:\radi) circle (\thick);
    \foreach \i in {3} 
    {
        \filldraw ({\i*360/5+\off}:\radi) circle (\thick);
    }
    \begin{scope}[shift={(2,0)}, xscale=-1]
     \coordinate (A) at (2,0);
     \coordinate (B) at ({3*360/5+\off}:\radi);
     \coordinate (C) at ({2*360/5+\off}:\radi);
         \coordinate (D) at ({1*360/5+\off}:\radi);

      \draw[hobby] plot coordinates {(A) (1,-0.21) ({180}:{0.35*\radi})  (D)}; 
     \draw (A) to[out=-130, in=-15] (B);
     \draw (A) to[out=-150, in=-65] (C);
            \filldraw ({3*360/5+\off}:\radi) circle (\thick);
             \filldraw ({2*360/5+\off}:\radi) circle (\thick);
        \foreach \i in {1} 
        {
            \filldraw ({\i*360/5+\off}:\radi) circle (\thick);
        }
    \end{scope}
    \end{scope}    
\end{tikzpicture}
    \caption{A sequence of arc removals and inclusions. The resulting arc system represents the tilting complex $T$ and is obtained from the original by a half-twist.}
    \label{FigureHalfTwistMutation}
\end{figure}
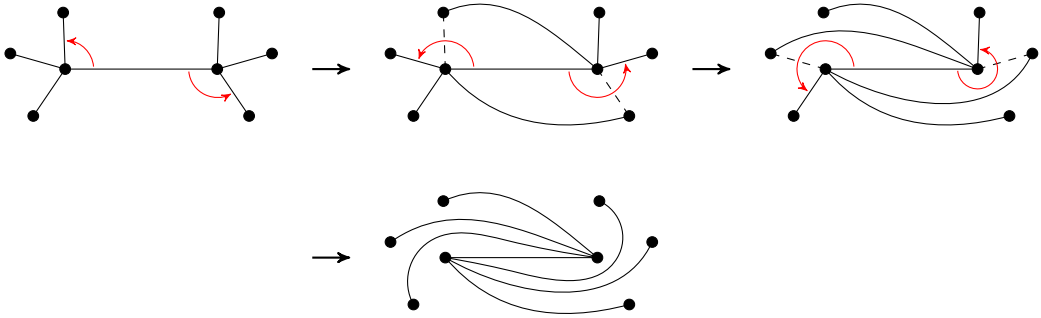

\begin{rem}\label{RemarkAAC} 
Irreducible silting mutations for Brauer graph algebras were studied, for example, in \cite{Antipov, AdachiAiharaChan}. On the level of Brauer graphs irreducible silting mutations are realized as the so-called \textit{Kauer moves}. In \cite[Example 4.6]{Antipov} Antipov constructs two Brauer graphs $\Gamma_1$, $\Gamma_2$ with coinciding derived invariants such that $B_{\Gamma_1}$ is not isomorphic to the endomorphism ring of any tilting complex over  $B_{\Gamma_2}$ which is obtained from $B_{\Gamma_2}$ through a sequence of irreducible silting mutations. Thus, Kauer moves are not sufficient to reach all Brauer graph algebras which are derived equivalent to a given Brauer graph algebra.
\end{rem}

\begin{rem}\label{RemarkMainTheoremRickardTheorem} As mentioned before Brauer graph algebras are trivial extensions of gentle algebras. A classical theorem of Rickard \cite[Theorem 3.1.]{RickardDerivedStable} states that trivial extensions of derived equivalent finite dimensional algebras are again derived equivalent.
One might wonder whether the derived equivalence classification of Brauer graph algebras from Theorem \ref{IntroTheoremCriterionDerivedEquivalence} follows from the derived equivalence classification of gentle algebras. In Example \ref{ExampleBGACuts} we describe two derived equivalent Brauer graph algebras $B_1$ and $B_2$ such that any two gentle algebras $\Lambda_1$ and $\Lambda_2$ with $\triv(\Lambda_1) \cong B_1$ and $\triv(\Lambda_2) \cong B_2$ are not derived equivalent. This implies that Theorem \ref{IntroTheoremCriterionDerivedEquivalence} does not follow from Rickard's Theorem and the derived equivalence classification of gentle algebras in \cite{OpperDerivedEquivalences} and \cite{AmiotPlamondonSchroll}.\end{rem}

\begin{exa}\label{ExampleBGACuts}
Let $\Gamma_1$ and $\Gamma_2$ denote the Brauer graphs with trivial multiplicities as depicted in Figure \ref{FigureExaCuts}. Edges with the same label are meant to be identified. The resulting Brauer graphs have $8$ vertices, $14$ edges and $2$ faces each, producing a ribbon surface $\Sigma_{\Gamma_1} \cong \Sigma_{\Gamma_2}$ of genus $3$. Every component of $\partial \Sigma_{\Gamma_i}$ corresponds to a face of $\Gamma_i$ with perimeter $14$. Moreover, $\Gamma_1$ and $\Gamma_2$ are bipartite and hence Theorem \ref{IntroTheoremCriterionDerivedEquivalence} shows that their associated Brauer graph algebras $B_1$ and $B_2$ are derived equivalent. Let $M_i$ and $N_i$ denote the number of marked points on the pair of boundary components of any cut surface $\Sigma_i$ of $(\Sigma_{\Gamma_i}, \BV{\Gamma_i}, \eta_{\Gamma_i})$, where as usual $\Gamma_i$ is regarded as an arc system inside its ribbon surface. Since $|\!\BV{\Gamma_1}|=2=|\!\BV{\Gamma_2}|$, we have $\{M_1, N_1\}=\{2-x, 6+x\}$ and  $\{M_2, N_2\}=\{3+y, 5-y\}$ for some $x,y \in \{0,1,2\}$ and therefore $\{M_1, N_1\} \neq \{M_2, N_2\}$. It follows that $\Sigma_1$ and $\Sigma_2$ are not diffeomorphic as marked surfaces as such a diffeomorphism is required to induce a bijection between the sets of marked points. Thus, \cite[Theorem B]{OpperDerivedEquivalences} (see also \cite{AmiotPlamondonSchroll}) shows that $\Fuk(\Sigma_1)$ and $\Fuk(\Sigma_2)$ are not equivalent and hence that any two gentle algebras obtained from cuts of $\Gamma_1$ and $\Gamma_2$ are not derived equivalent. Finally, any gentle algebra whose trivial extension is isomorphic to $B_i$ is obtained from $B_i$ by a cut. This is a consequence  of the fact that the Brauer graph of a non-local Brauer graph algebra is uniquely determined by the isomorphism class of the algebra.
\end{exa}

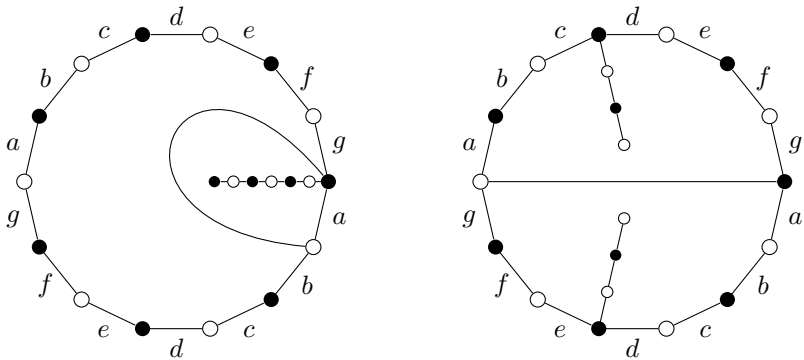
\begin{figure}[H]
    \centering
\begin{tikzpicture}[scale=1]
    \def\r{1};
    \def\rt{1.1};
    \def\thick{2pt};
    \def\angl{360/14};

    \begin{scope}[scale=2]

        \foreach \e [count=\i from 0]  in {a,...,g}
        {
            \node[draw, shape=circle, minimum size=\thick, inner sep=\thick] (P\i) at ({-(2*\i+1)*\angl}:\r){};
            \node[fill, shape=circle, inner sep=\thick] (Q\i) at (-{(2*\i)*\angl}:\r){};
            \draw (P\i)--(Q\i);
            \draw (P\i)--({-(2*\i+2)*\angl}:\r);

            \draw (-{(\i+0.5)*\angl}:\rt) node{$\e$};
            \draw (-{(\i+0.5)*\angl+180}:{\rt}) node{$\e$};
            
        }
        \def\m{0.25};
        
        \foreach \i in {0,2,4,6} 
        {
            \node[fill, shape=circle, inner sep=\thick*0.75] (p\i) at ({\m+(1-\m)/6*\i}, 0){};
        };
         \foreach \i in {1,3,5} 
        {
            \node[draw, shape=circle, inner sep=\thick*0.75] (p\i) at ({\m+(1-\m)/6*\i}, 0){};
        };
        
        \foreach \i in {1,...,6} 
            {
                \draw (p\i)--(p\the\numexpr\i-1\relax);
            };
            
            \draw (P0) to[in=130, out=175, looseness=8] (Q0);
    \end{scope}
    
     \begin{scope}[scale=2, shift={(3,0)}]

        \foreach \e [count=\i from 0]  in {a,...,g}
        {
            \node[draw, shape=circle, minimum size=\thick, inner sep=\thick] (P\i) at ({-(2*\i+1)*\angl}:\r){};
            \node[fill, shape=circle, inner sep=\thick] (Q\i) at (-{(2*\i)*\angl}:\r){};
            \draw (P\i)--(Q\i);
            \draw (P\i)--({-(2*\i+2)*\angl}:\r);

            \draw (-{(\i+0.5)*\angl}:\rt) node{$\e$};
            \draw (-{(\i+0.5)*\angl+180}:{\rt}) node{$\e$};
            
        }      
        \draw (Q0)--(P3);

         \def\m{0.25};
        
        \foreach \i in {1, 3} 
        {
            \node[fill, shape=circle, inner sep=\thick*0.75] (p\i) at (-{(2*2)*\angl}:{\m+(1-\m)/3*\i}){};
             \node[fill, shape=circle, inner sep=\thick*0.75] (q\i) at (-{(2*5)*\angl}:{\m+(1-\m)/3*\i}){};
        };
          \foreach \i in {0, 2} 
        {
            \node[draw, shape=circle, inner sep=\thick*0.75] (p\i) at (-{(2*2)*\angl}:{\m+(1-\m)/3*\i}){};
             \node[draw, shape=circle, inner sep=\thick*0.75] (q\i) at (-{(2*5)*\angl}:{\m+(1-\m)/3*\i}){};
        };
        
        \foreach \i in {1,...,3} 
            {
                \draw (p\i)--(p\the\numexpr\i-1\relax);
                 \draw (q\i)--(q\the\numexpr\i-1\relax);
            };
    \end{scope}

\end{tikzpicture}
    \caption{Bipartite Brauer graphs $\Gamma_1$ and $\Gamma_2$ with identical invariants.}
    \label{FigureExaCuts}
\end{figure}

	\bibliography{bibliography}{}
	\bibliographystyle{alpha}
	
	\ \medskip
	
\end{document}